\documentclass[11pt]{article}
\usepackage{graphicx}
\usepackage{url}
\usepackage{color}
\usepackage[mathscr]{euscript}		
\usepackage{dsfont}
\usepackage{psfrag}			
\usepackage{hyperref}
\usepackage{amsmath}
\usepackage{amsthm}
\usepackage{amssymb}
\usepackage{mathabx}
\usepackage{caption} 
\usepackage{anysize}
\usepackage{wasysym}
\usepackage[shortlabels]{enumitem}
\usepackage{constants}

\hypersetup{linktocpage,
  colorlinks   = true,
  urlcolor     = blue,
  linkcolor    = blue,
  citecolor   = blue
}

\DeclareMathOperator{\var}{var}

\newcommand{\ind}{\mathds}

\newcommand{\Z}{\ensuremath{\mathbb{Z}}}
\newcommand{\N}{\ensuremath{\mathbb{N}}}
\newcommand{\R}{\ensuremath{\mathbb{R}}}
\newcommand{\E}{\ensuremath{\mathbb{E}}}
\renewcommand{\P}{\ensuremath{\mathbb{P}}}

\newcommand{\asto}[1]{\underset{{#1}\to\infty}{\longrightarrow}}

\newtheorem{theorem}{Theorem}[section]
\newtheorem{lemma}[theorem]{Lemma}
\newtheorem{corollary}[theorem]{Corollary}
\newtheorem{proposition}[theorem]{Proposition}
\newtheorem{remark}[theorem]{Remark}

\newtheorem{claim}[theorem]{Claim}

\numberwithin{equation}{section}

\definecolor{Red}{rgb}{1,0,0}
\definecolor{Blue}{rgb}{0,0,1}
\definecolor{Olive}{rgb}{0.41,0.55,0.13}
\definecolor{Yarok}{rgb}{0,0.5,0}
\definecolor{Green}{rgb}{0,1,0}
\definecolor{MGreen}{rgb}{0,0.8,0}
\definecolor{DGreen}{rgb}{0,0.55,0}
\definecolor{Yellow}{rgb}{1,1,0}
\definecolor{Cyan}{rgb}{0,1,1}
\definecolor{Magenta}{rgb}{1,0,1}
\definecolor{Orange}{rgb}{1,.5,0}
\definecolor{Violet}{rgb}{.5,0,.5}
\definecolor{Purple}{rgb}{.75,0,.25}
\definecolor{Brown}{rgb}{.75,.5,.25}
\definecolor{Grey}{rgb}{.7,.7,.7}
\definecolor{Black}{rgb}{0,0,0}

\newcommand{\ignore}[1]{{}}

\renewcommand{\d}{\text{\normalfont\sffamily d}}

\newconstantfamily{small}{
	symbol=c,
}

\date{\today}
\begin{document}
\baselineskip=14pt

\title{An Invariance Principle for a Random Walk Among
Moving Traps
via Thermodynamic Formalism
}
\author{
Siva Athreya%
  \thanks{International Centre for Theoretical Sciences-TIFR, 560083 Bengaluru, India
and Indian Statistical Institute Bangalore Centre, 560059 Bengaluru, India.
    Email: \url{athreya@icts.res.in}}
  \and
  Alexander Drewitz%
  \thanks{Universit\"at zu K\"oln,
Department Mathematik/Informatik,
Weyertal 86--90,
50931 K\"oln, Germany.
    Email: \url{adrewitz@uni-koeln.de}}\, \thanks{NYU Shanghai, NYU-ECNU Institute of Mathematical Sciences,  China.}
  \and
 Rongfeng Sun%
  \thanks{Department of Mathematics, National University of Singapore,
S17, 10 Lower Kent Ridge Road
Singapore, 119076.
    Email: \url{matsr@nus.edu.sg}}
}

\maketitle

\begin{abstract}
    We consider a random walk among a Poisson cloud of moving traps on $\Z^d$, where the walk is killed at a rate proportional to the number of traps occupying the same position. In dimension $d=1$, we have previously shown that under the annealed law of the random walk conditioned on survival up to time $t$, the walk is sub-diffusive. Here we show that in $d\geq 6$ and under diffusive scaling, this annealed law satisfies an invariance principle with a positive diffusion constant if the killing rate is small. Our proof is based on the theory of thermodynamic formalism, where we extend some classic results for Markov shifts with a finite alphabet and a potential of summable variation to the case of an uncountable non-compact alphabet.
\end{abstract}

\section{Introduction} \label{sec:Intro}

The trapping problem, where particles diffuse in space with randomly located traps, has been studied extensively in the statistical physics and mathematics literature.  There are several different models that have been proposed to understand such trapping phenomena, and we refer the reader to the review article \cite{HoWe-94}, which explains in detail the background for the trapping problem and some early results.

When the particle motion is modelled via Brownian motion or random walks in Euclidean space, and immobile traps are Poisson distributed in space, there is substantial literature for the trapping problem: beginning with the seminal works of Donsker and Varadhan \cite{DoVa-75,DoVa-79}. Subsequent developments in the field can be found in  monographs \cite{Sz-98} and \cite{Ko-16}, as well as the sources therein, cf.\ also the survey article \cite{AtDrSu-17}. More recent results about detailed properties of the random walk amongst random obstacles can be found in \cite{DX20, DX19, DFSX20,DFSX21} and \cite{sz23}. In case the traps are mobile, however, much less is known. In
\cite{DrGaRaSu-10} (see also \cite{PSSS13}), the long-time asymptotics of the annealed and quenched survival probabilities were identified in all dimensions, extending earlier work in the physics literature~\cite{MOBC03, MOBC04}. In \cite{AtDrSu-17, Oz-19}  the path behavior of the one-dimensional random walk conditioned on survival up to time $t$ in the annealed setting, was investigated. The model of random walk among mobile traps can serve as a toy model for many physical and biological phenomena, such as foraging predators vs prey~\cite{MBOV13}.

In this article we focus on the specific setting of a single particle diffusing on $\Z^d$ according to a random walk, with randomly located traps that also move as random walks. For simplicity, we assume both the particle and the traps to evolve as simple symmetric random walks. At any given site and point in time, the particle is then killed at a rate proportional to the number of traps at that site.

\subsection{Model and Result}
We consider the model as described in \cite{DrGaRaSu-10}, where for an
intensity parameter $\alpha > 0$, by $(N_y)_{y \in \Z^d}$ one denotes a family of i.i.d.\ Poisson random variables with mean $\alpha$ each. Given
$(N_y)_{y\in\Z^d}$, we then start a family of independent simple
symmetric random walks $(Y^{j,y})_{y \in \Z^d, \; 1 \le j \le N_y}$ on
$\Z^d$, each with jump rate $\rho \ge 0$, with
$Y^{j,y}:=(Y^{j,y}_t)_{t\geq 0}$ representing the path of the $j$-th
trap starting from $y$ at time $0$. These will be referred to as
\lq$Y$-particles\rq\ or \lq traps\rq. For $t \ge 0$ and $x \in \Z^d,$
we denote by
\begin{equation}
\xi(t,x) := \sum_{y \in \Z^d, \; 1 \le j \le N_y} \delta_x (Y^{j,y}_t) \label{xidef}
\end{equation}
the number of traps at site $x$ at time $t.$

We now let $X:=(X_t)_{t\geq 0}$ denote a simple symmetric random walk on
$\Z^d$ with jump rate $\kappa \geq 0$ that evolves independently of
the $Y$-particles.  At each time $t$, the $X$-particle is killed with
rate $\gamma \xi(t,X_t),$ where $\gamma \ge 0$ is the interaction
parameter---i.e., the killing rate is proportional to the number of
traps that the $X$-particle sees at that time instant. We denote the
probability measure underlying the $X$- and $Y$-particles by $\P$, and
if we consider expectations or probabilities with respect to only a
subset of the defined random variables, we provide those as a
superscript; for the sake of clarity we also sometimes specify the starting configuration as
a subscript, such as in $\P_0^X.$

In order to set up notation, we introduce
the quenched survival
probability for a given realization of $\xi,$ i.e.,
\begin{equation} \label{eq:condAnnProb}
Z^{\gamma}_{t,\xi} := \E_0^X \Big[ \exp \Big\{ - \gamma \int_0^t \xi(s,X_s) \, {\rm d}s \Big\}
\Big].
\end{equation}
Alternatively, it will also be
useful to condition on $X$ and integrate out $\xi$ to obtain
\begin{equation} \label{eq:condAnnProb2}
Z^{\gamma}_{t,X} := \E^\xi \Big[ \exp \Big\{ - \gamma \int_0^t \xi(s,X_s) \, {\rm d}s \Big\}
\Big].
\end{equation}
Integrating out $X$ then yields the \lq annealed  survival probability\rq
\begin{equation} \label{eq:Ztg}
Z_{t}^\gamma:= \E^X_0 \Big[  \E^\xi\Big[ \exp \Big \{ -\gamma \int_0^t \xi(s,X_s) \, {\rm d}s \Big\} \Big]\Big] = \E_0^X [Z^\gamma_{t,X}]= \E^\xi [Z^\gamma_{t,\xi}].
\end{equation}
In \cite[Thm. 1.1]{DrGaRaSu-10}, the following asymptotics for the annealed survival probability has been established: For $\gamma\in (0,\infty]$, $\kappa\geq 0$, $\rho>0$ and $\alpha>0$ we have
\begin{align}\label{eq:annealedAsymptotics}
\begin{split}
Z_{t}^{\gamma} = \left\{
 \begin{array}{ll}
\exp\Big\{-\alpha \sqrt{\frac{8\rho t}{\pi}}(1+o(1))\Big\}, &  d=1, \\
&\\
\exp\Big\{-\alpha\pi\rho \frac{t}{\ln t}(1+o(1))\Big\}, & d=2,  \\
&\\
\exp\Big\{-\lambda_{d,\gamma,\kappa,\rho, \alpha}\, t(1+o(1))\Big\}, & d\geq 3,
\end{array}
\right.
\end{split}
\end{align}
where $\lambda_{d,\gamma,\kappa,\rho, \alpha}$ depends on $d$, $\gamma$, $\kappa$, $\rho$, $\alpha$, and is called the annealed Lyapunov exponent. 

A naturally ensuing problem then is to investigate how the asymptotics in
\eqref{eq:annealedAsymptotics} are actually achieved, both in terms of
the behavior of $X$ and $\xi.$ For this purpose, we consider the law of $X$ conditioned on survival up to time $t$ under the annealed law, i.e.\ the
family of Gibbs measures
\begin{equation} \label{eq:Gibbs}
P_{t}^\gamma ( X\in \cdot ) := \frac{\E_0^X \Big[ \E^\xi  \Big[ \exp \Big \{ -\gamma \int_0^t \xi(s,X_s) \, {\rm d}s \Big\} \Big] \ind{1}_{X \in \cdot} \Big] }
{  \E^X_0[Z^\gamma_{t,X}]}, \quad t \ge 0,
\end{equation}
on the Skorohod space $D([0,t],\Z^d)$ of c\`adl\`ag paths taking values in {${\mathbb Z}^d$}, starting from the origin. Note here that the above family of probability measures is in general non-consistent.
In previous work, \cite{AtDrSu-16},  we considered the case
$d=1$ and investigated the typical behavior of $X$ conditioned on
survival.  In  \cite[Theorem 1.2]{AtDrSu-16}, it is shown that there exists $c>0$ such that for all  $\varepsilon>0$,
\begin{equation} \label{eq:subdiff}
P_{t}^\gamma \Big( \Vert X \Vert_t \in \big(c t^{\frac13} ,\, t^{\frac{11}{24}+\varepsilon}\big) \Big) \asto{t} 1,
\end{equation}
with $\Vert f \Vert_t:= \sup_{s \in [0,t]}|f(s)|$ for $t \in (0,\infty)$ and $f \in D([0,t],\Z).$
In particular, this implies that under the Gibbs measure $P_{t}^\gamma$,
 the random walk behaves subdiffusively.  In the case of the random walks replaced by Brownian motions and hard killing, the upper bound on the fluctuations in \eqref{eq:subdiff} has been slightly sharpened in \cite{Oz-19}. More precisely,  if $X$ and the  $Y$-particles move according to independent Brownian motions in $\R,$ and $X$ is killed once it is within a fixed distance $r \in (0,\infty)$ of a $Y$-particle (corresponding to \lq $\gamma = \infty$\rq\ in our notation), then \"Oz  shows in \cite{Oz-19} that there exist constants $c, C  \in (0,\infty)$ such that
 \begin{equation} \label{eq:subdiff2}
P_{t}^\gamma \Big( \Vert X \Vert_t \in \big(c t^{\frac13} ,\, Ct^{\frac{5}{11}}\big) \Big) \asto{t} 1.
\end{equation}

In the current article, we consider dimensions $d\geq 6$ and show that $X$ under $P_t^\gamma$ has diffusive fluctuations and satisfies an invariance principle.  We now state our main result.

\begin{theorem}[Invariance Principle] \label{thm:fCLT}
   Let $d \ge 6$, and let $P^\gamma_t$ be the path measure defined in \eqref{eq:Gibbs},
   where the reference walk $X$ is a simple symmetric random walk on $\Z^d$ with jump rate $\kappa>0$. Furthermore, assume the traps evolve as a Poisson system of independent simple symmetric random walks with density $\alpha>0$ and jump rate $\rho>0$. Then under $P_{t}^\gamma$, the diffusively rescaled paths $(X_{st}/\sqrt{t})_{s\in [0,1]}$ converge in distribution to $(\sigma B_s)_{s\in [0,1]}$ on the Skorohod space $D([0,1], \R^d)$ as $t\to\infty$, where $B$ is a standard Brownian motion on $\R^d$ and $\sigma\in [0,\infty)$ is a deterministic constant. Furthermore, $\sigma>0$ for $\gamma> 0$ sufficiently small.
 \end{theorem}

In view of the sub-diffusive fluctuation in $d=1$, the diffusive fluctuation established in Theorem \ref{thm:fCLT} is perhaps not entirely obvious.  As explained in \cite[Section 3]{AtDrSu-17}, the heuristics for sub-diffusive fluctuation in low dimensions is based on a specific survival strategy that also applies to the case of immobile traps. The strategy is to create a ball of radius $R_t\ll \sqrt{t}$ free of traps at time $0$ and then force the particle $X$ to stay inside the ball and the traps to stay outside the ball up to time $t$. We then optimize over $R_t$ to maximize the survival probability, which leads to $R_t=t^{1/3}$ in $d=1$, $R_t=o(t^{\epsilon})$ for any $\epsilon>0$ in $d=2$, and $R_t=O(1)$ in $d\geq 3$. This seems to suggest that in $d\geq 3$, $X$ would be localized in a ball of constant radius. However, it only takes $O(1)$ amount of time to create a ball with a constant radius that is free of traps. It is then natural to expect that random fluctuations would cause that ball to move diffusively over time. Theorem \ref{thm:fCLT} validates this heuristic. We conjecture that the same result holds in $d\geq 3$, although our techniques are limited to $d\geq 6$. We also conjecture that the diffusion coefficient $\sigma^2$ in Theorem \ref{thm:fCLT} is strictly positive for all $\gamma\geq 0$.

As stated earlier, in contrast to the setting of immobile traps, where much is understood (see \cite{Sz-98, Ko-16} and the references therein), very little is known in the setting of mobile traps other than the asymptotics of the quenched survival probability stated in \eqref{eq:annealedAsymptotics} and the bounds on path fluctuation in $d=1$ stated in \eqref{eq:subdiff} and \eqref{eq:subdiff2}. The invariance principle established in Theorem \ref{thm:fCLT} is another significant step in the literature. One reason for the  slow progress in this area is that the path measure formulated as Gibbs measures in \eqref{eq:Gibbs} is in general non-consistent, which makes understanding the limiting behavior much harder.

Our proof of Theorem \ref{thm:fCLT} uses the theory of thermodynamic formalism, which originated in the study of dynamical systems. Broadly speaking,
thermodynamic formalism is the  study of  Gibbs measures with an interaction potential. The formalism allows the potential to be very general and helps understanding infinite volume limits in diverse areas as   geometric measure theory, Riemannian geometry, number theory, along with dynamical systems and statistical mechanics. Making the conceptual connection to the trapping problem is a key novelty of our approach. As we will detail in Section \ref{sec:previous} below, standard results from the theory of thermodynamic formalism are insufficient for our purposes, largely due to the restriction of state spaces (i.e.\ alphabet space) on which the Gibbs measure lives. However, we were able to extend and improve several of these very deep classical results to our setting. This involved several nuanced technical intermediate steps and some of the derived results may be of independent interest in the field of thermodynamic formalism. These will be explained further in Section \ref{sec:nov}. What is more, we presume that our work will generate further interest in the theory of thermodynamic formalism within the probability community.

\subsection{Novelty and Overview of Proof}\label{sec:nov}

To establish the connection to thermodynamic formalism, we decompose the random walk path $X$ into independent increments over time intervals of length one. The Gibbs weight $Z^\gamma_{t, X}$ in \eqref{eq:condAnnProb2} can then be rewritten in terms of a functional of the path increments $(X_{i+s}-X_i)_{s\in [0,1]}$ for $i\in \N_0$ (see \eqref{tZgttX} below for a precise formulation). In the language of thermodynamic formalism, the alphabet space then includes all c\`adl\`ag paths $D([0,1],\Z^d)$ starting from $0$, which is an uncountable and non-compact set, and hence standard results from the literature do not apply.  Therefore, the main technical contributions of our paper are  extensions of two classical results in the theory of thermodynamic formalism: (i) a Ruelle-Perron-Frobenius Theorem,  and (ii) rates of mixing for the equilibrium measure, to the setting where the potential acts on an uncountable non-compact alphabet space and is assumed to have only summable variation. We describe below these extensions along with a brief overview of the key steps in the proof of Theorem \ref{thm:fCLT}.

In principle, we could also decompose the random walk path $X$ into independent increments of either time duration $1$ or till the next time the random walk makes a jump, whichever occurs first. This would lead to a compact alphabet space, which allows one to apply existing results from thermodynamic formalism. The technical complications arising from random time durations can probably be overcome. However, we followed the above approach due to the intrinsic interest in extending the theory of thermodynamics formalism to an uncountable and non-compact alphabet space.

\begin{itemize}

\item[(i)] {\em Ruelle-Perron-Frobenius Theorem:}  We extend \cite[Theorem 3.3]{Wa-75} for Markov shifts with a finite alphabet $S$ and a potential $\phi$ of summable variation to the setting of an uncountable non-compact alphabet in Theorem \ref{thm:nuex}. We assume that our alphabet space is a complete separable metric space endowed with a probability measure (see \eqref{eq:locComp}) and that the potential is bounded, uniformly continuous, and has  summable variation (see \eqref{eq:varDef}). In Theorem \ref{thm:nuex}, we then prove the existence and uniqueness of a positive eigen-function, resp.\ eigen-probability measure, for the
Ruelle transfer operator, resp.\ its adjoint. Furthermore, we establish the  convergence of the normalized iterates of the transfer operator.  While our proof of Theorem \ref{thm:nuex} follows the strategy of Walters' proof of  \cite[Theorem 3.3]{Wa-75}, the non-compactness of the alphabet space  introduces significant difficulties. For example, the existence of a positive eigen-measure and eigen-function  cannot be shown via a standard application of Schauder–Tychonoff fixed point theorem. Here, we take advantage of an alternative Ces\'aro averaging argument  that is often used to establish the existence of a stationary distribution for a Feller Markov process (see e.g.\ \cite[Chapter I, Proposition 1.8]{Li-05}).

\item[(ii)] {\em  Mixing and Invariance Principle:} As  detailed in Remark \ref{R:Pih}, the positive eigenfunction from Theorem \ref{thm:nuex} \ref{item:harmonic} allows us to $h$-transform the non-conservative Markov transition kernel  associated with the Ruelle transfer operator into the probability transition kernel of a Markov chain.  We extend a result of Pollicott \cite{Po-00} on the temporal decay of correlation for this Markov chain under stationarity (see Theorem \ref{T:Pollicott}).  Pollicott had established the result in the case of  a finite alphabet. Furthermore, we improve Pollicott's bound on correlation decay from an $L^1$ bound to an $L^2$ bound of the same order, which is key in deducing an invariance principle for additive functionals of the stationary Markov process (see Corollary \ref{C:invariance}).  Our proof of Theorem \ref{T:Pollicott} follows closely that of Pollicott \cite{Po-00} with the crucial novelty of applying the Riesz-Thorin interpolation theorem to obtain the above-mentioned improvement from an $L^1$ bound to an $L^2$ bound of the same order.  The invariance principle we formulate in Corollary \ref{C:invariance} is a direct application of a condition in \cite{MePeUt-06}.

\item[(iii)] {\em Proof sketch for Theorem \ref{thm:fCLT}:}  In order to apply the above results from thermodynamic formalism to our original trapping model, we first show in  Proposition \ref{prop:sumVar} that the potential for our model is uniformly continuous and has summable variation in dimensions $d \geq 5$. As a consequence, the assumptions of the Ruelle-Perron-Frobenius theorem that we prove in Theorem \ref{thm:nuex} are satisfied. This allows us to compare the increments of $X$ (under the Gibbs measure $P^\gamma_t$) on the time interval $[\log t, t-\log t]$ to an ergodic process of increments; in fact, this process is the Markov chain defined via the $h$-transform of the Ruelle transfer operator (see also Remark \ref{R:Pih}).  We first prove that ignoring the path $X$ in the time intervals $[0, \log t]$ and $[t-\log t, t]$ and ignoring the fluctuations of the path between consecutive integer times has no effect in the scaling limit (see Corollary \ref{C:approx} and Lemma \ref{L:abscon}). The gap of size $\log t$ at the beginning of the time interval $[0, t]$ allows the ergodic Markov chain to relax and be close to its stationary distribution (see Lemma \ref{L:approx2}), while the $\log t$ gap at the end takes care of forgetting boundary effects of the $h$-transform. In Lemma \ref{L:inv}, we deduce that this stationary process is mixing fast enough thanks to Theorem \ref{T:Pollicott}, which allows
    us to appeal to the invariance principle from Corollary \ref{C:invariance}  and obtain  the convergence of the scaled path to a Brownian motion.
    This is the only place where we require $d\geq 6$ instead of $d\geq 5$ (see \eqref{eq:d6} and the discussions below it).
     We conclude the proof of  Theorem \ref{thm:fCLT} by showing that the diffusion coefficient in front of the Brownian motion is strictly positive whenever the interaction parameter $\gamma >0$ is sufficiently small (see  Proposition \ref{P:posgamma}).

     Extending Theorem \ref{thm:fCLT} to lower dimensions, especially $d=3, 4$, will probably require new approaches because the potential for our model no longer has summable variation.

\end{itemize}

\subsection{Related Results in the Literature } \label{sec:previous}
Functional central limit theorems have been investigated in related models  for Gibbs measures relative to Brownian motion.  In this setting, paths are attributed an energy which most often is a pair interaction of increments of these paths.   The interaction potentials are typically assumed to decay exponentially in time, but are allowed to have a singularity at the origin in space.  Central limit theorems are obtained in the infinite volume limit in a variety of settings; see for instance (we caution this is not an exhaustive list)  \cite{SPOHN1987278}, \cite{BeSp-05}, \cite{BetzEtAl},  \cite{betz2021functional}, \cite{MR4054359}, \cite{Mu-17}, \cite{Gu-06}, and \cite{GuL09}.  Such models can e.g.\ include models from quantum mechanics such as the Nelson model and also Polaron measures, where the energy exhibits a self-attractive  singularity at the origin in space.

 We comment briefly on  the techniques used to prove these results.  In \cite{BeSp-05}, a functional central limit theorem is derived by rewriting the original problem in terms of a Markov process and then ultimately reducing to classical results by Kipnis and Varadhan \cite{KiVa-86}. In \cite{SPOHN1987278} they recast the problem as a model of spins on $\Z$ with single spin space $C([0,1],\R^d)$, and in \cite{Gu-06}, they  consider increments of the Wiener path as basic spin space. To obtain a central limit theorem for the sample paths with respect to the infinite volume limit Gibbs measure, they require estimates on the correlation of functionals of Gibbs measure.  These correlation bounds are  obtained using the earlier work on Gibbs measures in \cite{MR661133} and \cite{Bo-82}. Once the decay of the interaction is fast enough, the non-triviality of the diffusion constant is obtained as well, consequently excluding a sub-diffusive behavior.

 We were not able to adapt the above results nor the techniques to our model for the following reasons. Indeed, while we can also rewrite our Gibbs measure in terms of a functional of the path increments (with an additional exponential), in our setting  the energy is a highly non-linear functional of the entire trajectory (see  \eqref{eq:GibbsRewrite} below), not just a function of of the difference of two increments. Secondly we need to allow  our potential to only exhibit weak polynomial decay, as is manifested in terms of a summable variation condition of the potential, cf.\ \eqref{eq:sumVar} below.

Let us also comment on the relation of our results to the vast literature of thermodynamic formalism. To the best of our knowledge, most results are for Markov shifts with either a finite \cite{Wa-75, Po-00}, or countable \cite{Sa-15}, or compact alphabet \cite{CiSi-16}. Under the additional assumptions that the alphabet is a compact metric space equipped with a Borel measure, and that the potential belongs to the so-called weak Walters class, the analogue of Theorem \ref{thm:nuex} has been proved in \cite[Theorem 4.4]{CiSi-16}. When the alphabet is a standard Borel space and the potential is H\"older continuous, related results were obtained in \cite{CiSiSt-19} (see also \cite{LMSV21} when the alphabet is $\R$). To tackle our original trapping problem, it is natural to consider a non-compact alphabet with a potential that has summable variation which is not H\"older continuous. However,  we have not been able to find results in the literature that covered this case. This motivated us to formulate and prove Theorem \ref{thm:nuex}, which is of independent interest and which we expect to be useful for other problems from statistical physics also.
\medskip

 {\bf Layout of the rest of the paper:}
In Section \ref{sec:thermo}, we state and prove all the required results in thermodynamic formalism. In Section \ref{sec:reform}, we will reformulate our trapping problem in the thermodynamic formalism and show that the potential has the desired regularity properties. In Section \ref{sec:fCLT}, we  prove Theorem \ref{thm:fCLT} using the results from Section \ref{sec:reform}. In each section, we will provide a brief overview of the results and proofs in that section.

\section{Thermodynamic Formalism} \label{sec:thermo}
The theory of thermodynamic formalism is a powerful tool in the study of dynamical systems with roots in equilibrium statistical physics, see e.g.\ the lecture notes \cite{Bo-08} and the survey \cite{Sa-15}. Most results require either a finite \cite{Wa-75, Po-00}, or countable \cite{Sa-15}, or compact alphabet \cite{CiSi-16}. Only \cite{CiSiSt-19}, resp.\ \cite{LMSV21}, considered the case when the alphabet is a standard Borel space, resp.\ $\R$.

Our goal here is to extend two classic results in thermodynamic formalism to the setting of an uncountable non-compact alphabet equipped with a metric $\d$ and a probability measure $\mu$. The first result that we extend is Walters' Ruelle-Perron-Frobenius theorem \cite{Wa-75} for Markov shifts with a finite alphabet $S$ and a potential of summable variation (see Theorem \ref{thm:nuex} below), which will be stated and proved in Sections \ref{S:RPF} and \ref{sec:claim}
(when the potential is H\"older continuous, such a Ruelle-Perron-Frobenius theorem was proved in \cite{CiSiSt-19}). The second result that we extend is a mixing result of Pollicott \cite{Po-00} (see Theorem \ref{T:Pollicott} below). It will be treated in Section \ref{sec:decay}, where we also improve Pollicott's result from an $L^1$ bound to an $L^2$ bound, which is essential in deducing an invariance principle for additive functionals of an ergodic process arising from the thermodynamic formalism (see Corollary \ref{C:invariance}).

\subsection{Markov shifts and Ruelle-Peron-Frobenius Theorem}\label{S:RPF}
We first recall the basic setup of thermodynamic formalism, which can be thought of as the study of Gibbs measures on a space of semi-infinite sequences.
More precisely, we consider an arbitrary
\begin{equation}\label{eq:locComp}
\text{complete separable metric space $(S,\d)$ equipped with a Borel probability measure $\mu$.}
\end{equation}
 Without loss of generality, we may assume ${\rm supp}(\mu)=S$, otherwise we just replace $S$ by the support of $\mu$.
The semi-infinite sequence space is then given by
\begin{equation} \label{eq:SigmaDef}
\Sigma:=S^{\Z_-} \qquad (\Z_-:=-\N),
\end{equation}
where $\N := \{1, 2, \ldots\},$ and
we write ${\cal B}$ for the Borel-$\sigma$-algebra  on $\Sigma$ with the product topology.
As a shorthand we will usually write $\overline{x}=(\ldots x_{-2}x_{-1})$ instead of $(\ldots, x_{-2},x_{-1})$ for elements of $\Sigma,$ and similarly for elements in $S^N,$ $N \in \N.$

In a slight abuse of notation, we also write $\d$ for the metric
\begin{equation} \label{eq:prodMet}
\d(\overline x, \overline w) := \sum_{i=1}^{\infty} \frac{1}{2^{i}} \min\{\d(x_{-i}, w_{-i}),1\}, \qquad \overline{x}, \overline{w} \in \Sigma,
\end{equation}
on $\Sigma.$ We can extend this definition to
$\Sigma \cup \bigcup_{n=1}^\infty S^n$ by first extending $S$ to
$S^*=S\cup\{*\}$ with an isolated cemetery state $* \notin S$, and define $\d(*, z) :=1$ for all $z\in S$ and $\d(*, *):=0$.
For any $(x_i)_{-n\leq i\leq -1} \in S^n$, we extend it to an element $\overline{x}\in (S^*)^{\Z_-}$ by setting  $x_i :=*$ for all $i<-n$. We can then define $\d$ on $\Sigma \cup \bigcup_{n=1}^\infty S^n$ by setting
\begin{equation}\label{eq:prodMet2}
\d\big((x_i)_{-m\leq i\leq -1}, (w_j)_{-n\leq j\leq -1}\big)
:= \d(\overline{x}, \overline{w}), \qquad x_i, w_j \in S,
\end{equation}
where $m, n\in \N\cup\{\infty\}$ and $\d(\overline{x}, \overline{w})$ is defined as in \eqref{eq:prodMet}.

By assumption \eqref{eq:locComp} and Prokhorov's theorem, there exists an increasing sequence of compact sets
$S_1\subset S_2 \subset \ldots$ such that $\mu(S_1)>0$ and
$\mu(S_n)\uparrow 1$. For later purposes, we will also introduce the  increasing sequence of subsets of $\Sigma$ given by
\begin{equation}\label{Sigm}
\Sigma^{(m)}
:=\big\{\overline{x}=(\ldots x_{-2} x_{-1}) \in \Sigma : x_{-n} \in S_{m+n} \mbox{ for all } n\in\N \big\}, \qquad m\in \N.
\end{equation}
Note that
$\Sigma^{(m)}$ is compact for each $m \in \N$ due to Tychonoff's theorem, and  that the assumption ${\rm supp}(\mu)=S$ entails that
$\bigcup_{m \in \N} \Sigma^{(m)}$ is dense in $(\Sigma, \d)$.

We will furthermore use the following notation:
\begin{itemize}
    \item $\mathcal{M}(\Sigma)$/$\mathcal{M}_+(\Sigma)$/$\mathcal{M}_1(\Sigma)$: the space of finite signed/finite positive/probability measures on $(\Sigma, {\cal B})$, equipped with the weak topology;
    \item $C(\Sigma)/C_b(\Sigma)/C_{b,u}(\Sigma)$:  the set of all real-valued continuous/bounded continuous/bounded uniformly continuous functions on $\Sigma$.
\end{itemize}

Given $\phi \in C_b(\Sigma)$, which we will refer to as a potential in the following, we can define an---in general non-conservative---Markov transition kernel $\Pi:\Sigma \to \mathcal{M}_+(\Sigma)$ via
\begin{equation}\label{eq:origTrans}
    \Pi(\overline{x}, A) = \int_S 1_{\{\overline{x}z\in A\}} e^{\phi(\overline{x}z)} \mu({\rm d}z), \qquad \overline{x}\in \Sigma, A\in {\cal B},
\end{equation}
where $\overline{x}z:=(\dots x_{-2} x_{-1} z)\in\Sigma$ is obtained by appending the symbol $z\in S$ to \lq the right of $\overline{x}.$\rq\ In the language of thermodynamic formalism, the Ruelle transfer operator $L_\phi$ is then defined via
\begin{equation} \label{eq:phiDef}
(L_\phi f)(\overline x) := (\Pi f)(\overline x) := \int_\Sigma f(\overline{y}) \, \Pi(\overline{x}, {\rm d}\overline{y})
= \int_S  f(\overline {x}z)  e^{\phi(\overline{x}z)}\,  \mu({\rm d}z),
\end{equation}
for all $f\in C_b(\Sigma)$.
Note that $L_\phi$ maps from $C_b(\Sigma)$ to  $C_b(\Sigma)$ by dominated convergence.
Furthermore,  the map $L_\phi^*: \mathcal{M}(\Sigma) \to \mathcal{M}(\Sigma)$ is defined via
 \begin{equation} \label{eq:adjointDefI}
 \big(L_{\phi}^{*} \nu\big)( {\rm d}\overline{y}) :=(\nu \Pi)({\rm d} \overline{y}) := \int_\Sigma \nu({\rm d} \overline x)\, \Pi(\overline{x}, {\rm d}\overline{y}),  \qquad \, \nu \in \mathcal{M}(\Sigma).
 \end{equation}
 This notation is suggested by the observation that for $ f\in C_b(\Sigma)$ and $\nu \in \mathcal{M}(\Sigma),$ if we denote their dual pairing by
 $$
 \langle \nu, f\rangle := \int_\Sigma f(\overline{x}) \,  \nu({\rm d}\overline{x}),
$$
then $\langle \nu, L_\phi f\rangle = \langle L_\phi^* \nu, f\rangle$.

When there is a clear probabilistic interpretation, we will often use the transition kernel $\Pi$ instead of the operators $L_\phi$ and $L_\phi^*$.
In this vein,  note that in the language of Gibbs measures, $\frac{(L_\phi^*)^n \nu}{\langle (L_\phi^*)^n \nu, 1\rangle} = \frac{\nu\Pi^n}{\langle \nu \Pi^n, 1\rangle}$
defines a Gibbs measure with reference measure $\nu({\rm d}\overline{x})\mu({\rm d}z_1)\ldots \mu({\rm d}z_n)$ for $\overline{x}z_1\ldots z_n\in\Sigma$ and Gibbs weight $e^{\sum_{i=1}^n\phi(\overline{x}z_1\ldots z_i)}$.

For $\phi\in C_b(\Sigma)$, the following result establishes that the Markov transition kernel $\Pi$ satisfies the
so-called Feller property.
\begin{lemma}[Feller Property] \label{cl:adjCont}
If $\phi\in C_b(\Sigma)$, then the mapping $L_\phi^*: \mathcal{M}(\Sigma) \to \mathcal{M}(\Sigma)$ is continuous with respect to the weak topology.
\end{lemma}
\begin{proof}
Let $(\nu_n)_{n \in \N}$ be a sequence of measures in $\mathcal M(\Sigma)$ converging weakly to $\nu \in \mathcal M(\Sigma).$ Then for any $f \in C_b(\Sigma),$  we have
$$
\langle L_\phi^*\nu_n, f\rangle = \langle \nu_n, L_\phi f\rangle
\to \langle \nu, L_\phi f\rangle = \langle L_\phi^* \nu, f\rangle,
$$
since $L_\phi f\in C_b(\Sigma)$ as observed below \eqref{eq:phiDef}. Such a convergence for all $f\in C_b(\Sigma)$ is exactly the definition of the weak convergence $L_\phi^*\nu_n \Rightarrow L_\phi^*\nu$.
\end{proof}

To formulate our results in terms of the thermodynamic formalism, we will impose the following additional assumptions on the potential $\phi \in C_b(\Sigma)$:
\begin{enumerate}[label=\textbf{(A\arabic*)}]
    \item \label{item:sumVar} (Summable Variation)
     For $\phi \in C_b(\Sigma)$ and $n \in \N_0$, we define its {\em $n$-th variation} by
 \begin{equation} \label{eq:varDef}
  {\rm var}_n (\phi) :=
   \sup_{\overline x,\overline y \in \Sigma }  \sup_{z_1, z_2, \ldots, z_n \in S} | \phi(\overline{x} z_1 \ldots z_n) - \phi(\overline{y} z_1 \ldots z_n) |,
 \end{equation}
    and define $\var_n^\infty(\phi):=\sum_{k=n}^\infty \var_k(\phi)$. We assume that
 \begin{equation}  \label{eq:sumVar}
 M:=  \sum_{n =0}^\infty {\rm var}_n (\phi) < \infty,
 \end{equation}
 which is referred to as $\phi$ having {\em summable variation}.

     \item \label{item:cont} (Uniform Continuity)
    For every $\varepsilon >0$  there exists $\delta >0$ such that
    $$
    \sup_{\substack{\overline x, \overline y \in \Sigma:\, \d(\overline{x}, \overline{y}) \leq \delta }} |\phi(\overline{x}) - \phi(\overline{y})|\leq \varepsilon.
    $$
\end{enumerate}

We are now ready to state and prove the Ruelle-Perron-Frobenius type result which we were alluding to before.
\begin{theorem}[Ruelle-Perron-Frobenius] \label{thm:nuex}
 Assume (\ref{eq:locComp}) and that $\phi \in C_b(\Sigma)$  satisfies \ref{item:sumVar}--\ref{item:cont}.  Then:
  \begin{enumerate}[label=\alph*)]
  \item \label{item:fixedPoint} There exist $\lambda >0$  and $\nu \in \mathcal{M}_1(\Sigma)$ such that
  \begin{equation} \label{eq:eigenmeasure}
  L^{*}_\phi \nu = \lambda \nu.
  \end{equation}

  \item \label{item:harmonic} For all $\lambda$ and $\nu$ as in \ref{item:fixedPoint}, there exists $ h \in C_{b,u}(\Sigma)$ (more precisely, $h\in \Lambda$ as defined in \eqref{eq:Lambda} below)  with $\inf h>0$, $\langle \nu, h\rangle =1, $ and $L_\phi h = \lambda h.$

  \item \label{item:sgConv} For  $\lambda$ and $\nu$ as in \ref{item:fixedPoint}, $h$ as in \ref{item:harmonic}, and for all $f \in C_{b}(\Sigma)$ and $\overline x\in \Sigma$, we have $$\lim\limits_{n \rightarrow \infty} \frac{1}{\lambda^n} (L^n_\phi f)(\overline x) = \langle \nu, f\rangle\, h(\overline{x}),$$
  and the convergence is uniform on each of the compact sets $(\Sigma^{(m)})_{m\in\N}$ defined in \eqref{Sigm}.
\end{enumerate}
Furthermore, $\nu$ and $h$ are unique.
\end{theorem}

\begin{remark}\label{R:Pih}
The positive eigenfunction $h$ with eigenvalue $\lambda$ from Theorem \ref{thm:nuex} \ref{item:harmonic} allows us to $h$-transform the non-conservative Markov transition kernel $\Pi$ from \eqref{eq:origTrans} into the probability transition kernel $\Pi_h$ of a Markov chain with state space $\Sigma$:
\begin{equation}\label{eq:hTrans}
    \Pi_h(\overline{x}, A) := \frac{1}{\lambda h(\overline{x})} \int_A h(\overline{y}) \, \Pi(\overline{x}, {\rm d}\overline{y})  =  \frac{1}{\lambda h(\overline{x})} \int_S 1_{\{\overline{x}z\in A\}} h(\overline{x}z) e^{\phi(\overline{x}z)} \, \mu({\rm d}z), \quad A\in {\cal B}.
\end{equation}
We then observe that $\nu_h({\rm d}\overline{x}) := h(\overline{x}) \nu({\rm d} \overline{x})$, known as the Ruelle-Perron-Frobenius measure, is invariant with respect to $\Pi_h$, i.e., $\nu_h \Pi_h= \nu_h$. As a consequence, writing
\begin{equation*}
\Pi_h f(\overline{x}):=
 \frac{1}{\lambda h(\overline{x})} \int_S f(\overline{y})h(\overline{y}) \, \Pi(\overline{x}, {\rm d}\overline{y}),
\end{equation*}
the convergence in Theorem \ref{thm:nuex} \ref{item:sgConv}  can be rewritten as
\begin{align} \label{erg}
\begin{split}
    \lim_{n\to\infty} \Big(\Pi_h^n \Big(\frac{f}{h}\Big)\Big)(\overline{x}) = \langle \nu, f\rangle = \big \langle \nu_h, \frac{f}{h}\big\rangle.
\end{split}
\end{align}
Choosing $f=h \psi$ for $\psi\in C_b(\Sigma)$, the convergence in \eqref{erg} is seen to be equivalent to the weak convergence of $\Pi_h^n(\overline{x}, \cdot)$ to $\nu_h$ for any $\overline{x} \in \Sigma$. In other words, the Markov chain with state space $\Sigma$ and transition kernel $\Pi_h$ is ergodic with invariant measure $\nu_h$.
\end{remark}

We will now prove Theorem \ref{thm:nuex}. Our overall proof strategy follows that of Walter in \cite{Wa-75}. The non-compactness of the alphabet space $S$, however, introduces significant difficulties that need to be overcome with novel arguments. Throughout the proof, we rely crucially on the assumptions that $\phi$ has summable variation and is bounded as well as uniformly continuous.  In order not to hinder the flow of reading, we will prove Theorem \ref{thm:nuex} assuming various claims, and then provide the proofs of these auxiliary results in the subsequent Section \ref{sec:claim}.

\begin{proof}[Proof of Theorem~\ref{thm:nuex}~\ref{item:fixedPoint}]
If the alphabet $S$ was compact, we could find an eigen-probability measure $\nu$ for $L_\phi^*$ by applying the Schauder–Tychonoff fixed point theorem to the normalized map
\begin{equation}\label{tildeL*}
    \widetilde L^*_\phi \varrho := \frac{L^*_\phi \varrho}{\langle L^*_\phi \varrho, 1 \rangle } = \frac{L_\phi^*\varrho}{(L^*_\phi \varrho)(\Sigma)}, \qquad \varrho \in {\cal M}_1(\Sigma),
\end{equation}
which can be shown to define a continuous map from ${\cal M}_1(\Sigma)$ to ${\cal M}_1(\Sigma)$ (see Claim \ref{L*continuity} below). To handle our non-compact setting, we will first identify fixed points of maps $\widetilde L^*_{\phi, m}$, $m\in\N$, which approximate $\widetilde L^*_\phi$ but map from ${\cal M}_1(\Sigma^{(m)})$ back to itself, where $(\Sigma^{(m)})_{m\in\N}$ is the increasing sequence of compact subsets of $\Sigma$ defined in \eqref{Sigm}.  We then conclude by showing that the fixed points of $\widetilde L^*_{\phi, m}$, $m\in\N$, form a tight subset of ${\cal M}_1(\Sigma)$, and any limit point of this tight subset is actually a fixed point of $\widetilde L_\phi^*$.

For each $m\in\N$, we define
\begin{equation}\label{blSm}
(L^*_{\phi, m} \varrho)({\rm d} \overline{y}) :=
1_{\{\overline{y} \in \Sigma^{(m)}\}} (L_\phi^* \varrho) ({\rm d} \overline{y}), \qquad \varrho \in {\cal M}(\Sigma^{(m)}).
\end{equation}
We further define the normalized map
\begin{equation}\label{eq:adjointDef}
    \widetilde L^*_{\phi, m} \varrho := \frac{L^*_{\phi, m} \varrho}{\langle L^*_{\phi, m} \varrho, 1 \rangle } = \frac{L^*_{\phi, m} \varrho}{(L^*_\phi \varrho)(\Sigma^{(m)})} , \qquad \varrho \in {\cal M}_1(\Sigma^{(m)}),
\end{equation}
so that $\widetilde L^*_{\phi, m} \varrho \in {\cal M}_1(\Sigma^{(m)})$.

\begin{claim} \label{L*continuity}
The maps $\widetilde L^*_\phi : {\cal M}_1(\Sigma) \to {\cal M}_1(\Sigma)$ from \eqref{tildeL*} and $\widetilde L^*_{\phi, m} : {\cal M}_1(\Sigma^{(m)}) \to {\cal M}_1(\Sigma^{(m)})$, $m\in\N$, from \eqref{eq:adjointDef}  are all continuous.
\end{claim}
We defer the proof of this and all subsequent claims to Section \ref{sec:claim}. Note that because $\Sigma^{(m)}$ is compact, ${\cal M}_1(\Sigma^{(m)})$ is a non-empty compact convex subset of the locally convex Hausdorff topological vector space ${\cal M}(\Sigma^{(m)})$ equipped with the weak topology. Therefore, we can apply the Schauder-Tychonoff fixed point  theorem (\cite[Theorem V.10.5]{DuSc-58}) to the continuous map
$\widetilde L^*_{\phi, m} : {\cal M}_1(\Sigma^{(m)}) \to {\cal M}_1(\Sigma^{(m)})$ in order to deduce that for each $m \in \N,$ there exists $\nu^{(m)} \in \mathcal{M}_1(\Sigma^{(m)})\subset \mathcal{M}_1(\Sigma)$ such that
\begin{equation} \label{eq:invMeas}
\widetilde L_{\phi,m}^{*}\nu^{(m)} = \nu^{(m)}.
\end{equation}
This family of fixed points is well-behaved in the following sense.
\begin{claim} \label{cl:tight}
The family of probability measures $(\nu^{(m)})_{m \in \N}$  is tight in $\mathcal{M}_1(\Sigma)$.
\end{claim}
Combining Claim \ref{cl:tight} and Prokhorov's theorem, it follows that there exists a subsequence $(\nu^{(m_k)})_{k \in \N}$ such that $\nu^{(m_k)}\Rightarrow \nu$ for some $\nu \in \mathcal{M}_1(\Sigma).$ In order to show $\widetilde L^*_\phi \nu =\nu$, we take limits on both sides of \eqref{eq:invMeas} along the sequence $(m_k)_{k\in\N}$, and it remains to show that $\widetilde L_{\phi,m_k}^{*}\nu^{(m_k)} \Rightarrow \widetilde L^*_\phi \nu$. Since $\widetilde L_\phi^{*}\nu^{(m_k)} \Rightarrow \widetilde L^*_\phi \nu$ by the continuity of $\widetilde L_\phi^*$, it suffices to
compare $\widetilde L_{\phi,m_k}^{*}\nu^{(m_k)}$ with $\widetilde L_\phi^{*}\nu^{(m_k)}$ by showing that for any $f \in C_b(\Sigma)$, we have
\begin{equation} \label{eq:approxInv}
\big | \langle \widetilde L_{\phi,m_k}^{*}\nu^{(m_k)}, f\rangle - \langle \widetilde L_\phi^{*}\nu^{(m_k)} , f\rangle  \big | \to 0 \quad \text{as } k \to \infty.
\end{equation}
By the definitions of $\widetilde L_{\phi,m_k}^{*}\nu^{(m_k)}$
and $\widetilde L_\phi^{*}\nu^{(m_k)}$ in \eqref{eq:adjointDef} and \eqref{tildeL*}, the left-hand side of \eqref{eq:approxInv} equals
\begin{align} \label{eq:plugInDefsNormalAdj}
\begin{split}
\Big | \frac{\int_{\Sigma^{(m_k)}} f(\overline{y}) \, (L_\phi^* \nu^{(m_k)})({\rm d} \overline{y})}{(L^*_\phi \nu^{(m_k)})(\Sigma^{(m_k)})}  - \frac{\int f(\overline{y}) \, (L^*_\phi \nu^{(m_k)})({\rm d}\overline y)}{(L^*_\phi \nu^{(m_k)})(\Sigma)} \Big |.
\end{split}
\end{align}
Recalling the definition of $L_\phi^*$ from \eqref{eq:adjointDefI}, the fact that $\mu$ is a probability measure, and using that by assumption we have $\phi \in C_b(\Sigma)$, it is then easily seen that we have the uniform bounds
\begin{equation} \label{eq:massLUB}
e^{-\Vert \phi\Vert_\infty} \leq (L^*_\phi \nu^{(m_k)})(\Sigma) \leq e^{\Vert \phi\Vert_\infty}, \quad k \in \N.
\end{equation}
Furthermore, the difference of the numerators and the difference of the denominators in \eqref{eq:plugInDefsNormalAdj} are both upper bounded in absolute value by \begin{align*}
\max\{\Vert f\Vert_\infty, 1\}\cdot  (L_\phi^*\nu^{(m_k)})(\Sigma\backslash \Sigma^{(m_k)}) & \leq C \int_{\Sigma^{(m_k)}} \nu^{(m_k)}({\rm d}\overline{x}) \int_{S^{\mathsf c}_{m_k}} e^{\phi(\overline{x}z)} \mu( {\rm d}z) \\
& \leq C e^{\Vert \phi\Vert_\infty} \mu(S^{\mathsf c}_{m_k}) \to 0 \quad \mbox{ as }\ k\to\infty.
\end{align*}
Therefore, in combination with \eqref{eq:massLUB}, the difference in \eqref{eq:plugInDefsNormalAdj} converges to $0$ as $k\to\infty$. This concludes the proof that $\nu$ is a fixed point of $\widetilde L_\phi^*$, and hence $L_\phi^*\nu = \lambda \nu$ with $\lambda := (L_\phi^*\nu)(\Sigma) \in [e^{-\Vert \phi\Vert_\infty}, e^{\Vert \phi\Vert_\infty}].$
\end{proof}
\medskip

\begin{proof}[Proof of Theorem~\ref{thm:nuex}~\ref{item:harmonic}]
Let $\nu$ and $\lambda$ be as in Theorem~\ref{thm:nuex}~\ref{item:fixedPoint}, and for $\overline x ,\overline y \in \Sigma$ define
\begin{equation}\label{Bxy}
B(\overline x,\overline y) := \exp \Big\{\sum_{k=1}^{\infty}  \sup_{z_1, z_2, \ldots z_k \in S} \big | \phi(\overline{x} z_1 z_2 \ldots z_k) - \phi(\overline{y} z_1 z_2 \ldots z_k) \big | \Big\}.
\end{equation}
We observe that due to \eqref{eq:sumVar}  in Assumption \ref{item:sumVar}, we have
\begin{equation} \label{eq:supBfin}
    M_1:= \sup_{\overline x,\overline y \in \Sigma} B(\overline x,\overline y) \leq e^{\sum_{n=0}^\infty \var_n(\phi)}    <\infty.
\end{equation}
Still for $\nu$ as in Theorem~\ref{thm:nuex}~\ref{item:fixedPoint}, we furthermore define
\begin{equation} \label{eq:Lambda}
    \Lambda :=   \Big \{ f \in C_b(\Sigma) \,  | \, f >0, \langle \nu, f\rangle =1 , f( \overline x) \leq B( \overline x, \overline y) f( \overline y)\, \mbox{ for all }  \overline x,  \overline y \in \Sigma \Big\},
\end{equation}
and obtain the following auxiliary result.
\begin{claim} \label{cl:cbd} The set $\Lambda$ satisfies the following properties:
\begin{enumerate}[label=(\roman*)]
    \item \label{item1} Uniformly in $f\in \Lambda$, we have $\frac{1}{M_1}\leq \inf f\leq \sup f\leq M_1$;
    \item \label{item2} $\Lambda$ is uniformly equicontinuous, i.e., uniformly in $f\in \Lambda$, $|f(\overline{x})-f(\overline{y})| \leq \varpi(d(\overline{x}, \overline{y}))$ for some uniform modulus of continuity $\varpi$ with $\varpi(0) = \lim_{\delta\downarrow 0} \varpi(\delta)=0$;
    \item \label{item3} $\Lambda$ is convex and closed under pointwise convergence;
    \item \label{item4} $L_{\widetilde{\phi}} (\Lambda) \subset \Lambda$, where $\widetilde \phi:= \phi -\log \lambda$ so that $L_{\widetilde{\phi}}^*\nu =\nu$.
\end{enumerate}
\end{claim}

We now show that there exists $h \in \Lambda$ such that $L_{\widetilde{\phi}} h = h$, which we note is equivalent to $L_\phi h =\lambda h$. Because $\Sigma$ is not assumed to be compact and $\Lambda$ is not a compact subset of $C(\Sigma),$ as before we cannot apply standard fixed point theorems. Instead, we will take advantage of a Ces\`aro averaging argument that is often used to establish the existence of a stationary distribution for a Feller Markov process (see e.g.\ \cite[Chapter I, Proposition 1.8]{Li-05}). To this end, define for each $n \in \N$ the averaging operator
\begin{align} \label{eq:cesaro}
A_n f := \frac{1}{n} \sum_{m=1}^n L_{\widetilde{\phi}}^m f, \qquad f\in \Lambda.
\end{align}
We now choose and fix some $f\in \Lambda$, and note that by Claim \ref{cl:cbd} we have that $A_n f \in \Lambda$ and that $\Lambda$ is a bounded uniformly equicontinuous family.
We can therefore, apply the Arzel\`a-Ascoli theorem to find a subsequence $(A_{n_k}f)_{k\in\N}$ which converges to a function $h: \Sigma \to \R$ uniformly on each compact set $\Sigma^{(m)}\subset \Sigma$, $m\in\N$, defined in \eqref{Sigm}, and below which it was also noted that $\cup_m \Sigma^{(m)}$ is dense in $(\Sigma, d)$. Together with the uniform equicontinuity of $(A_{n_k}f)_{k\in\N}$, it is then easily seen that the uniform convergence of $(A_{n_k}f)_{k\in\N}$ to $h$ on each $\Sigma^{(m)}$ implies that $(A_{n_k}f)_{k\in\N}$ converges pointwise on $\Sigma$. Hence, by Claim \ref{cl:cbd} \ref{item3}, the limit $h$ also belongs to $\Lambda$, which in particular implies $\inf h>0$ by Claim \ref{cl:cbd} \ref{item1}. 

It only remains to show that $L_{\widetilde \phi}h= h$.
For this purpose, observe that for any $k \geq 1$,
 \begin{align}
L_{\widetilde{\phi}} (A_{n_k} f) =  A_{n_k} f +  \frac{1}{n_k}\big( L^{n_k+1}_{\widetilde{\phi}}f -L_{\widetilde{\phi}}f\big).
 \end{align}
 Now due to the pointwise convergence of $(A_{n_k}f)_{k\in\N}$ and Claim \ref{cl:cbd},
for each $\overline{x}\in \Sigma$ the right-hand side converges to $h(\overline{x})$ as $k\to\infty$. For the left-hand side, we compare with $(L_{\widetilde \phi}h)(\overline{x})$ and note that
\begin{align}\label{Ldiff}
& \big(L_{\widetilde{\phi}} (A_{n_k} f) \big)(\overline{x}) - (L_{\widetilde{\phi}} h)(\overline{x})
=  \int_S \frac{1}{\lambda} e^{\phi(\overline{x}z)}
\big((A_{n_k}f)(\overline{x}z)-h(\overline{x}z)\big) \, \mu({\rm d}z).
\end{align}
If $\overline{x} \in \Sigma^{(m)}$ for some $m\in\N$, then the expression in \eqref{Ldiff} is seen to converge to $0$ by splitting the integral over $z\in S$ into $z\in S_K$ and $z\in S_K^{\mathsf c}$ for some large $K \ge m$: due to the boundedness of the integrand, the contribution to the integral from integration over $z\in S_K^{\mathsf c}$ can be made arbitrarily small (uniformly in $n_k$) by choosing $K$ large, while the contribution from $z\in S_K$ tends to $0$ due to the fact that $\overline{x}z \in \Sigma^{(K)}$ and the uniform convergence of $A_{n_k}f\to h$ on $\Sigma^{(K)}$. For $\overline{x}\in \Sigma \backslash \cup_m \Sigma^{(m)}$, we can approximate $\overline{x}$ arbitrarily well by $\overline{x}^m \in \Sigma^{(m)}$ with $m$ large. In this case, due to the previous reasoning, the difference on the left-hand side of \eqref{Ldiff} is then seen to converge to $0$ if we replace $\overline{x}$ by $\overline{x}^m$. The error created by this replacement can be made arbitrarily small (uniformly in $n_k$) by choosing $m$ large and using the uniform equicontinuity of $\Lambda$ and the observation that $L_{\widetilde{\phi}} (A_{n_k} f)$ and $L_{\widetilde{\phi}} h$ are both elements of  $\Lambda$. This concludes the proof of $L_{\widetilde \phi}h=h$. This concludes the proof of Theorem~\ref{thm:nuex}~\ref{item:harmonic}.
\end{proof}
\bigskip

\begin{proof}[Proof of Theorem~\ref{thm:nuex}~\ref{item:sgConv}]
Let $\lambda$, $\nu$ and $h$ be as in parts \ref{item:fixedPoint}
and \ref{item:harmonic}. Following Remark \ref{R:Pih}, we can write
\begin{equation}\label{Pihconv}
\frac{(L^n_\phi f)(\overline x)}{\lambda^n h(\overline{x})}  = \Big(\Pi_h^n \Big(\frac{f}{h}\Big)\Big)(\overline{x}),
\end{equation}
where $\Pi_h: C_b(\Sigma) \to C_b(\Sigma)$ is the Markov operator
defined from the transition  probability kernel $\Pi_h(\overline{x}, {\rm d}\overline{y})$ in \eqref{eq:hTrans}, that is,
\begin{equation}\label{Pihpsi}
    (\Pi_h \psi)(\overline{x}) =
     \frac{1}{\lambda h(\overline{x})} \int_S \psi(\overline{x}z) h(\overline{x}z) e^{\phi(\overline{x}z)}  \mu({\rm d}z).
\end{equation}

As observed below \eqref{erg}, the convergence of the right-hand side of \eqref{Pihconv} to $\langle \nu, f\rangle=\langle \nu_h, \frac{f}{h}\rangle$ for all $f\in C_b(\Sigma)$ is then equivalent to the weak convergence of $\Pi_h^n(\overline{x}, \cdot)$ to $\nu_h$ for each $\overline x \in \Sigma$. Therefore, by the Portmanteau theorem, it suffices to consider bounded and uniformly continuous $\psi:=f/h\in C_{b,u}(\Sigma)$, which is equivalent to considering $f\in C_{b,u}(\Sigma)$ since we know from Theorem~\ref{thm:nuex}~\ref{item:harmonic} that $\inf h>0$ and $h\in C_{b,u}(\Sigma)$.

\begin{claim} \label{cl:uconlg} Let $\psi \in C_{b,u}(\Sigma).$
Then the family of functions $(\Pi_h^m \psi)_{m\in\N}$ is uniformly equicontinuous, i.e., for each $\varepsilon >0$, there exists $\delta >0$ such that for all $m\in\N$ and $\overline x, \overline y \in \Sigma$ with  $\d(\overline x, \overline y) \leq \delta$,
\begin{equation} \label{eq:i1b}
\big|(\Pi_h^m \psi)(\overline{x}) - (\Pi_h^m \psi)(\overline{y}) \big| \leq
\varepsilon.
\end{equation}
\end{claim}
Assuming this claim, then employing the same argument as below \eqref{eq:cesaro} in the proof of Theorem~\ref{thm:nuex}~\ref{item:harmonic} (with $A_nf$ replaced by $\Pi_h^n\psi$ here and noting that the boundedness of the family $(\Pi_h^n\psi)_{n \in \N}$ follows from the fact that $\Pi_h$ is a Markov operator), there exists a subsequence $(\Pi_h^{m_i}\psi)_{i\in\N}$ which converges pointwise on $\Sigma$ to some
$\psi_* \in C_{b,u}(\Sigma)$ with the same modulus of continuity as $(\Pi_h^m \psi)_{m\in\N}$, and the convergence is uniform on each of the compact sets $\Sigma^{(m)}$, $m\in\N$, defined in \eqref{Sigm}.

We will now prove that $\psi_*$ is constant and equals $\langle \nu_h, \psi\rangle = \langle \nu, h\psi\rangle$. Subsequently, we
will then show that the full sequence $(\Pi_h^m \psi)_{m\in\N}$ also converges to $\psi_*$, which together with \eqref{Pihconv}  implies Theorem~\ref{thm:nuex}~\ref{item:sgConv}.

We start by observing that for all $m\geq 1$ and
$\psi \in C_{b,u}(\Sigma),$
\begin{equation} \label{eq:infmon}
 \inf_{\overline x \in \Sigma}\psi(\overline x) \leq
 \inf_{\overline x \in \Sigma}(\Pi_h\psi)(\overline x)
 \cdots \leq \inf_{\overline x \in \Sigma}(\Pi^m_{h}\psi)  (\overline x) \leq \cdots  \leq  \inf_{\overline x \in \Sigma} \psi_*(\overline x),
\end{equation}
where all but the last inequality follow from the fact that $\Pi_h$ is a Markov operator (i.e., an integral operator with a probability kernel), while the last inequality follows from the pointwise convergence of $\Pi_h^{m_i}\psi$ to $\psi_*$ along the sequence $(m_i)_{i\in\N}$.

In order to show that $\psi_*$ is constant, the following claim is key.
\begin{claim}\label{cl:fstarinf}
 For all $m \geq 1,$
\begin{equation} \label{eq:fstarinf}
    \inf_{\overline x \in \Sigma}(\Pi^{m}_{h}\psi_*)(\overline x) = \inf_{\overline x \in \Sigma} \psi_*(\overline x).
\end{equation}
\end{claim}
We proceed as follows. To prove that $\psi_*$ is constant, i.e., $\psi_*\equiv \inf \psi_*$, we will show that
\begin{align} \label{eq:psi*const}
    \text{for all } \overline y \in \Sigma \text{
    and } \varepsilon >0,
    \text{ we have }\psi_*(\overline y) \le \inf \psi_* + 2\varepsilon.
\end{align}
For this purpose,  fix $\overline y = (y_i)_{i\leq -1}\in \Sigma$ and
$\varepsilon >0$. Since $\psi_* \in C_{b,u}(\Sigma),$ we can choose $\delta > 0$ such that  \begin{equation} \label{eq:psi*cont}
\text{for all $\overline x, \overline z \in \Sigma$ with $\d(\overline x, \overline z)< \delta$, we have }
    |\psi_*(\overline x) - \psi_*(\overline z)| < \varepsilon.
\end{equation}

Let $B(\overline{y}, \delta)$ denote the ball of radius $\delta$ centered at $\overline{y}\in \Sigma$. We first show that for $m\in\N$ sufficiently large, $\inf_{\overline{x}\in \Sigma}\Pi^m_h(\overline{x}, B(\overline{y}, \delta))>0$, which will be needed later on. The definition of $\Pi_h(\overline{x}, {\rm d}\overline{y})$ from \eqref{eq:hTrans} implies that
\begin{align}\label{pimh}
\begin{split}
    &\Pi^m_h(\overline{x}, B(\overline{y}, \delta))
    := \frac{1}{\lambda^m h(\overline{x})} \\
    &\quad \times\idotsint_{S^m} 1_{\{\overline{x}z_1\ldots z_m\in B(\overline{y}, \delta)\}} h(\overline{x}z_1\ldots z_m) \prod_{i=1}^m e^{\phi(\overline{x}z_1\ldots z_i)} \mu({\rm d}z_i).
\end{split}
\end{align}
Note that if $m$ is sufficiently large, then
for any $\overline x \in \Sigma$, we have
$\d((\overline x y_{-m} \ldots y_{-1}), \overline y) < \delta/2.$ To ensure $\overline{x}z_1\ldots z_m\in B(\overline{y}, \delta)$, we only need to require
$(z_1, \ldots, z_m)$ to be in a small neighbourhood of $(y_{-m}, \ldots, y_{-1})$. Using the fact that $h$ is bounded away from $0$ and $\infty$, $\inf \phi>-\infty$, and the assumption that $\mu$ has full support on $S$ (see \eqref{eq:locComp} and below), we infer using \eqref{pimh} that
\begin{equation*}
 \Cl[small]{const:lbPim}= \Cr{const:lbPim}(\overline{y}, m, \delta):= \inf_{\overline x \in \Sigma} \Pi_h^m(\overline x, B(\overline y, \delta)) > 0.
\end{equation*}
Claim \ref{cl:fstarinf} implies that there exists $\overline{x}_{\varepsilon} \in \Sigma$ such that
\begin{equation}\label{psi*barx}
 \inf_{\overline x \in \Sigma} \psi_*(\overline x)
 = \inf_{\overline x \in \Sigma}(\Pi^{m}_{h}\psi_*)(\overline x)
 \leq     (\Pi^{m}_{h}\psi_*)(\overline{x}_{ \varepsilon}) < \inf_{\overline x \in \Sigma} \psi_*(\overline{x}) + \Cr{const:lbPim} \varepsilon,
\end{equation}
and where we also used Claim \ref{cl:fstarinf} to obtain the equality.
Since $\Pi_h^m$ is a Markov operator with $(\Pi_h^m \psi_*)(\overline{x}_{\varepsilon}) = \int
\psi_*(\overline{z}) \, \Pi_h^m(\overline{x}_{\varepsilon}, {\rm d}\overline{z})$ and by definition $\Cr{const:lbPim}\leq 1$,  for the last inequality in \eqref{psi*barx} to hold, there must exist some $\overline{z} \in B(\overline{y}, \delta)$ with $\psi_*(\overline{z}) \leq \inf \psi_* + \varepsilon$.
Combined with the choice of $\delta$ in \eqref{eq:psi*cont}, it follows that
$$
\psi_*(\overline{y}) \leq \psi_*(\overline{z}) +\varepsilon \leq \inf \psi_* + 2 \varepsilon,
$$
which establishes \eqref{eq:psi*const} and hence the fact that $\psi_*$ is constant.

As noted in Remark \ref{R:Pih}, $\nu_h({\rm d}\overline{x})$ is an invariant measure for $\Pi_h$. Therefore, we have that
$$
\lim_{i\to\infty} \langle \nu_h, \Pi_h^{m_i}\psi\rangle = \langle \nu_h, \psi\rangle.
$$
On the other hand, the bounded convergence theorem implies that the limit on the left-hand side equals $\psi_*$, so we conclude that
\begin{equation}\label{eq:psi*def}
\psi_*=\langle \nu_h, \psi\rangle.
\end{equation}

Lastly, the convergence of $\Pi_h^m(\overline{x}) \psi \to \psi_*$ along the full sequence $m\in\N$ is formulated as Claim \ref{cl:lncfs} below and will be proved in the next subsection.
\end{proof}

\begin{claim} \label{cl:lncfs}
For each $\psi \in C_{b,u}(\Sigma)$, we have $\Pi_h^m \psi\to \psi_*$ with $\psi_*$ as in \eqref{eq:psi*def}.
The convergence holds pointwise, and also uniformly on each of the compact sets $\Sigma^{(m)}$, $m\in\N$, defined before \eqref{blSm}.
\end{claim}

\begin{proof}[Proof of Theorem~\ref{thm:nuex}]
It remains to show that $h$ and $\nu$ are unique, which follow from Theorem~\ref{thm:nuex}~\ref{item:sgConv} by setting $f=1$, and the observation in Remark~\ref{R:Pih} that the convergence in Theorem~\ref{thm:nuex}~\ref{item:sgConv} is equivalent to the ergodicity of the Markov chain with transition kernel $\Pi_h$ and unique invariant measure $\nu_h$.
\end{proof}

\subsection{Proof of Claims \ref{L*continuity} to \ref{cl:lncfs}} \label{sec:claim}
In this subsection, we prove all the claims used in the proof of Theorem \ref{thm:nuex}. We start with Claim \ref{L*continuity} on the continuity of the normalized maps $\widetilde L^*_\phi: {\cal M}_1(\Sigma) \to {\cal M}_1(\Sigma)$ and $\widetilde L^*_{\phi, m}: {\cal M}_1(\Sigma^{(m)}) \to {\cal M}_1(\Sigma^{(m)})$.

\begin{proof}[Proof of Claim \ref{L*continuity}]
We recall the definition of $\widetilde L^*_\phi$ from \eqref{tildeL*} via
$$
    \widetilde L^*_\phi \varrho = \frac{L_\phi^*\varrho}{(L^*_\phi \varrho)(\Sigma)}, \qquad \varrho \in {\cal M}_1(\Sigma).
$$
If $\varrho_n \in {\cal M}_1(\Sigma)$ and $\varrho_n\Rightarrow \varrho$, then the Feller property of $L_\phi^*$ established in Lemma \ref{cl:adjCont} implies that $L_\phi^* \varrho_n \Rightarrow L_\phi^* \varrho$. Hence, it is sufficient to show that $(L^*_\phi \varrho_n)(\Sigma) \to (L^*_\phi \varrho)(\Sigma)>0$, which follows readily from the fact that
$$
(L_\phi^* \varrho_n)(\Sigma) = \int_{\Sigma} \varrho_n({\rm d}\overline{x}) \int_S \mu({\rm d}z) e^{\phi(\overline{x}z)} = \langle \varrho_n, \psi\rangle,
$$
where $\psi(\overline{x}) := \int_S \mu({\rm d}z) e^{\phi(\overline{x}z)}\in [e^{-\Vert \phi\Vert_\infty}, e^{\Vert \phi\Vert_\infty}]$ is clearly a bounded continuous function on $\Sigma$.

The continuity of $\widetilde L^*_{\phi, m}: {\cal M}_1(\Sigma^{(m)}) \to {\cal M}_1(\Sigma^{(m)})$ follows by the same argument once we note that
\begin{enumerate}
    \item
the analogue of Lemma \ref{cl:adjCont} also holds for $L^*_{\phi, m}: {\cal M}(\Sigma^{(m)}) \to {\cal M}(\Sigma^{(m)})$ since for $f\in C_b(\Sigma^{(m)})$,
we also have $(L_{\phi, m}f)(\overline{x}) := \int_{S_m} \mu({\rm d}z) f(\overline{x}z) e^{\phi(\overline{x}z)} \in C_b(\Sigma^{(m)})$, and
\item $\psi_m(\overline{x}) := \int_{S_m} \mu({\rm d}z) e^{\phi(\overline{x}z)}\in [e^{-\Vert \phi\Vert_\infty}\mu(S_1), e^{\Vert \phi\Vert_\infty}]$
is a bounded continuous function on $\Sigma^{(m)}$, and $\mu(S_1)>0$ by assumption.
\end{enumerate}
\end{proof}

Next we prove Claim \ref{cl:tight} on the tightness of the family $(\nu^{(m)})_{m\in \N}$ of eigenmeasures from (\ref{eq:invMeas}).
\begin{proof}[Proof of Claim \ref{cl:tight}]
For $k\in \Z_-$, let $\pi_k: \Sigma \to S$ denote the $k$-th coordinate projection map defined via $\pi_k: \Sigma \ni \overline x \mapsto x_k$. In order to prove the tightness of $(\nu^{(m)})_{m\in \N}$, taking advantage of Tychonoff's theorem it is sufficient to show that the marginal distributions of each coordinate, i.e., $(\nu^{(m)}\circ \pi_k^{-1})_{m\in \N}$, form a tight family in ${\cal M}_1(S),$ for each $k\in \Z_-$.

To prove the latter, we fix $k\in \Z_-$ and note that
\begin{equation} \label{eq:homAdj}
\widetilde L_{\phi, m}^* (c\varrho) = \widetilde L_{\phi, m}^* \varrho \qquad \mbox{for all } c>0 \mbox{ and }  \varrho \in {\cal M}(\Sigma^{(m)}).
\end{equation}
For any Borel set $B\subset S$, observing that $\nu^{(m)} = \widetilde L_{\phi, m}^* \nu^{(m)} = (\widetilde L_{\phi, m}^*)^{|k|} \nu^{(m)},$ we then have
\begin{align*}
\nu^{(m)}(\pi_k^{-1}(B))
&= ((\widetilde L_{\phi, m}^*)^{|k|} \nu^{(m)})(\pi_k^{-1}(B))\\
 &= \frac{\idotsint_{\Sigma^{(m)}\times S^{|k|}_m} 1_{\{z_k \in B\}} \nu^{(m)}({\rm d}\overline{x})  \prod_{i=k}^{-1} e^{\phi(\overline{x}z_k\ldots z_{i})} \mu({\rm d}z_i)}{\idotsint_{\Sigma^{(m)}\times S^{|k|}_m}  \nu^{(m)}({\rm d}\overline{x})  \prod_{i=k}^{-1} e^{\phi(\overline{x}z_k\ldots z_{i})} \mu({\rm d}z_i)} \\
& \leq e^{2|k| \cdot \Vert \phi\Vert_\infty} \frac{\mu(B)}{\mu(S_m)} \leq \frac{e^{2|k| \cdot \Vert \phi\Vert_\infty}}{\mu(S_1)} \, \mu(B),
\end{align*}
where we took advantage of  \eqref{eq:homAdj} in the second equality.
Since the factor in front of $\mu(B)$ on the right-hand side  is independent of $m,$ it then follows that $(\nu^{(m)}\circ \pi_k^{-1})_{m\in \N}$ is a tight family in ${\cal M}_1(S)$.
\end{proof}

Next we prove Claim \ref{cl:cbd}, which establishes a list of properties for the set $\Lambda\subset C_b(\Sigma)$ defined in \eqref{eq:Lambda}.
\begin{proof}[Proof of Claim \ref{cl:cbd}]
In order to establish \ref{item1}, we observe that if $f \in \Lambda$, then $f>0$ and $\langle \nu, f\rangle=1$ imply that there exists $\overline{y} \in \Sigma$ with $f(\overline{y}) \in (0,1]$, and hence $f(\overline{x}) \leq B(\overline{x}, \overline{y}) f(\overline{y}) \leq M_1$ for all $\overline{x}\in\Sigma$, where $M_1$ was defined in \eqref{eq:supBfin}. Similarly, there exists $\overline{x} \in \Sigma$ with $f(\overline{x}) \geq 1$, and hence $f(\overline{y})\geq 1/M_1$ for all $\overline{y}\in\Sigma$.

For proving \ref{item2}, we note that if $f\in \Lambda$, then since $B$ is symmetric we have that
$$
|\log f(\overline{x}) - \log f(\overline{y})| \leq \log B(\overline{x}, \overline{y}).
$$
It suffices to show that $\lim_{\d(\overline{x}, \overline{y})\downarrow 0} \log B(\overline{x}, \overline{y})=0$, which implies the uniform continuity of $\log f$, uniform in $f\in \Lambda$. Together with \ref{item1}, this then implies \ref{item2}. To prove
$\log B(\overline{x}, \overline{y})\to 0$ as $\d(\overline{x}, \overline{y})\downarrow 0$, we start with choosing $\varepsilon>0$ arbitrary. Recalling $\var_k(\phi)$ from \eqref{eq:varDef} and letting $\varpi_\phi(\delta):=\sup_{\d(\overline{x}, \overline{y})\leq \delta} |\phi(\overline{x})-\phi(\overline{y})|$ denote the uniform modulus of continuity of $\phi$, we can then upper bound for any $k_0 \in \N$:
\begin{align*}
\log B(\overline x,\overline y) & = \sum_{k=1}^{\infty}  \sup_{z_1, z_2, \ldots z_k \in S} \big | \phi(\overline{x} z_1 z_2 \ldots z_k) - \phi(\overline{y} z_1 z_2 \ldots z_k) \big | \\
& \leq \sum_{k=k_0+1}^\infty \var_k(\phi) + \sum_{k=1}^{k_0}  \sup_{z_1, z_2, \ldots z_k \in S} \big | \phi(\overline{x} z_1 z_2 \ldots z_k) - \phi(\overline{y} z_1 z_2 \ldots z_k) \big | \\
& \leq  \sum_{k=k_0+1}^\infty \var_k(\phi) + k_0 \varpi_\phi(d(\overline{x}, \overline{y})).
\end{align*}
Since $\phi$ has summable variation and is uniformly continuous by
assumptions \ref{item:sumVar} and \ref{item:cont}, we can first choose $k_0$ sufficiently large and then choose $\delta>0$ sufficiently small such that uniformly in $\overline{x}, \overline{y}\in \Sigma$ with $\d(\overline{x}, \overline{y})\leq \delta$, we have $\log B(\overline{x}, \overline{y})\leq \varepsilon$. This finishes the proof of \ref{item2}.

Regarding \ref{item3}, the convexity of $\Lambda$ immediately follows from its definition. The claim that $\Lambda$ is closed under pointwise convergence also follows from the definition and the uniform bounds obtained in \ref{item1}.

For \ref{item4}, we proceed as in \cite{Wa-75}. Fix $\overline x, \overline y \in \Sigma$ and $f \in \Lambda$. First note that
\begin{align*}
\frac{(L_{\widetilde{\phi}}f)(\overline x)}{(L_{\widetilde{\phi}}f)(\overline y)} &= \frac{\int_S   f(\overline{x}z) e^{\widetilde \phi(\overline{x}z)} \mu({\rm d}z)}{\int_S  f(\overline{y}z)  e^{\widetilde \phi(\overline{y}z)} \mu({\rm d}z)} \\
      &\leq \sup_{z \in S} \frac{f(\overline{x}z) e^{\widetilde{\phi}(\overline{x} z)} }
      { f(\overline{y}z) e^{\widetilde{\phi}(\overline{y}z)}}
      \leq \sup_{z \in S} \big(e^{\widetilde{\phi}(\overline{x} z) - \widetilde{\phi}(\overline{y}z)} B(\overline{x}z, \overline{y}z) \big),
\end{align*}
with $B$ as in \eqref{Bxy}. The term on the right-hand side can then be upper bounded by
\begin{eqnarray}\label{est:B}
 \lefteqn{e^{\widetilde{\phi}(\overline{x}z) - \widetilde{\phi}(\overline{y}z)} B(\overline{x}z, \overline{y} z)}\nonumber\\
   &=&   e^{\widetilde{\phi}(\overline{x}z) - \widetilde{\phi}(\overline{y}z)} \exp\Big\{\sum_{k=1}^{\infty} \sup_{z_1, z_2, \ldots z_k \in S} \big| \widetilde{\phi}(\overline{x}z z_1\ldots z_k) - \widetilde{\phi}(\overline{y}z z_1\ldots z_k) \big| \Big\}\nonumber\\
  &\leq& \exp\Big\{\sum_{k=1}^{\infty} \sup_{w_1, w_2, \ldots w_k \in S} \big| \widetilde{\phi}(\overline{x}w_1\ldots w_k) - \widetilde{\phi}(\overline{y}w_1\ldots w_k)\big| \Big\} = B(\overline{x}, \overline{y}),
\end{eqnarray}
which establishes the inequality
\begin{equation} \label{est:bxylp}
 (L_{\widetilde{\phi}}f)(\overline x) \leq B(\overline x,\overline y)  (L_{\widetilde{\phi}}f)(\overline y).
\end{equation}
Since $f\in \Lambda$ and $\inf f>0$ by \ref{item1}, we also infer that $ L_{\widetilde{\phi}}f >0$. The fact that $ L_{\widetilde{\phi}}^*\nu =  \nu$ implies that $\langle \nu,  L_{\widetilde{\phi}}f\rangle = \langle L_{\widetilde{\phi}}^*\nu, f\rangle = \langle \nu, f\rangle =1$. Altogether, we therefore, deduce $L_{\widetilde{\phi}}f\in \Lambda$ and, since $f \in \Lambda$ was chosen arbitrarily,
$L_{\widetilde{\phi}} (\Lambda) \subset \Lambda$ also.
\end{proof}

Next we prove Claim \ref{cl:uconlg} on the uniform
equicontinuity of the family $(\Pi_h^m \psi)_{m\in\N},$ for any $\psi\in C_{b, u}(\Sigma)$.
\begin{proof}[Proof of Claim \ref{cl:uconlg}]
Let $\psi \in C_{b,u}(\Sigma)$. For any $m\in\N$ and $\overline x, \overline y \in \Sigma,$ we can rewrite (cf.~\eqref{pimh})
\begin{align*}
&(\Pi_h^m \psi)(\overline x) - (\Pi_h^m \psi)(\overline y)\\
&\quad = \int_{S^m} \Big( \psi(\overline{x}z_1\ldots z_m) G_m(\overline{x}z_1\ldots z_m) - \psi(\overline{y}z_1\ldots z_m) G_m(\overline{y}z_1\ldots z_m) \Big)
 \prod_{i=1}^m \mu({\rm d}z_i),
\end{align*}
where $G_m(\overline{x}z_1\ldots z_m)\prod_{i=1}^m \mu({\rm d}z_i)$ is just the probability measure $\Pi^m_h(\overline{x}, \cdot)$, with
\begin{equation}\label{Gmbarx}
G_m(\overline{x}z_1\ldots z_m) :=  \frac{h(\overline{x}z_1\ldots z_m)}{\lambda^m h(\overline{x})} e^{\sum_{i=1}^m \phi(\overline{x}z_1\ldots z_i)}.
\end{equation}
As a consequence, we can upper bound
\begin{eqnarray}\label{eqna:i1i2}
&&  \big |(\Pi_h^m \psi)(\overline x) - (\Pi_h^m \psi)(\overline y)\big | \nonumber \\
  &\leq &\int_{S^m} \big | G_m(\overline{x} z_1 \ldots z_m) -   G_m(\overline{y} z_1 \ldots z_m)\big | \,\big | \psi(\overline{y} z_1 \ldots z_m) \big | \prod_{i=1}^m \mu({\rm d}z_i)  \nonumber \\
  &&+ \int_{S^m} \big | \psi(\overline{x} z_1 \ldots z_m) -  \psi(\overline{y} z_1 \ldots z_m) \big | G_m(\overline{x} z_1 \ldots z_m)\prod_{i=1}^m \mu({\rm d}z_i) \nonumber \\
  &\leq& \Vert\psi\Vert_\infty \idotsint_{S^m}  \Big | \frac{G_m(\overline{x} z_1 \ldots z_m)}{G_m(\overline{y}z_1 \ldots z_m)} -1 \Big|\,  G_m(\overline{y}z_1 \ldots z_m) \prod_{i=1}^m \mu({\rm d}z_i) \nonumber \\
  && +  \sup_{z_1, \ldots, z_m \in S } |\psi(\overline{x} z_1 \ldots z_m) -  \psi(\overline{y} z_1 \ldots z_m)|.
\end{eqnarray}
The supremum on the right-hand side can be controlled via the uniform modulus of continuity of $\psi$ given by
$$
\varpi_\psi(\delta):= \sup_{\d(\overline{x}, \overline{y})\leq \delta} |\psi(\overline{x}) - \psi(\overline{y})|
$$
since $\d(\overline{x}z_1\ldots z_m, \overline{y}z_1 \ldots z_m) < \d(\overline{x}, \overline{y})$ by the definition of $\d(\cdot, \cdot)$ in \eqref{eq:prodMet}. To  bound the first term in \eqref{eqna:i1i2}, since $G_m(\overline{y}z_1 \ldots z_m) \prod_{i=1}^m \mu({\rm d}z_i)$ is a probability kernel, it suffices to give a uniform bound on
\begin{equation} \label{eq:GQuotDist}
\Big| \frac{G_m(\overline{x} z_1 \ldots z_m)}{G_m(\overline{y} z_1 \ldots z_m)} -1 \Big| = \Big |\frac{h(\overline{y})}{h(\overline{x})}\cdot \frac{h(\overline{x}z_1\ldots z_m)}{h(\overline{y}z_1\ldots z_m)}\cdot e^{\sum_{i=1}^m (\phi(\overline{x}z_1\ldots z_i) - \phi(\overline{y}z_1\ldots z_i))} - 1 \Big |.
\end{equation}
We now note that Claim \ref{cl:uconlg} is only used in the proof of Theorem~\ref{thm:nuex}~\ref{item:sgConv}, and that we had already proved in Theorem~\ref{thm:nuex}~\ref{item:harmonic} that $h$ is uniformly continuous and bounded away from $0$ and $\infty$. As a consequence, the two quotients of $h(\cdot)$ above can be made arbitrarily close to $1$ if $\d(\overline{x}, \overline{y})$ is sufficiently small. The same is true for the exponential term since we can bound
\begin{align*}
\Big|\sum_{i=1}^m (\phi(\overline{x}z_1\ldots z_i) - \phi(\overline{y}z_1\ldots z_i))\Big|
& \leq \sum_{i=1}^\infty |\phi(\overline{x}z_1\ldots z_i) - \phi(\overline{y}z_1\ldots z_i)| \\
& \leq \sum_{i=1}^M |\phi(\overline{x}z_1\ldots z_i) - \phi(\overline{y}z_1\ldots z_i)| + \!\!\!\!\! \sum_{i=M+1}^\infty {\rm var}_i (\phi),
\end{align*}
which can be made arbitrarily small by first choosing $M$ large and using the assumption that $\phi$ has summable variation, and then letting $\d(\overline{x}, \overline{y})$ be sufficiently small and using the uniform continuity of $\phi$. Noting that the bounds are uniform in $m\in\N$, this concludes the proof of Claim \ref{cl:uconlg}.
\end{proof}

Next we prove Claim \ref{cl:fstarinf}, which was critical in establishing that $\psi_*$ was constant in the proof of Theorem~\ref{thm:nuex}~\ref{item:sgConv}.

\begin{proof}[Proof of Claim \ref{cl:fstarinf}]
Fix $m\in\N$. In this proof, we will consider a subsequence of $(S_n)_{n\in\N}$, renamed as $(\widetilde S_n)_{n\in\N}$, such that
\begin{equation}\label{Smcond}
\sum_{n=1}^\infty e^{(\Vert\phi\Vert_\infty- \log\lambda)^+ n} \mu(\widetilde S_n^{\mathsf c})<\infty.
\end{equation}
We then define $\widetilde \Sigma^{(m)}$ from $(\widetilde S_n)_{n\in\N}$ the same way as we define $\Sigma^{(m)}$ from $(S_n)_{n\in\N}$ in \eqref{Sigm}. For $\psi\in C_{b,u}(\Sigma)$, analogous to the reasoning below \eqref{eq:i1b}, by going into a further subsequence if necessary, we can also assume that $(\Pi_h^{m_i}\psi)_{i\in\N}$ converges pointwise on $\Sigma$ to $\psi_* \in C_{b,u}(\Sigma)$ with the same modulus of continuity as $(\Pi_h^m \psi)_{m\in\N}$, and the convergence is uniform on each of the compact sets $(\widetilde \Sigma^{(m)})_{m\in\N}$ defined above.

First note that
\begin{equation}\label{infPih}
\inf_{\overline{x}\in \Sigma} \Pi_h^m \psi_* (\overline{x}) \geq \inf_{\overline{x} \in \Sigma} \psi_*(\overline{x})
\end{equation}
because $\Pi_h$ is a Markov operator, and recall that
\begin{equation}\label{infPmh}
\Pi_h^m\psi_* (\overline{x}) = \idotsint_{S^m} \psi_*(\overline{x}z_1\ldots z_m) G_m(\overline{x}z_1\ldots z_m) \prod_{i=1}^m \mu({\rm d}z_i)
\end{equation}
where $G_m(\overline{x}z_1\ldots z_m) :=  \frac{h(\overline{x}z_1\ldots z_m)}{\lambda^m h(\overline{x})} e^{\sum_{i=1}^m \phi(\overline{x}z_1\ldots z_i)}$ was defined in \eqref{Gmbarx}.

To prove the reverse inequality to \eqref{infPih}, we will show that for any $\varepsilon>0$,
\begin{equation}\label{2epspsi}
\inf_{\overline{x}\in \Sigma} \Pi^m_h\psi_*(\overline{x}) \leq \inf_{\overline{x}\in \Sigma} \psi_*(\overline{x})+ 2\varepsilon.
\end{equation}
Since $\inf\Pi^m_h\psi_*$ is non-decreasing in $m$, it suffices to prove \eqref{2epspsi} with $m$ replaced by some $n_0\geq m$, to be chosen later on. The basic idea then is to approximate $\psi_*$ in \eqref{infPmh} by $\Pi_h^{m_L}\psi$ for some large $L$. Since by \eqref{eq:infmon} we have
$$
\inf_{\overline{x}} \Pi_h^{n_0} \Pi_h^{m_L} \psi(\overline{x}) = \inf_{\overline{x}} \Pi_h^{n_0+m_L} \psi (\overline{x}) \leq \inf_{\overline{x}} \psi_* (\overline{x}),
$$
it only remains to show that for some suitable $n_0\geq m$ and $L$,
\begin{equation}\label{Pin0}
\inf_{\overline{x}\in \Sigma} \Pi_h^{n_0} \psi_*(\overline{x}) \leq \inf_{\overline{x}\in \Sigma} \Pi_h^{n_0} \Pi_h^{m_L}\psi(\overline{x}) + 2\varepsilon.
\end{equation}
To see how to choose $n_0$ and $L$, note that given $n_0+m_L$, we can find $\overline{w}\in \Sigma$ such that
$$
\Pi_h^{n_0+m_L} \psi(\overline{w}) \leq \inf_{\overline{x}\in \Sigma} \Pi_h^{n_0} \Pi_h^{m_L}\psi(\overline{x}) + \varepsilon.
$$
As a consequence, in order to prove \eqref{Pin0}, it is then sufficient to bound the difference
\begin{align}
& \big|\Pi_h^{n_0} \psi_*(\overline{w}) - \Pi_h^{n_0} (\Pi_h^{m_L}\psi) (\overline{w})\big| \nonumber \\
\leq \, & \idotsint_{S^{n_0}} \big| \psi_*(\overline{w}z_1\ldots z_{n_0}) - \Pi_h^{m_L}\psi (\overline{w}z_1\ldots z_{n_0})\big|
G_{n_0}(\overline{w}z_1\ldots z_{n_0}) \prod_{i=1}^{n_0} \mu({\rm d}z_i) \label{Pin0K}
\end{align}
from above by $\varepsilon$.
We recall the compact set $\widetilde \Sigma^{(n_0)}=\{(\ldots x_{-2} x_{-1}): x_{-i} \in \widetilde S_{n_0+i} \mbox{ for all } i\in\N\}$. If $\overline{w}\notin \widetilde \Sigma^{(n_0)}$, then in $\psi_*(\overline{w}z_1\ldots z_{n_0})$ and $\psi(\overline{w}z_1\ldots z_{n_0})$ above, we first replace $\overline{w}$ by any point $\overline{w}'\in \widetilde\Sigma^{(n_0)}$. The integral in \eqref{Pin0K} can then be upper bounded by
\begin{equation}\label{varpiPi}
     2\varpi(2^{-n_0})
    + \idotsint_{S^{n_0}} \big| \psi_*(\overline{w}'z_1\ldots z_{n_0}) - \Pi_h^{m_L}\psi (\overline{w}'z_1\ldots z_{n_0})\big|
G_{n_0}(\overline{w}z_1\ldots z_{n_0}) \prod_{i=1}^{n_0} \mu({\rm d}z_i),
\end{equation}
where $\varpi(\cdot)$ denotes a uniform modulus of continuity for $(\Pi_h^{m_i}\psi)_{i\in\N}$ and $\psi_*$, and we used the fact that $\d(\overline{w}z_1\ldots z_{n_0}, \overline{w}'z_1\ldots z_{n_0})\leq 2^{-n_0}$ by the definition of $\d$ in \eqref{eq:prodMet}. The first term in \eqref{varpiPi} can then be made smaller than $\varepsilon/3$ by choosing $n_0$ large.

For the second term in \eqref{varpiPi}, if we restrict the integral  to $\overline{w}'z_1\ldots z_{n_0}\in \widetilde\Sigma^{(n_0)}$, then---using again that $G_{n_0}(\overline{w}z_1\ldots z_{n_0}) \prod_{i=1}^{n_0} \mu({\rm d}z_i)$ is a probability kernel---the contribution can be bounded from above by
$$
\sup_{\overline{x} \in \widetilde\Sigma^{(n_0)}} \big|\psi_*(\overline{x}) - \Pi_h^{m_L}\psi(\overline{x})\big|,
$$
which, given $n_0$, can be made smaller than $\varepsilon/3$ by choosing $L$ sufficiently large since $\Pi_h^{m_i}\psi \to \psi_*$ uniformly on $\widetilde \Sigma^{(n_0)}$.

If we relabel $z_1\ldots z_{n_0}$ as $z_{n_0}\ldots z_1$ and restrict the integral in \eqref{varpiPi} to $\overline{w}'z_{n_0}\ldots z_1 \in (\widetilde\Sigma^{(n_0)})^{\mathsf c}$, then since  $\overline{w}'\in \widetilde\Sigma^{(n_0)}$, we must have $z_i\not\in \widetilde S_{n_0+i}$ for some $1\leq i\leq n_0$. The restricted integral can therefore, be bounded from above by
\begin{align} \label{2psiinf}
& 2 \Vert \psi\Vert_\infty \sum_{i=1}^{n_0}
\idotsint_{S^{n_0}} 1_{\{z_i\notin \widetilde S_{n_0+i} \}}
\frac{h(\overline{w}z_{n_0}\ldots z_1)}{\lambda^{n_0} h(\overline{w})} e^{\sum_{i=1}^{n_0} \phi(\overline{w}z_{n_0}\ldots z_{n_0-i+1})} \prod_{i=1}^{n_0} \mu({\rm d}z_i) \\
\leq\ &  2 \Vert \psi\Vert_\infty \frac{\sup h}{\inf h}\sum_{i=1}^{n_0} e^{(\Vert \phi\Vert_\infty -\log \lambda)^+n_0} \mu(\widetilde S_{n_0+i}^{\mathsf c}) =
2 \Vert \psi\Vert_\infty \frac{\sup h}{\inf h}\sum_{n=n_0+1}^{2n_0} e^{(\Vert \phi\Vert_\infty -\log \lambda)^+n} \mu(\widetilde S_n^{\mathsf c}), \nonumber
\end{align}
which can be made smaller than $\varepsilon/3$ by choosing $n_0$ sufficiently large, using the assumption \eqref{Smcond} and the fact that $h$ is bounded away from $0$ and $\infty$.

Collecting the estimates above, we have shown that by first choosing $n_0$ sufficiently large and then $L$ large, the right-hand side of \eqref{Pin0K} can be bounded from above by $\varepsilon$, which implies \eqref{Pin0} and concludes the proof of Claim \ref{cl:fstarinf}.
\end{proof}

\begin{proof}[Proof of Claim \ref{cl:lncfs}]
Recall the definition of $(\widetilde S_n)_{n\in\N}$ and $(\widetilde \Sigma^{(m)})_{m\in\N}$ from \eqref{Smcond} and the line that follows. Suppose that the subsequence $(\Pi_h^{m_i} \psi)_{i\in\N}$ converges uniformly to $\psi_*$ on each $\widetilde \Sigma^{(n)}$, where $\psi_*$ has been shown to be a constant using Claim \ref{cl:fstarinf}. We need to strengthen it to show uniform convergence to $\psi_*$ on each $\widetilde \Sigma^{(n)}$ along the full sequence $(\Pi_h^m \psi)_{m\in\N}$, which then implies uniform convergence on each $\Sigma^{(m)}$ (since $\Sigma^{(m)} \subset \widetilde \Sigma^{(m)}$) and pointwise convergence on $\Sigma$ (thanks to the uniform equicontinuity of $(\Pi_h^m \psi)_{m\in\N}$ and the fact that $\cup \widetilde \Sigma^{(m)}$ is dense in $\Sigma$).

Fix $n\in\N$ and $\varepsilon>0$. We will show that $\sup_{\overline{x}\in \widetilde \Sigma^{(n)}} |\Pi_h^m\psi (\overline{x}) -\psi^*| \leq \varepsilon$ for all $m$ sufficiently large. The argument is similar to the proof of Claim \ref{cl:fstarinf}, except now we approximate $\Pi^{m_L}\psi$ by $\psi_*$ instead of the other way round.

To start with, thanks to \eqref{Smcond}, we can choose $n_0\geq n$ sufficiently large such that
$$
2 \Vert \psi\Vert_\infty \frac{\sup h}{\inf h}\sum_{i=n_0+1}^\infty e^{(\Vert \phi\Vert_\infty -\log \lambda)^+i} \mu(\widetilde S_i^{\mathsf c}) \leq \varepsilon/2.
$$
Next, we choose $m_L$ from the sequence $(m_i)_{i\in\N}$ such that
$$
\sup_{\overline{x}\in \widetilde \Sigma^{(n_0)}} |\Pi_h^{m_L}\psi (\overline{x}) -\psi^*| \leq \varepsilon/2.
$$
For all $m= \widetilde m+ m_L$ with $\widetilde m\geq n_0$ and $\overline{x} \in \widetilde \Sigma^{(n)} \subset \widetilde \Sigma^{(n_0)}$, we can therefore bound
\begin{align*}
\big|\Pi_h^m \psi(\overline{x}) - \psi_*(\overline{x})\big|
& = \big|\Pi_h^{\widetilde m}\Pi_h^{m_L}\psi(\overline{x}) - \Pi_h^{\widetilde m} \psi_*(\overline{x})\big| \nonumber \\
& \leq \idotsint_{S^{\widetilde m}} \big| \Pi_h^{m_L}\psi (\overline{x}z_{\widetilde m}\ldots z_1) - \psi_* \big|
G_{\widetilde m}(\overline{x} z_{\widetilde m}\ldots z_1) \prod_{i=1}^{\widetilde m} \mu({\rm d}z_i).
\end{align*}
Separating the integral into $\overline{x}z_1 \ldots z_{\widetilde m} \in \widetilde \Sigma^{(n_0)}$ or $(\widetilde \Sigma^{(n_0)})^{\mathsf c}$, we can then apply the same calculations as the ones after \eqref{varpiPi} to obtain a bound of $\varepsilon$, the only difference now being that $\sum_{i=1}^{n_0}$ in \eqref{2psiinf} should be replaced by $\sum_{i=1}^{\widetilde m}$, which does not affect the final bound. The bound $\varepsilon$ is uniform in $m\geq n_0+m_L$ and $\overline{x} \in \widetilde \Sigma^{(n)}$, which concludes the proof of Claim \ref{cl:lncfs}.
\end{proof}

\subsection{A mixing result and an invariance principle} \label{sec:decay}

In this section, we extend a result of Pollicott \cite{Po-00} on the temporal decay of correlations for a stationary $\Sigma$-valued Markov process with transition kernel $\Pi_h$ and stationary distribution $\nu_h$ defined in Remark \ref{R:Pih}. More precisely, in \cite{Po-00} Pollicott considered the setting of a finite alphabet $S$ equipped with the counting measure. We will now extend this result to the setting of an uncountable non-compact alphabet $S$ equipped with a metric $\d$ and a probability measure $\mu$. Subsequently, we will show how this result can be used to prove an invariance principle for additive functionals of the stationary Markov process with transition kernel $\Pi_h$
and stationary distribution $\nu_h$. In what follows, $\Vert \cdot \Vert_p$ will denote $L^p$ norms for functions on the space $(\Sigma, \d, \nu_h)$.

\begin{theorem}\label{T:Pollicott}
Let $(S, \d, \mu)$, $\Sigma=S^{\Z_-}$, and $\phi\in C_b(\Sigma)$ satisfy the same assumptions as in Theorem \ref{thm:nuex}. Recall $\lambda$, $h$, $\nu$, $\Pi$, $\nu_h$, and $\Pi_h$ from Theorem \ref{thm:nuex} and Remark \ref{R:Pih}. Furthermore, assume that $\var_n(\phi) =O(n^{-r})$ for some $r>1$, and let $\bar f\in C_b(\Sigma)$ be such that $\int \bar f(\overline{x}) \, \nu_h({\rm d}\overline{x})=0$, and $\var_n(\bar f) = O(n^{-r+1})$. Then for all $\varepsilon>0$, there exists $0< C_{\phi, \bar f}<\infty$ depending only on $\phi$ and $\bar f$ such that
\begin{equation}\label{Pol3}
 \Vert \Pi_h^n \bar f\Vert_2 \leq \frac{C_{\phi, \bar f}}{n^{r-1-\varepsilon}}, \qquad n\in \N.
\end{equation}
Furthermore, if $\phi$ is replaced by $\widetilde \gamma\phi$ with $\widetilde \gamma\in [0,1]$, then the constants $(C_{\widetilde \gamma\phi, \bar f})_{\widetilde \gamma \in [0,1]}$ are uniformly bounded.
\end{theorem}

We postpone the proof of Theorem \ref{T:Pollicott} and first provide the following corollary. Indeed, the $L^2$ bound in Theorem \ref{T:Pollicott} is  key in deducing the following invariance principle for additive functionals of the stationary $\Sigma$-valued Markov process with transition kernel $\Pi_h$ and stationary distribution $\nu_h$.

\begin{corollary}\label{C:invariance}
Assume that $\phi$ satisfies the same assumptions as in Theorem \ref{T:Pollicott} for some $r>1$. Let $\eta=(\eta_i)_{i\in\Z}$ be the unique stationary process such that $\overline{\eta}^{(n)}:=(\eta_{n+i})_{i\leq -1}$ has  distribution $\nu_h$ for each $n\in\Z$. Let $\widehat f\in L^2(\Sigma, \d, \nu_h)$ be such that $\int_\Sigma \widehat f(\overline{x}) \, \nu_h({\rm d} \overline{x})=0$, and let $Y_n := \sum_{j=1}^n \widehat f(\overline{\eta}^{(j)})$. Assume that $\bar f:=\Pi_h \widehat f \in C_b(\Sigma)$ and $\var_n(\bar f)=O(n^{-r+1})$.

Then, if $r>3/2$, the processes $(Y_{\lfloor Nt\rfloor}/\sqrt{N})_{t\in [0, 1]}$ converge in distribution to a Brownian motion $(\sigma B_t)_{t\in [0,1]}$ with a deterministic diffusion coefficient $\sigma^2\geq 0$, where
\begin{equation}\label{eq:sigma2}
\sigma^2 = \sum_{n\in \Z} \E\big[\widehat f(\overline{\eta}^{(0)})
\widehat f(\overline{\eta}^{(n)})\big] = \E\big[\widehat f(\overline{\eta}^{(0)})^2\big] + 2\sum_{n=1}^\infty \E\big[\widehat f(\overline{\eta}^{(0)})
\widehat f(\overline{\eta}^{(n)})\big].
\end{equation}
Furthermore, there exists $0 < C(r)<\infty$ depending only on $r$ such that for all $k\in\N$, we have
\begin{equation}\label{eq:sigmabd}
\sum_{n=k}^\infty \big| \E\big[\widehat f(\overline{\eta}^{(0)})
\widehat f(\overline{\eta}^{(n)})\big] \big| \leq C(r) C_{\phi, \bar f}^2 \Delta(k),
\end{equation}
where $\Delta(k):=\sum_{n=1}^\infty \frac{1}{n^{1/2}(n+k)^{(2r-1)/4}}$ and $C_{\phi, \bar f}$ are obtained from \eqref{Pol3}, with a choice of $\varepsilon = (2r-3)/4$.
\end{corollary}
\begin{proof}
The invariance principle follows by applying \cite[Corollary 12]{MePeUt-06} (see also \cite{PU2005, PU2006}) to the real valued stationary process $(\widehat f(\overline{\eta}^{(n)}))_{n\in \Z}$. The only condition we need to check is the finiteness of
\begin{equation} \label{eq:MPUcond}
\begin{aligned}
\sum_{n=1}^\infty \frac{1}{\sqrt n} \E\Big[\E\big[\widehat f(\overline{\eta}^{(n)}) | \overline{\eta}^{(0)}\big]^2\Big]^{1/2}
& = \sum_{n=1}^\infty \frac{1}{\sqrt n} \Vert \Pi_h^n \widehat f\Vert_2  \\
& = \sum_{n=1}^\infty \frac{1}{\sqrt{n}} \Vert \Pi_h^{n-1} \bar f\Vert_2\\
& \leq  \Vert \bar f\Vert_2+\sum_{n=2}^\infty \frac{1}{\sqrt{n}} \cdot \frac{2C_{\phi, \bar f}}{n^{(2r-1)/4}}<\infty,
\end{aligned}
\end{equation}
where the first inequality follows by applying Theorem \ref{T:Pollicott} with $\varepsilon = (2r-3)/4>0$.

In \cite[Corollary 12]{MePeUt-06}, the diffusive coefficient $\sigma^2$ is identified as the $L^1$ limit of $\frac{1}{n} \E[Y_n^2 \,  |\, {\cal I}]$, where ${\cal I}$ is the $\sigma$-algebra of translation invariant sets with respect to $\nu_h$. The ergodicity of $\nu_h$ as shown in Remark \ref{R:Pih} implies that ${\cal I}$ is trivial, and hence $\sigma^2$ is non-random and admits the representation \eqref{eq:sigma2}.

Lastly, we bound the left-hand side of \eqref{eq:sigmabd} by following the proof of Corollary 2 in \cite{PU2006}. To simplify notation, for $n\in\Z$, write $\xi_n:= \widehat f(\overline{\eta}^{(n)})$ and denote by ${\cal F}_n$ the $\sigma$-algebra generated by $\overline{\eta}^{(n)}$. Then the $({\cal F}_n)_{n \in \Z}$ form a filtration, and by martingale decomposition we can write
$$
\xi_n =  \sum_{i=-\infty}^n P_i (\xi_n).
$$
where $P_i (\xi_n) := \E[\xi_n | {\cal F}_i] -
\E[\xi_n | {\cal F}_{i-1}].$
Noting that a similar decomposition holds for $\xi_0$, and  observing that $P_i(\xi_n)$ and $P_j(\xi_0)$ are uncorrelated when $i\neq j$ and $n \in \Z,$ we deduce that
\begin{align*}
\big|\E[\xi_n\xi_0]\big| = \Big|\sum_{i\leq 0} \E[P_i(\xi_n) P_i(\xi_0)]\Big|  \leq \sum_{i\leq 0} \Vert P_i(\xi_n)\Vert_2 \Vert P_i(\xi_0)\Vert_2
= \sum_{i\leq 0} \Vert P_{i-n}(\xi_0)\Vert_2 \Vert P_i(\xi_0)\Vert_2,
\end{align*}
and hence for $k\in\N$, we have
\begin{align}\label{eq:sumcovbd}
\sum_{n=k}^\infty \big|\E[\xi_n\xi_0]\big| \leq \sum_{i\leq 0} \Vert P_i(\xi_0)\Vert_2 \sum_{n\geq k} \Vert P_{i-n}(\xi_0)\Vert_2 \leq
\Big(\sum_{i\leq 0} \Vert P_i(\xi_0)\Vert_2\Big) \Big(\sum_{i\leq -k} \Vert P_i(\xi_0)\Vert_2 \Big).
\end{align}
By applying \cite[Lemma A.2]{PU2006}, for $k\geq 2$, we obtain the first inequality in
\begin{align*}
\sum_{i\leq -k} \Vert P_i(\xi_0)\Vert_2 & \leq 3 \sum_{n=1}^\infty \frac{1}{\sqrt n} \Big( \sum_{i\leq 1-n-k} \Vert P_{i}(\xi_0)\Vert_2^2 \Big)^{1/2} \\
& = 3 \sum_{n=1}^\infty \frac{1}{\sqrt n} \Big( \sum_{i\leq 1-n-k} \E\Big[\big(\E[\xi_0 | {\cal F}_i] -
\E[\xi_0 | {\cal F}_{i-1}] \big)^2 \Big] \Big)^{1/2} \\
&  = 3 \sum_{n=1}^\infty \frac{1}{\sqrt n} \Big( \sum_{i\leq 1-n-k} \E\Big[\E[\xi_0 | {\cal F}_i]^2 -
\E[\xi_0 | {\cal F}_{i-1}]^2 \Big] \Big)^{1/2} \\
& \leq 3 \sum_{n=1}^\infty \frac{1}{\sqrt n}
\E\big[\E[\xi_0 | {\cal F}_{1-n-k}]^2\big]^{1/2} \\
& = 3 \sum_{n=1}^\infty \frac{1}{\sqrt n} \Vert \Pi_h^{n+k-1} \widehat f \Vert_2 \\
&= 3 \sum_{n=1}^\infty \frac{1}{\sqrt n} \Vert \Pi_h^{n+k-2} \bar f \Vert_2 \leq \sum_{n=1}^\infty \frac{3C_{\phi, \bar f}}{n^{1/2}(n+k-2)^{(2r-1)/4}},
\end{align*}
where the right-hand side is uniformly bounded in $k$ and tends to $0$ as $k\to\infty$ since $r>3/2$ by assumption. Substituting this bound into \eqref{eq:sumcovbd} then gives \eqref{eq:sigmabd}.
\end{proof}

\begin{proof}[Proof of Theorem \ref{T:Pollicott}]
The $L^1$ analogue of \eqref{Pol3} appears in \cite[Section 4(1)]{Po-00}, where their $g$ corresponds to our $\phi$, and their $g_0(\overline{x})$ corresponds to
\begin{equation}\label{eq:phi0}
\phi_0(\overline{x}) := \phi(\overline{x}) + \log h(\overline{x}) - \log h(\theta \overline{x}) -\log \lambda, \qquad \overline{x}\in \Sigma,
\end{equation}
with $\theta \overline{x} :=  (x_{i-1})_{i\leq -1}\in \Sigma$ being the canonical shift map. Note that $\Pi_h(\overline{x}, {\rm d}\overline{y}) = 1_{\{\overline{y} = \overline{x} z\}} e^{\phi_0(\overline{x}z)} \mu({\rm d}z)$. Although \cite{Po-00} only considered the case of a finite alphabet $S$, we can adapt their proof strategy. One difference is that, instead of summing over $z\in S$, we integrate with respect to  $\mu({\rm d}z)$ on $S$. Another difference is that we prove the stronger $L^2$ bound instead of the $L^1$ bound, which requires some new ideas.

Let $\eta=(\eta_n)_{n\in\Z}$ be the stationary process introduced at the beginning of the section, with $(\eta_{n+i})_{i\leq -1}$ having law $\nu_h$ for each $n\in\Z$. For each $k\in\N$, we define
\begin{equation} \label{Ak}
\begin{aligned}
(A_k \bar f)(\overline{z})=(A_k \bar f)(z_{-k}\ldots z_{-1}) & := \E\big[\bar f((\eta_i)_{i\leq -1}) | \eta_{-k}\ldots \eta_{-1} = z_{-k} \ldots z_{-1}\big] \\
& = \frac{\int_\Sigma \bar f(\overline{y}z_{-k}\ldots z_{-1})
e^{\sum_{i=1}^k \phi_0(\overline{y}z_{-k}\ldots z_{-i})} \, \nu_h({\rm d}\overline{y}) }{\int_\Sigma
e^{\sum_{i=1}^k \phi_0(\overline{y}z_{-k}\ldots z_{-i})} \, \nu_h({\rm d}\overline{y})}.
\end{aligned}
\end{equation}
We note that
$$
\Vert A_k\bar f\Vert_\infty \leq \Vert \bar f\Vert_\infty
$$
and
$$
\langle A_k\bar f , \nu_h\rangle=\int_\Sigma (A_k\bar f)(\overline{z}) \, \nu_h({\rm d}\overline{z})=\int_\Sigma \bar f(\overline{z}) \, \nu_h({\rm d}\overline{z}) = \langle \bar f , \nu_h\rangle,
$$
and following \cite[(3.2)]{Po-00}, if $n=km$ for some $k, m\in \N$, then we can decompose our quantity of interest $\Vert \Pi_h^n \bar f\Vert_2$ and upper bound it via
\begin{equation}\label{eq:Literate}
\Vert \Pi_h^n \bar f\Vert_2
\leq \Vert \Pi_h^n \bar f - (\Pi_h^k A_k)^m \bar f\Vert_2 + \Vert (\Pi_h^k A_k)^m \bar f\Vert_2.
\end{equation}
Recalling the representation of $\Pi_h^k$ from \eqref{Gmbarx}, we note that the operator $\Pi_h^k A_k$ can be written as an integral operator with
\begin{align*}
(\Pi_h^k A_k \bar f)(\overline{x}) &=
\int_{S^k} (A_k \bar f)(z_{-k}\ldots z_{-1}) e^{\sum_{i=1}^k \phi_0(\overline{x}z_{-k}\ldots z_{-i})} \prod_{i=1}^k \mu({\rm d}z_{-i}) \\
&=
\int_\Sigma K(\overline{x}, \overline{z}) \bar f(\overline{z}) \, \nu_h({\rm d}\overline{z}),
\end{align*}
where
\begin{align}\label{eq:kernUnity}
\begin{split}
K(\overline{x}, \overline{z}) &:=
K_k(\overline{x}, \overline{z}) := \frac{e^{\sum_{i=1}^k \phi_0(\overline{x}z_{-k}\ldots z_{-i})}}{\int_\Sigma e^{\sum_{i=1}^k \phi_0(\overline{y}z_{-k}\ldots z_{-i})} \, \nu_h({\rm d}\overline{y})}, \\
&\mbox{with}\quad \int_\Sigma K(\overline{x}, \overline{z}) \, \nu_h ({\rm d}\overline{x}) =1.
\end{split}
\end{align}
Furthermore, we note that uniformly in $k\in\N$, $\overline{x}, \overline{y}\in \Sigma$ and $(z_{-i})_{1\leq i\leq k}\in S^k$, we have the bound
\begin{equation}\label{eq:5varphi}
\begin{aligned}
\frac{e^{\sum_{i=1}^k \phi_0(\overline{x}z_{-k}\ldots z_{-i})}}{e^{\sum_{i=1}^k \phi_0(\overline{y}z_{-k}\ldots z_{-i})}}
& = \frac{e^{\sum_{i=1}^k \phi(\overline{x}z_{-k}\ldots z_{-i})}}{e^{\sum_{i=1}^k \phi(\overline{y}z_{-k}\ldots z_{-i})}}
\cdot \frac{h(\overline{x}z_{-k}\ldots z_{-1})}{h(\overline{x})}\cdot \frac{h(\overline{y})}{h(\overline{y}z_{-k}\ldots z_{-1})} \\
& \geq \Cl[small]{const:ebd}:= e^{-5\sum_{i=0}^\infty \var_i(\phi)}>0,
\end{aligned}
\end{equation}
which follows readily from the bounds on $h$ in Theorem \ref{thm:nuex} and Claim \ref{cl:cbd}, and the assumption \ref{item:sumVar} that $\phi$ has summable variation.
As a consequence, we then have uniformly in $k\in\N$ that
\begin{equation} \label{eq:bddRatios}
0 < \Cr{const:ebd} \leq \inf_{\overline x, \overline z} K(\overline x, \overline z)  \leq 1 \leq  \sup_{\overline x, \overline z} K(\overline x, \overline z ) \leq \frac{1}{\Cr{const:ebd}} < \infty.
\end{equation}
Next, analogous to \cite[Lemma 3]{Po-00}, we can show that the $\Pi_h^kA_k$, $k\in\N$, are strict contractions in the sense that for all $k\in\N$ and all $\bar f\in L^1(\Sigma, \nu_h)$ with $\langle \bar f, \nu_h\rangle=0$,
\begin{equation}\label{contract}
    \Vert \Pi_h^kA_k \bar f\Vert_1 \leq (1-\Cr{const:ebd}) \Vert \bar f\Vert_1.
\end{equation}
Indeed, using the assumption $\int_\Sigma \bar f(\overline{x}) \, \nu_h({\rm d}\overline{x})=0$ and letting $\widetilde K(\cdot, \cdot) := \frac{K(\cdot, \cdot)-\Cr{const:ebd}}{1-\Cr{const:ebd}}$,
this $L^1$ contraction bound can be seen by computing
\begin{align*}
\Vert \Pi_h^kA_k \bar f\Vert_1  & = \int_\Sigma \Big|\int_\Sigma K(\overline{x}, \overline{z}) \bar f(\overline{z}) \, \nu_h({\rm d}\overline{z})\Big| \, \nu_h({\rm d}\overline{x}) \\
& = (1-\Cr{const:ebd}) \int_\Sigma \Big|\int_\Sigma \widetilde K(\overline{x}, \overline{z}) \bar f(\overline{z}) \, \nu_h({\rm d}\overline{z})\Big| \,\nu_h({\rm d}\overline{x}) \\
&\leq (1-\Cr{const:ebd}) \int_{\Sigma} |\bar f(\overline{z})| \int_\Sigma \widetilde K(\overline{x}, \overline{z}) \, \nu_h({\rm d}\overline{x}) \, \nu_h({\rm d}\overline{z}) = (1-\Cr{const:ebd}) \Vert \bar f\Vert_1,
\end{align*}
where we used $\int_\Sigma K(\overline{x}, \overline{z}) \, \nu_h({\rm d}\overline{x})=1$.

We observe that the same argument can also be used to bound the $L^2$ norm. Indeed, we have
\begin{align}
\Vert \Pi_h^kA_k \bar f\Vert_2  & = \Big(\int_\Sigma \Big|\int_\Sigma K(\overline{x}, \overline{z}) \bar f(\overline{z}) \, \nu_h({\rm d}\overline{z})\Big|^2 \, \nu_h({\rm d}\overline{x})\Big)^{\frac{1}{2}} \notag\\
& = (1-\Cr{const:ebd}) \Big(\int_\Sigma \Big|\int_\Sigma \widetilde K(\overline{x}, \overline{z}) \bar f(\overline{z}) \, \nu_h({\rm d}\overline{z})\Big|^2 \, \nu_h({\rm d}\overline{x})\Big)^{\frac{1}{2}} \notag\\
&\leq (1-\Cr{const:ebd}) \Big(\int_{\Sigma} \int_\Sigma \widetilde K^2(\overline{x}, \overline{z}) \bar f^2(\overline{z}) \, \nu_h({\rm d}\overline{z}) \, \nu_h({\rm d}\overline{x}))\Big)^{\frac{1}{2}} \notag\\
&\leq (1-\Cr{const:ebd}) \Big(\int_{\Sigma} \bar f^2(\overline{z}) \int_\Sigma \frac{\frac{1}{\Cr{const:ebd}}-\Cr{const:ebd}}{1-\Cr{const:ebd}}\widetilde K(\overline{x}, \overline{z}) \, \nu_h({\rm d}\overline{x}) \, \nu_h({\rm d}\overline{z})\Big)^{\frac{1}{2}} \notag\\
& = \Big(\frac{(1-\Cr{const:ebd})(1-\Cr{const:ebd}^2)}{\Cr{const:ebd}}\Big)^{\frac{1}{2}} \Vert \bar f\Vert_2, \label{contract2}
\end{align}
which is a contraction only if $\Cr{const:ebd}<1$ is sufficiently close to $1$.

To obtain a contraction in $L^2$ without assuming  $\Cr{const:ebd}<1$ being close to $1$, we consider powers of $\Pi_h^k A_k$. Let $m_0\in \N$ be chosen such that $1-\Cl[small]{const:sm1} := 2(1-\Cr{const:ebd})^{m_0/2}<1$. We then define the linear operator $T$ on $L^1(\Sigma, \nu_h)$ by
\begin{equation}\label{eq:Tpsi}
T \psi(\overline{x}) := (\Pi_h^k A_k)^{m_0} \psi(\overline{x}) - \langle \psi, \nu_h\rangle = (\Pi_h^k A_k)^{m_0} \bar\psi(\overline{x}) = T \bar\psi(\overline{x}),
\end{equation}
where $\bar \psi:= \psi - \langle \psi, \nu_h\rangle$. Note that $\langle T\psi, \nu_h\rangle=0$, and we have the operator norm bounds
\begin{gather*}
    \Vert T\psi \Vert_\infty  \leq \Vert \bar \psi\Vert_\infty \leq 2 \Vert \psi\Vert_\infty, \\
    \Vert T\psi \Vert_1 = \Vert (\Pi_h^k A_k)^{m_0} \bar \psi\Vert_1 \leq (1-\Cr{const:ebd})^{m_0} \Vert \bar \psi\Vert_1 \leq 2 (1-\Cr{const:ebd})^{m_0} \Vert \psi\Vert_1,
\end{gather*}
where we have applied the $L^1$ bound \eqref{contract} $m_0$ times. We can now apply Riesz-Thorin interpolation (see e.g.~\cite[Chapter 2, Theorem 2.1]{Stein-11}) to obtain
$$
    \Vert T\psi \Vert_2  \leq 2(1-\Cr{const:ebd})^{m_0/2} \Vert\psi\Vert_2 = (1-\Cr{const:sm1}) \Vert \psi\Vert_2.
$$
In particular, we have
\begin{equation}\label{contract3}
    \Vert T\bar f\Vert_2 = \big \Vert (\Pi_h^kA_k)^{m_0} \bar f \big \Vert_2 \leq (1-\Cr{const:sm1}) \Vert \bar f\Vert_2.
\end{equation}
We now assume $m=m_0\ell$ and $n=km_0\ell$ for some $k, \ell \in \N$.
Applying the contraction bound \eqref{contract3} to the second term in the right-hand side of \eqref{eq:Literate} then provides us with
\begin{align} \label{eq:LnormBd}
\big \Vert (\Pi_h^k A_k)^m \bar f \big \Vert_2 = \Vert T^\ell \bar f\Vert_2 \le (1-\Cr{const:sm1})^\ell \Vert \bar f \Vert_2.
\end{align}
For the first term in the right-hand side of \eqref{eq:Literate}, we can again use \eqref{contract3} to bound
\begin{align}
    \big \Vert (\Pi_h^k A_k)^m \bar f - \Pi_h^{km} \bar f \big \Vert_2 &\leq \sum_{i=0}^{\ell-1} \big \Vert T^{\ell-i}(\Pi_h^{km_0})^i \bar f - T^{\ell-i-1}(\Pi_h^{km_0})^{i+1} \bar f \big \Vert_2 \notag \\
    &= \sum_{i=0}^{\ell-1} \big \Vert T^{\ell-i-1} (T - \Pi_h^{km_0}) (\Pi_h^{km_0})^i \bar f \big \Vert_2 \notag \\
    &\le\sum_{i=0}^{\ell-1} (1-\Cr{const:sm1})^{\ell-i-1}
    \big \Vert (T - \Pi_h^{km_0}) (\Pi_h^{km_0})^i  \bar f \big  \Vert_2. \label{PiAkdiff}
\end{align}
Furthermore, the last norm can be bounded from above by
\begin{align}
 \big \Vert (T - \Pi_h^{km_0}) (\Pi_h^{km_0})^i \bar f \big \Vert_2  & = \big \Vert ( (\Pi_h^k A_k)^{m_0} - \Pi_h^{km_0}) (\Pi_h^{km_0})^i \bar f \big \Vert_2 \notag\\
& \leq \sum_{j=0}^{m_0-1} \big \Vert (\Pi_h^k A_k)^{m_0-j-1} \Pi_h^k(A_k-1) \Pi_h^{kj} (\Pi_h^{km_0})^i \bar f \big \Vert_2 \notag\\
& \leq \sum_{j=0}^{m_0-1} (1-\Cr{const:sm1})^{m_0-j-1}
\big \Vert \Pi_h^k (A_k-1) \Pi_h^{kj} (\Pi_h^{km_0})^i \bar f \big \Vert_2. \label{PiAkdiff2}
\end{align}
This can now be upper bounded by applying the following analogue of \cite[Lemma 4]{Po-00}.
\begin{lemma}\label{L:varkPi}
\begin{enumerate}[label=(\roman*)]
    \item \label{item:2infBd} For all $k\in\N$ and $\psi\in L^1(\Sigma, \nu_h)$ we have
    \[
    \Vert A_k \psi - \psi\Vert_2\leq \Vert A_k \psi - \psi\Vert_\infty \leq \var_k(\psi).\]

    \item \label{item:varkBd} For some $C_\phi\in (0,\infty)$ depending only on $\phi$, and for all $k, \ell\in\N$ and $\psi\in L^1(\Sigma, \nu_h)$, we have
    \begin{align*}
        \var_k(\Pi_h^\ell \psi) \leq  C_\phi \Vert\psi\Vert_\infty  \sum_{i=k+1}^\infty \var_i(\phi) + \var_{k+\ell}(\psi).
    \end{align*}
\end{enumerate}
\end{lemma}\noindent
It is not difficult to see that \ref{item:2infBd} follows from the definition of $A_k$ in \eqref{Ak}, while \ref{item:varkBd} follows from the calculations in \eqref{eqna:i1i2} and \eqref{eq:GQuotDist}. No specific properties of the alphabet space $S$ are needed.

Continuing with \eqref{PiAkdiff2}, we then infer that
\begin{align*}
  \big \Vert \Pi_h^k (A_k-1) \Pi_h^{kj} (\Pi_h^{km_0})^i \bar f \big \Vert_2
  &\leq \big \Vert (A_k - 1) (\Pi_h^k)^{im_0+j} \bar f \big \Vert_2  \\
  & \leq \var_{k}\big(\Pi_h^{k(im_0+j)} \bar f\big) \\
  & \leq C_\phi \Vert\bar f\Vert_\infty \sum_{a=k+1}^\infty \var_a(\phi) + \var_{k(im_0+j+1)}(\bar f) \\
  & \leq \frac{C'_\phi\Vert \bar f\Vert_\infty}{k^{r-1}} + \frac{C_{\bar f}}{(k(im_0+j+1))^{r-1}},
\end{align*}
where we used our assumptions on $\var_j(\phi)$ and $\var_j(\bar f)$. Substituting this bound into \eqref{PiAkdiff2} then provides us with
$$
\big \Vert (T - \Pi_h^{km_0}) (\Pi_h^{km_0})^i \bar f \big \Vert_2
\leq \frac{C'_{\phi, \bar f}}{k^{r-1}},
$$
where $C'_{\phi, \bar f}$ depends only on $\phi$ and $\bar f$. Plugging this bound into \eqref{PiAkdiff} and combining it with \eqref{eq:LnormBd}, we then deduce from \eqref{eq:Literate} that for $n=km_0\ell$ with $k, \ell\in \N$, we have
$$
\Vert \Pi_h^n \bar f\Vert_2 \leq C''_{\phi, \bar f} \big((1-\Cr{const:sm1})^\ell + k^{1-r} \big).
$$
For any $\varepsilon>0$, we can then choose $k=\lfloor n^\beta\rfloor$ with $\beta \in (0,1)$ sufficiently close to $1$ such that if $n=km_0\ell$ for some $\ell\in\N$, then $\Vert \Pi_h^n \bar f\Vert_2 \leq C_{\phi, \bar f} n^{1+\varepsilon -r}$. For $\ell\notin \N$, we can use $\Vert \Pi_h^n \bar f\Vert_2 \leq \Vert \Pi_h^{\lfloor \ell\rfloor km_0} \bar f\Vert_2$ to obtain the same bound with an adjusted $C_{\phi, \bar f}$ that depends only on $\phi$ and $\bar f$.

As the last step, an inspection of the proof shows that if $\phi$ is replaced by $\widetilde \gamma\phi$ for some $\widetilde \gamma\in [0,1]$, then we can choose $C_{\widetilde \gamma\phi, \bar f} := C_{\phi, \bar f}$ for all  $\widetilde \gamma \in [0,1]$.
This concludes the proof of Theorem \ref{T:Pollicott}.
\end{proof}

\section{Thermodynamic formalism for the trapping problem} \label{sec:reform}

In this section, we will rewrite the trapping problem from Section \ref{sec:Intro} in the language of thermodynamic formalism introduced in Section  \ref{sec:thermo}. We then verify that the potential $\varphi$ that arises satisfies the assumptions \ref{item:sumVar}--\ref{item:cont}
so that the results from Section \ref{sec:thermo} can be applied.

\subsection{Reformulating the trapping problem} \label{sec:setup}
Recall from \eqref{eq:condAnnProb2} that
$Z^{\gamma}_{t,X} = \E^\xi \big[ \exp \big\{ - \gamma \int_0^t \xi(s,X_s) \, {\rm d}s \big\}
\big]$. By integrating out the Poisson field $\xi$ and using the Feynman-Kac formula, it was shown in \cite[Section 2.1]{DrGaRaSu-10} that
\begin{equation} \label{eq:ZtX}
Z^{\gamma}_{t,X} = \exp \left\{-\alpha\gamma \int_0^t v(s, X)\, {\rm d}s \right\},
\end{equation}
where
\begin{equation} \label{eq:vX}
v(s, X) := \E^Y_{X_s}\left[\exp \left\{ - \gamma\int_0^s \delta_0(Y_r-X_{s-r})\,{\rm d}r\right\}   \right];
\end{equation}
here, $Y$ is a random walk that starts from $X_s$ and follows the law of a time reversed mobile trap, which coincides with the law of a trap by our symmetry assumption. Note that $Z^\gamma_{t,X}$ can be interpreted as the Gibbs weight of the trajectory $(X_s)_{0\leq s\leq t}$ under $P_t^\gamma$ conditioned on annealed survival up to time $t$, and that $v(s, X)$ depends only on the past increments of $X$ up to time $s$, i.e., $(X_s-X_{s-r})_{0\leq r \leq s}$.  More precisely, we have that
\begin{equation} \label{eq:GibbsRewrite}
P_t^\gamma (X \in \cdot) = \frac{\E_0^X\Big[\exp \Big\{ -\alpha \gamma \int_0^t v(s,X)\, {\rm d}s \Big\} \ind{1}_{X \in \cdot} \Big ]}{\E_0^X\Big[\exp \Big\{ -\alpha \gamma \int_0^t v(s,X)\, {\rm d}s \Big\} \Big]} = \frac{\E_0^X[Z_{t,X}^\gamma \ind{1}_{X \in \cdot} ]}{\E_0^X[Z_{t,X}^\gamma]}, \quad t \ge 0.
\end{equation}

We can also define $P^\gamma_t$ when $X$ has a non-empty past $(X_s)_{s\in (-T, 0]}$, with $T$ possibly equal to $\infty$. A unified treatment is to define $X_s:=*$ for some cemetery state when $s<-T$ and regard $(X_s)_{s<0}$ as a c\`adl\`ag path taking values in $\Z^d \cup \{*\}$. We then generalize
the definition of $v(s, X)$ in \eqref{eq:vX} by setting
\begin{align} \label{eq:vX2}
v(s, X) & := \E^Y_{X_s}\left[\exp \left\{ - \gamma\int_0^\infty \delta_0(Y_r-X_{s-r})\,{\rm d}r\right\}   \right],
\end{align}
with $v(s, X):=0$ if $X_s=*$. We denote the resulting path measure by $P^\gamma_{X, t}$ to emphasize the dependence on the past $(X_s)_{s<0}$.

To study $P^\gamma_t$ via the thermodynamic formalism,
it is necessary to approximate $P^\gamma_t$ by $P^\gamma_{X, t}$ for some c\`adl\`ag $X: (-\infty, 0)\to \Z^d$. A special choice is to sample a random $(X_s)_{s< 0}$ according to a probability distribution $\widetilde \nu$ that we will specify later on, and independently sample $(X_s)_{s\geq 0}$ according to a simple symmetric random walk with $X_0=0$. Denote by $\E^X_{\widetilde \nu}[\cdot]$ the expectation with respect to  such an $(X_s)_{s\in \R}$.
Analogous to $P^\gamma_{X, t}$ for a deterministic past $(X_s)_{s<0}$, we can define the Gibbs measure
\begin{equation} \label{eq:GibbsRewrite2}
P_{\widetilde \nu, t}^\gamma (X \in \cdot) = \frac{\E_{\widetilde \nu}^X[Z_{t, X}^\gamma \ind{1}_{X \in \cdot} ]}{\E_{\widetilde \nu}^X[Z_{t, X}^\gamma]}, \quad t \ge 0.
\end{equation}

We now explain how measures such as $P^\gamma_{\widetilde \nu, t}$ and
$P^\gamma_{X, t}$ with a c\`adl\`ag $X:(-\infty, 0)\to\Z^d$ fit
into the framework of the thermodynamic formalism introduced in Section \ref{sec:thermo}. To start with, we give a precise definition of our alphabet space $S$ as a complete separable metric space equipped with a probability measure $\mu$. For this purpose, let
\begin{equation}\label{eq:pathS}
S:=\{0\}\cup \bigcup_{m\in\N} \Big\{ ((t_1, a_1), \ldots, (t_m, a_m)): 0\leq t_1 \leq \cdots \leq t_m\leq 1, a_1, \ldots, a_m \in \{\pm e_1, \ldots, \pm e_d\} \Big\},
\end{equation}
where $(\pm e_i)_{1\leq i\leq d}$ are the $2d$ unit vectors in $\Z^d$, $0$ encodes the constant path of unit length at the origin in $\Z^d$, and the ordered tuple $((t_1, a_1), \ldots, (t_m, a_m))$ encodes the path that starts from $0$ and makes the increment $a_i$ at time $t_i$ for each $1\leq i\leq m$. Note that we admit multiple nearest neighbor increments at the same time, the sum of which correspond to one aggregated increment. For $x \in S$, let $|x|$ denote the number of increments in $x$. We can then define the  metric
\begin{equation}\label{eq:pathd}
    \d(x, x') := \big ||x|-|x'| \big| + 1_{\{|x|=|x'|\}} \min \Big\{1, \sum_{i=1}^{|x|} |t_i-t'_i| +\sum_{i=1}^{|x|} 1_{\{a_i\neq a'_i\}}  \Big\},
\end{equation}
on $S$, for  $x = ((t_1, a_1), \ldots, (t_m, a_m)), x' = ((t_1', a_1'), \ldots, (t_n', a_n'))\in S.$ It is not hard to check that $(S, \d)$ is a complete separable metric space. As in \eqref{eq:prodMet}, in a slight abuse of notation we will also use $\d$ to denote the metric on $\Sigma=S^{\Z_-}$. This metric will turn out to be convenient to prove the required regularity of the potential $\varphi$ later on. The probability measure $\mu$ on $S$ is simply the one induced by the increments of the simple symmetric random walk on a unit time interval.

Conversely, given any $x\in S$, we can reconstruct a c\`adl\`ag path on the unit time interval starting from the origin, which we will also denote by $x=(x_s)_{s\in [0,1]}$. If $x$ contains multiple increments at the same time, then we simply treat the sum of these increments as a single increment. If there is an increment at time $0$, then we shift the starting point of $x$ accordingly in order to obtain a c\`adl\`ag trajectory. Similarly, from any $\overline{x}=(x_i)_{i\leq -1}\in \Sigma$, we can reconstruct a c\`adl\`ag path $X: (-\infty, 0]\to \Z^d$ with $X_0=0$, whose increments on the time interval $[i,i+1]$ are determined by $x_i$, that is,
$$
X_{i+s} = X_i +x_i(s), \qquad i\in \Z_-,\ s\in (0,1],
$$
which is equivalent to the backward recursive relation
\begin{equation}\label{eq:x2X}
X_{i+s} = X_{i+1} -x_i(1) +x_i(s), \qquad i\in \Z_-,\ s\in (0, 1].
\end{equation}
The conditions that $X_0=0$ and that $X$ has to be c\`adl\`ag then uniquely determine $(X_s)_{s\leq 0}$ from $\bar x=(x_i)_{i\leq -1}$. Let $K$ denote the map from $\overline{x} \in \Sigma$ to $(X_s)_{s\leq 0}$ so that $(K\overline{x})_s=X_s$ for $s\leq 0$. Note that if for each $i\leq -1$, the law of $x_i$ is absolutely continuous with respect to  the law $\mu$ of the increments of a simple symmetric random walk on a unit interval, then almost surely $K$ is invertible, and hence we will frequently identify $\overline{x}\in \Sigma$ with the path $(X_s)_{s\leq 0}=K\overline{x}$.

We can now define a potential $\varphi: \Sigma \to \R$ with
\begin{align} \label{eq:varphiDef}
    \varphi(\overline x) := -\alpha \gamma \int_{-1}^0  v(s, X) \, {\rm d}s, \qquad \bar x\in \Sigma,
\end{align}
where $v(s, X)$ was defined in \eqref{eq:vX2} and $(X_s)_{s\leq 0}=K\overline{x}$. In the definition of $P^\gamma_{\widetilde \nu, t}$ in \eqref{eq:GibbsRewrite2}, if $\overline{x}=(x_i)_{i\leq -1}$ encodes the increments of the past $(X_s)_{s< 0}$, which has law $\nu:= \widetilde \nu \circ K$, and $z_i$ encodes the increments of $X$ on the time interval $[i-1, i]$, $i\in\N$, which are i.i.d.\ with law $\mu$, then we can rewrite
\begin{equation}\label{tZgttX}
Z^\gamma_{t, X} = \exp\Big\{
\sum_{i=1}^t \varphi(\overline{x}z_1\cdots z_i)
\Big\}
\end{equation}
for $t\in \N$.
In the language of the thermodynamic formalism introduced in Section \ref{sec:thermo}, for $t\in \N$, we have
$$
\E^X_{\widetilde \nu}[Z^\gamma_{t, X}] =
\idotsint_{S^t} \exp\Big\{
\sum_{i=1}^t \varphi(\overline{x}z_1\cdots z_i)
\Big\} \,  \nu({\rm d}\overline{x}) \prod_{i=1}^t \mu({\rm d}z_i)
= \langle \nu, L_\varphi^t 1\rangle = \langle (L_\varphi^*)^t \nu, 1\rangle,
$$
where $L_\varphi$ is the Ruelle transfer operator with alphabet space $(S, \d, \mu)$ and potential $\varphi$. Writing $L_\varphi 1=\Pi 1$ and $L_{\varphi}^* \nu = \nu \Pi$ in terms of the non-conservative Markov transition kernel $\Pi(\overline{x}, {\rm d}\overline{y})$, we note that the normalized probability measure on $\Sigma$,
\begin{equation} \label{tildePgamma}
\overline{P}^\gamma_{\nu, t} :=
\frac{(L_\varphi^*)^t\nu}{\langle (L_\varphi^*)^t \nu, 1\rangle} = \frac{\nu \Pi^t}{\langle \nu \Pi^t, 1\rangle},
\end{equation}
is exactly the law of the increments of the time shifted path $(X_{t+s})_{s\leq 0}$, where $X$ is distributed according to the path measure $P_{\widetilde \nu, t}^\gamma(X\in \cdot)$. When $\nu$ is the delta measure at $\overline{w}\in \Sigma$, we will just write $\overline{P}^\gamma_{\overline{w}, t}$. This formulation in the language of thermodynamic formalism will allow us to apply the results from Section \ref{sec:thermo}.

For the Gibbs measure $P^\gamma_t$ with an empty past, or $P^\gamma_{X, t}$ with a finite past $X: [-n, 0)\to\Z^d$, we can still write the partition function $Z^\gamma_{t, X}$ in the same form as in \eqref{tZgttX}, provided we extend the definition of $\varphi$ to $\bigcup_{n=1}^\infty S^n$ by setting
\begin{align} \label{eq:varphiDef2}
    \varphi((x_i)_{-n\leq i\leq -1}) := -\alpha \gamma \int_{-1}^0 v(s, X) \, {\rm d}s, \qquad (x_i)_{-n \leq i\leq -1}\in S^n,
\end{align}
where $(X_s)_{s\in [-n, 0]}$ is defined from the increments $(x_i)_{-n\leq i\leq -1}$ and $X_s=*$ for all $s<-n$.

\subsection{Regularity of the potential} \label{sec:rpot}
We now show that in dimensions $d\geq 5$, the potential $\varphi$ defined in \eqref{eq:varphiDef} satisfies the assumptions \ref{item:sumVar}--\ref{item:cont}, so that Theorem \ref{thm:nuex} can be applied.

\begin{proposition} \label{prop:sumVar}
Let $(S,\d, \mu)$ be defined as in \eqref{eq:pathS}--\eqref{eq:pathd}, and let $\Sigma=S^{\Z_-}$ with metric $\d$ defined as in
\eqref{eq:prodMet}. Let $\varphi$ be
the potential defined in \eqref{eq:varphiDef}
and \eqref{eq:varphiDef2}. Then
for dimension $d\geq 3$, $\varphi$ is bounded and satisfies the following properties:
\begin{enumerate}[label=\alph*)]
  \item \label{item:varn}
     There exists $C \in (0,\infty)$ such that  ${\rm var}_n(\varphi) \le C n^{-\frac{d}{2} +1}$ for all $n\in\N$. In particular, $\varphi$ has summable variation for $d\geq 5$.

    \item \label{item:finInfDiff}
    There exists $C \in (0,\infty)$ such that for all $\overline x \in \Sigma$ we have $|\varphi (\overline x) - \varphi((x_i)_{-n \le i \le -1})| \le C n^{-\frac{d}{2} +1}$ for all $n\in\N$.

    \item \label{item:UC}  $\varphi$ is uniformly continuous on $(\Sigma \cup \bigcup_{n=1}^\infty S^n, \d)$ with $\d$ defined in \eqref{eq:prodMet2}.
\end{enumerate}
\end{proposition}


\begin{proof}
\begin{enumerate}[label=\alph*)]
\item
From the definition of $\varphi$ in \eqref{eq:varphiDef} and $v$ in \eqref{eq:vX2}, it is clear that $\Vert \varphi\Vert_\infty\leq \alpha\gamma$. Now let $n\in\N$, and let $\overline{u}=(u_i)_{i\leq -1}, \overline{w}=(w_i)_{i\leq -1}\in \Sigma=S^{\Z_-}$ such that $u_i=w_i$ for $-n\leq i\leq -1$. Denote by $U, W: (-\infty, 0]\to \Z^d$, with $U(0)=W(0)=0$, the c\`adl\`ag paths constructed from $\overline{u}$ and $\overline{w}$ respectively. Then we have
\begin{align*}
   |\varphi(\overline{u}) -\varphi(\overline{w}) |
   & \leq \alpha\gamma \int_{-1}^0 |v(s, U) - v(s, W)| \, {\rm d}s.
\end{align*}
By the definition of $v$ in \eqref{eq:vX2} and the Lipschitz continuity of $e^{-a}$ on $[0, \infty)$, for $s\in (-1, 0)$, we have
\begin{align}
    |v(s, U) - v(s, W)| & =  \Big| \E^Y_{U_s}\left[e^{- \gamma\int_0^\infty \delta_0(Y_r- U_{s-r})\,{\rm d}r}   - e^{ - \gamma\int_0^\infty \delta_0(Y_r-W_{s-r})\,{\rm d}r} \right] \Big| \notag\\
    & \leq \gamma \int_0^\infty  \E^Y_{U_s} \big|\delta_0(Y_r- U_{s-r}) -  \delta_0(Y_r-W_{s-r}) \big| \, {\rm d}r \notag \\
& \leq  \gamma \int_{n+s}^\infty  \big(\P^Y_{U_s}(Y_r=U_{s-r})+ \P^Y_{U_s}(Y_r=W_{s-r})\big) \, {\rm d}r \notag \\
    & \leq \gamma \int_{n+s}^\infty \frac{C}{r^{d/2}} \,{\rm d}r \leq \frac{C\gamma}{n^{d/2-1}}, \label{eq:tvdiff}
\end{align}
where we used the observation that $U(\cdot)=W(\cdot)$ on $[-n, 0]$, $n\in\N$, $s\in (-1,0]$, and a local central limit theorem bound on the transition kernel of $Y$ (see e.g.\ \cite[Theorem 2.1.1]{LaLi-10}). This concludes the proof of \ref{item:varn}.

\item The proof proceeds similarly to that of \ref{item:varn}. Indeed, denote by $X$ the infinite c\`adl\`ag path constructed from $\overline x$, and $X^{(n)}$ the modified trajectory with $X^{(n)}_s=X_s$ for $s\in [-n, 0]$ and $X^{(n)}_s=*$ for all $s<-n$. Then by the definition of $\varphi$ in \eqref{eq:varphiDef} and \eqref{eq:varphiDef2}, we have
\begin{align*}
   |\varphi(\overline{x}) -\varphi((x_i)_{-n\leq i\leq -1})) |
   & \leq \alpha\gamma \int_{-1}^0 |v(s, X) - v(s, X^{(n)})| \, {\rm d}s
\end{align*}
and thus, similarly to \ref{item:varn},
\begin{align*}
    |v(s, X) -  v(s, X^{(n)})| & =  \Big| \E^Y_{X_s}\left[e^{- \gamma\int_0^\infty \delta_0(Y_r- X_{s-r})\,{\rm d}r}   - e^{ - \gamma\int_0^s \delta_0(Y_r-X^{(n)}_{s-r})\,{\rm d}r} \right] \Big| \notag\\
    & \leq \gamma \int_{n+s}^\infty \frac{C}{r^{d/2}} \,{\rm d}r \leq \frac{C\gamma}{n^{d/2-1}},
\end{align*}
which proves the claim.

\item Fix any $\epsilon>0$. Due to \ref{item:varn}, we can first choose $n_0\in\N$ large enough such that ${\rm var}_{n_0}(\varphi)\leq \epsilon/2$. Now for any $\overline{u}, \overline{w}\in \Sigma$ with $\d(\overline{u}, \overline{w})\leq \delta$ for some $\delta>0$ to be chosen later, define $\overline{z}=(z_i)_{i\leq -1}$ with $z_i=u_i$ for $i<-n_0$ and $z_i=w_i$ for $-n_0\leq i\leq -1$.
Then  $|\varphi(\overline{z}) -\varphi(\overline{w})|\leq {\rm var}_{n_0}(\varphi)\leq \epsilon/2$, and by the definition of $(\Sigma, \d)$ in \eqref{eq:prodMet}, we furthermore have
$$
\d(\overline{z}, \overline{u}) \leq \d(\overline{w}, \overline{u}) \leq \delta.
$$
By the definition of $(S, \d)$ in \eqref{eq:pathd}, we can choose $\delta>0$ sufficiently small to ensure that for any choice of $\overline{u}$ and $\overline{w}$ with $\d(\overline{u}, \overline{w})\leq \delta$, and for each $-n_0\leq i\leq -1$, $u_i$ and $w_i$ will have exactly the same number $|u_i|=|w_i|$ of increments  and the same sequence of increments. Furthermore, if we denote the vectors of the times of increments of $u_i$ and $w_i$ by
$t^{u_i}, t^{w_i}\in [0,1]^{|u_i|}$, and  denote their $\ell_1$ distance by $\vert t^{u_i}-t^{w_i}\vert_1$, then (recall \eqref{eq:prodMet}) we can infer that
$$
\sum_{i=-n_0}^{-1} 2^{i} \vert t^{u_i}-t^{w_i}\vert_1 \leq \delta.
$$
If $Z: (-\infty, 0]\to \Z^d$ with $Z(0)=0$ denotes the
c\`adl\`ag path constructed from $\overline{z}$, then we note that $Z$ is simply a time change of $U$ such that $Z(s)=U(s)$ for all $s\leq -n_0$ and such that
$$
\int_{-\infty}^0 1_{\{Z(s)\neq U(s)\}} \,{\rm d}s \leq \sum_{i=-n_0}^{-1} \vert t^{u_i}-t^{w_i}\vert_1 \leq 2^{n_0}\delta.
$$
Substituting this bound into \eqref{eq:tvdiff}  with $\overline{w}$ and $W$ replaced by $\overline{z}$ and $Z$, we then obtain
$$
|\varphi(\overline{u}) - \varphi(\overline{z})| \leq \alpha \gamma^2 2^{n_0}\delta,
$$
which can be made smaller than $\epsilon/2$ by choosing $\delta$ sufficiently small. Together with our choice of $n_0$, this implies that $|\varphi(\overline{u}) - \varphi(\overline{w})| \leq \epsilon$ uniformly in $\d(\overline{u}, \overline{w})\leq \delta$, which concludes the proof of the uniform continuity of $\varphi$ on $\Sigma$. When either $\overline{u}$ or $\overline{w}$, or both are replaced by elements of $\bigcup_{n\in\N}S^n$, the proof proceeds analogously.
\end{enumerate}
\end{proof}

\section{Proof of Theorem \ref{thm:fCLT}} \label{sec:fCLT}
Having set up our trapping problem in the language of thermodynamic formalism, we are now prepared to apply results from Section \ref{sec:thermo} to deduce Theorem \ref{thm:fCLT}. The coarse proof strategy is to compare the increments of the path under the Gibbs measure $P^\gamma_t$, defined in \eqref{eq:Gibbs}, with a stationary process such that we can apply the invariance principle in Corollary \ref{C:invariance}.

We verified in Proposition \ref{prop:sumVar} that the potential $\varphi$ defined in \eqref{eq:varphiDef} satisfies assumptions \ref{item:sumVar}--\ref{item:cont}. Therefore, we can apply Theorem \ref{thm:nuex} and identify the eigenvalue $\lambda>0$,  positive eigenfunction $h$, and eigenmeasure $\nu$ (which are all uniquely determined) for the transfer operator $L_\varphi$, which is associated with the non-conservative Markov transition kernel $\Pi(\overline{x}, {\rm d}\overline{y}) = 1_{\{\overline{y} =\overline{x}z\}} e^{\varphi(\overline{x}z)} \mu({\rm d}z)$. Following Remark \ref{R:Pih}, we can then define the associated conservative Markov transition kernel $\Pi_h(\overline{x}, {\rm d}\overline{y}) = \frac{h(\overline{y})}{\lambda h(\overline{x})} \Pi(\overline{x}, {\rm d}\overline{y})$ such that $\nu_h({\rm d}\overline{x}) = h(\overline{x}) \, \nu({\rm d}\overline{x})$ is the unique stationary distribution for a $\Sigma$-valued Markov chain with transition kernel $\Pi_h$.

Recall the path measure $P_{\widetilde \nu, t}^\gamma (X \in \cdot)$ from \eqref{eq:GibbsRewrite2}, where the reference measure for $X$ is chosen such that its past $(X_s)_{s\leq 0}$ follows a distribution $\widetilde \nu$ yet to be determined. Let $\overline{x}=(x_i)_{i\leq -1}\in \Sigma$
be the increments of $(X_s)_{s\leq 0}$ with induced law $\nu:= \widetilde \nu \circ K = \widetilde \nu \circ (K^{-1})^{-1}$, where $K$ was defined as the map that constructs a c\`adl\`ag path $(X_s)_{s\leq 0}$ from $\overline{x}\in \Sigma$ in \eqref{eq:x2X}. Recall furthermore from \eqref{tildePgamma} that if $X$ is distributed according to $P_{\widetilde \nu, t}^\gamma(X\in \cdot)$ for some $t\in\N$, then the increments
of the shifted path $(X_{t+s})_{s\leq 0}$ follow the distribution
\begin{equation}\label{tildePgamma2}
\overline{P}^\gamma_{\nu, t} ({\rm d}\overline{y}) := \frac{(\nu \Pi^t)({\rm d}\overline{y})}{\langle \nu, \Pi^t1\rangle}.
\end{equation}
We now choose $\nu$ to be precisely the eigenmeasure $\nu$ of $L_\varphi$, which in turn determines the law $\widetilde \nu = \nu \circ K^{-1}$ that the path measure $P^\gamma_{\widetilde \nu, t}$ depends upon. Again by Remark \ref{R:Pih} (and using the notation introduced therein), if $\Pi_h$ denotes the $h$-transformed conservative Markov operator  and $\frac{1}{h}(\bar y):= \frac{1}{h(\bar y)}$, then we have
\begin{align}\label{statred1}
\begin{split}
\overline{P}^\gamma_{\nu, t} ({\rm d}\overline{y}) &= \frac{\int \nu({\rm d}\overline{x}) h(\overline{x})  \Pi_h^t(\overline{x}, {\rm d} \overline{y})\frac{1}{h(\overline{y})}}{
\int \nu({\rm d}\overline{x}) h(\overline{x}) (\Pi_h^t (\frac{1}{h}))(\overline{x})} = \frac{\int \nu_h({\rm d}\overline{x})
 \Pi_h^t(\overline{x}, {\rm d} \overline{y})\frac{1}{h(\overline{y})}}{
\int \nu_h({\rm d}\overline{x})(\Pi_h^t (\frac{1}{h}))(\overline{x})} \\
&= \frac{\frac{1}{h(\overline{y})} \nu_h({\rm d}\overline{y})}{
\int \nu_h({\rm d}\overline{x})(\Pi_h^t (\frac{1}{h}))(\overline{x})},
\end{split}
\end{align}
where we used $\int \nu_h({\rm d}\overline{x}) \Pi_h^t(\overline{x}, {\rm d}\overline{y}) = \nu_h({\rm d}\overline{y})$. Further note that by \eqref{erg}, $(\Pi_h^t (\frac{1}{h}))(\overline{x})\to 1$ for every $\overline{x}\in \Sigma$ as $t\to\infty$, and hence the denominator in \eqref{statred1} tends to $1$.

If we could remove the factor $\frac{1}{h(\overline{y})}$ from the right-hand side of \eqref{statred1}, then the measure $\overline{P}^\gamma_{\nu, t} ({\rm d}\overline{y})$ would become
essentially $\nu_h$, which is stationary with respect to  the shift map $\theta \overline{x} =(x_{-1+i})_{i\leq -1}$. We could then apply the invariance principle in Corollary \ref{C:invariance}. We note that in some sense the factor $\frac{1}{h(\overline{y})}$ can be regarded as a boundary effect, similar to the choice of the past $(X_s)_{s\leq 0}$ in the definition of the path measures $P^\gamma_t$ and $P^\gamma_{\widetilde \nu, t}$ in Section \ref{sec:setup}. To treat both left and right boundary effects simultaneously and establish that they are negligible, we fix $1\ll L=L(t)\ll t$, say $L=\lceil \log t\rceil$, and compare the laws of the path increments under $P^\gamma_t$ and $P^\gamma_{\widetilde \nu, t}$ on the time interval $[L, t-L]$. Theorem \ref{thm:fCLT} will then be a consequence of the following findings.

The first result shows that under any Gibbs measure $P^\gamma_{\varsigma, t}$ as defined in \eqref{eq:GibbsRewrite2},
where the past $X: (-\infty, 0)\to \Z^d \cup\{*\}$ has (possibly degenerate) distribution $\varsigma$, the law of the path increments on any unit time interval is absolutely continuous with respect to  the law $\mu$ of the increments of a simple symmetric random walk, and furthermore the density is uniformly bounded from above. The same is true for the marginal distributions of the $\Sigma$-valued Markov process with transition kernel $\Pi_h$ and arbitrary initial state.

\begin{lemma}\label{L:abscon}
Let $\varsigma$ be the law of a random c\`adl\`ag path $(X_s)_{s<0}$ taking values in $\Z^d\cup\{*\}$, where $*$ is a cemetery state.
For $t>0$, let $P^\gamma_{\varsigma, t}$ be the Gibbs measure defined as in \eqref{eq:GibbsRewrite2}, with $\varsigma$ instead of $\widetilde \nu$. For any  $s\in [0, t]$, if $\rho_s$ denotes the law of the increments $(X_{s+r}-X_s)_{r\in [0,1]}$ under $P^\gamma_{\varsigma, t}(X\in \cdot)$, then
there exists $D\in (0,\infty)$ such that uniformly in $\varsigma$, $t>0$,  $s\in [0, t]$, and $z\in S$, we have
\begin{equation}\label{eq:rhosmu}
\frac{{\rm d}\rho_s}{{\rm d} \mu}(z) \leq D.
\end{equation}
If $\overline{X}_n=(x_i)_{i<n}$, $n\geq 0$,  is a $\Sigma$-valued Markov chain with transition kernel $\Pi_h$, and for $0\leq j< n$, $\rho_j$
denotes the marginal law of $x_j$ (does not depend on $n>j$), then the same bound $\frac{{\rm d}\rho_j}{{\rm d} \mu}(z)\leq D$ holds uniformly in the initial condition $\overline{X}_0=(x_i)_{i\leq -1}$, $j\geq 0$, and $z\in S$.
\end{lemma}
\begin{proof}
It suffices to show that conditioned on $Y^-:=(X_r)_{r<s}$ and $Y^+:=(X_{s+1+r}-X_{s+1})_{r\geq 0}$, the conditional law $P^\gamma_{\varsigma, t} ((X_{s+r}-X_s)_{r\in [0,1]}\in \cdot |Y^-, Y^+)$ is absolutely continuous with respect to  $\mu$ with a uniformly bounded density. By the definition of the Gibbs measure $P^\gamma_{\varsigma, t}$, we note that
\begin{align*}
   P^\gamma_{\varsigma, t} \big((X_{s+r}-X_s)_{r\in [0,1]}\in {\rm d}z |Y^-, Y^+\big) =  \frac{1}{Z_1} e^{-\alpha\gamma \int_s^t v(r, X)\, {\rm d}r} \, \mu({\rm d}z), \quad z\in S,
\end{align*}
where the path $X$ is constructed by concatenating $Y^-$, $z\in S$, $Y^+$, and shifted such that $X_{s+1}=0$, while $v$ was defined in \eqref{eq:vX2}, and $Z_1$ is a normalizing constant that
depends on $Y^+$ and $Y^-$. Let $X'$ be the path constructed by concatenating $Y^-$, the constant path $0\in S$ and $Y^+$, and shifted such that $X'_{s+1}=0$, so that $X_r=X'_r$ for all $r\geq s+1$. We can change the normalizing constant and rewrite
\begin{align*}
   P^\gamma_{\varsigma, t} \big((X_{s+r}-X_s)_{r\in [0,1]}\in {\rm d}z |Y^-, Y^+ \big) =  \frac{1}{Z_2} e^{-\alpha\gamma \int_s^t (v(r, X)- v(r, X')) \, {\rm d}r} \mu({\rm d}z).
\end{align*}
where we note that for $d>4$
\begin{align*}
\Big| \int_s^t (v(r, X)- v(r, X'))\, {\rm d}r \Big|  < \infty.
\end{align*}
Indeed,  with a similar reasoning as in \eqref{eq:tvdiff},  we have
\begin{align*}
\Big| \int_s^t (v(r, X)- v(r, X'))\, {\rm d}r \Big|
&\leq 2 + \int_{s+2}^\infty \big| v(r, X)- v(r, X')\big| \, {\rm d}r \\
 & =  2+\int_{s+2}^\infty \Big| \E^Y_{X_r}\left[e^{- \gamma\int_0^\infty \delta_0(Y_u- X_{r-u})\,{\rm d}u}   - e^{ - \gamma\int_0^\infty \delta_0(Y_u-X'_{r-u})\,{\rm d}u} \right] \Big|\, {\rm d}r \notag\\
    & \leq 2+ \gamma \int_{s+2}^\infty \int_{r-s-1}^\infty \big(\P^Y_{X_r}(Y_u=X_{r-u}) + \P^Y_{X_r}(Y_u=X'_{r-u})\big) \, {\rm d}u \, {\rm d}r \notag \\
    & \le 2+ \gamma \int_{s+2}^\infty \frac{C}{(r-s-1)^{d/2-1}}\, {\rm d}r  \leq  C(2+\gamma) < \infty,
    \notag
\end{align*}
with $Y$ denoting an independent random walk as in \eqref{eq:vX} (not to be confused with $Y^+$ and $Y^-$). This bound is uniform in $s<t$ and $X$, which implies \eqref{eq:rhosmu}. The analogous result for the marginal law of the components of $\overline{X}_n$ follows readily from the definition of $\Pi_h$ in \eqref{eq:hTrans}.
\end{proof}

Lemma \ref{L:abscon} suggests that discarding the part of the path $X$ in the time intervals $[0, L]$ and $[t-L, t]$ has a negligible effect in the diffusive scaling limit. It also shows that the fluctuations of the path between consecutive integer times are negligible. More precisely, we have the following corollary.

\begin{corollary}\label{C:approx}
Assume the same setting as in Theorem \ref{thm:fCLT}, and let $(X_s)_{s\in [0,t]}$ be sampled according to $P^\gamma_t(X\in \cdot)$. Let $L:=\lceil \log t\rceil$, and for $0\leq s\leq \lfloor t\rfloor -2L=: t_L$, define $Y_s:= X_{L+\lfloor s\rfloor} - X_L$. Then $(Y_{st_L}/\sqrt{t_L})_{s\in [0,1]}$ and $(X_{st}/\sqrt{t})_{s\in [0,1]}$ converge in distribution to the same limit as $t\to\infty$, if the weak limit exists.
\end{corollary}
\begin{proof} Since $\lim_{t\to\infty} t_L/t = 1$, changing the space time normalizing factors from $\sqrt{t_L}$ and $t_L$ to $\sqrt{t}$ and $t$ has no effect asymptotically. To show that discarding $(X_s)_{s\in [0, L]}$ and $(X_s)_{s\in [t-L, t]}$ and replacing $X_{L+s}-X_L$ by $X_{L+\lfloor s\rfloor}-X_L$ has no effect on the diffusive scaling limit, we start with observing that by Lemma \ref{L:abscon}, for $s \in [0,t]$
\begin{align*}
&P^\gamma_t \Big(\sup_{r\in [0, 1]} \Vert X_{s+r}-X_s\Vert_\infty \geq t^{1/4}\Big)
\quad \leq D \mu\Big(x\in S: \sup_{r\in [0, 1]} \Vert x(r)\Vert_\infty \geq t^{1/4}\Big)
\end{align*}
    We will use  a large deviation bound on the number of random walk jumps in a unit time interval to estimate the right hand side above. In order for the event
$$
\left \{x \in S: \sup_{r\in [0, 1]} \Vert x(r)\Vert_\infty \geq t^{1/4} \right \}
$$
to occur, the number of jumps of the continuous time random walk $x$ in a unit interval  under $\mu$ has to be at least $t^{1/4}.$ This number is Poisson distributed with parameter $1$ under $\mu$. Consequently, tail bounds for Poisson random variables, implies that
$$D \mu\Big(x\in S: \sup_{r\in [0, 1]} \Vert x(r)\Vert_\infty \geq t^{1/4}\Big) \leq e^{- c t^{1/4}\log t}.$$
 Finaly, a union bound over all such events with integer $s\in [0, t]$ then implies that if their limits exist, the finite dimensional distributions of $(Y_{st_L}/\sqrt{t_L})_{s\in [0,1]}$ and $(X_{st}/\sqrt{t})_{s\in [0,1]}$ converge to the same limit. In combination with Skorokhod's theorem, this establishes Corollary \ref{C:approx}.
\end{proof}

We will now state two lemmas (Lemma \ref{L:approx2} and Lemma \ref{L:inv}) and Proposition \ref{P:posgamma} from which Theorem \ref{thm:fCLT} will follow.  Assuming these we shall prove Theorem \ref{thm:fCLT} and subsequently prove them in Section \ref{sec:twolem}.

Our first lemma states that the law of $Y^{(t)}:=(Y_s)_{s\in [0, t_L]}$ introduced in Corollary \ref{C:approx}, or rather, its increments, is close in distribution to the stationary process associated with the $\Sigma$-valued Markov chain with transition kernel $\Pi_h$ and stationary distribution $\nu_h$.

\begin{lemma}\label{L:approx2} Assume $d\geq 5$ and
let $\eta:=(\eta_i)_{i\in\Z}$ be a stationary process such that for each $n\in\Z$, $(\eta_{n+i})_{i\leq -1}$ has distribution $\nu_h$. For $t>0$, let $X$, $L$, and $t_L$ be as in Corollary \ref{C:approx}. Define $\overline{Y}_+^{(t)} :=(y_i)_{0\leq i\leq t_L-1}$ such that $y_i=(X_{L+i+s}-X_{L+i})_{s\in [0,1]}\in S$, and similarly denote $\eta^{(t)}=(\eta_{i})_{0\leq i\leq t_L-1}$.
Then
\begin{equation}\label{eq:TV}
\d_{\rm TV} \big( {\mathcal L}(\eta^{(t)}), {\mathcal L}(\overline{Y}_+^{(t)}) \big ) \to 0 \qquad \mbox{as} \quad t\to\infty,
\end{equation}
where $\mathcal L(\cdot)$ denotes the law of a random variable, and $\d_{\rm TV}$ denotes the total variation distance.
\end{lemma}

Taking advantage of Corollary \ref{C:invariance}, the next lemma deduces that this stationary process satisfies an invariance principle.

\begin{lemma}\label{L:inv}
Assume $d\geq 6$ and let $\eta$ and $t_L$ be as in Lemma \ref{L:approx2}. For $s\in [0, t_L]$, define $Z_s:= \sum_{i=0}^{\lfloor s\rfloor-1} \eta_i(1)$. Then $(Z_{st_L}/\sqrt{t_L})_{s\in [0, 1]}$ converges in distribution to $(\sigma B_s)_{s\in [0, 1]}$, as $t \rightarrow \infty$, where $B$ is a standard Brownian motion in $\R^d$ and $\sigma\in [0, \infty)$ is a deterministic constant.
\end{lemma}

Our last required auxillary  result for the proof of Theorem \ref{thm:fCLT} establishes that the diffusion coefficient $\sigma^2$ in Lemma \ref{L:inv} is strictly positive when $\gamma>0$ is sufficiently small.

\begin{proposition} \label{P:posgamma}
Under the same assumptions as in Theorem \ref{thm:fCLT}, if $\gamma >0$ is sufficiently small, then $\sigma^2 \in (0,\infty).$
\end{proposition}
This result implies that the random walk conditioned on annealed survival among the moving traps is indeed diffusive. We conjecture that the same is true for any $\gamma>0$; however, it is currently not clear how to verify this.

\begin{proof}[Proof of Theorem \ref{thm:fCLT}]
By Lemma \ref{L:inv}, $(Z_{st_L}/\sqrt{t_L})_{s\in [0, 1]}$ converges in distribution to $(\sigma B_s)_{s\in [0, 1]}$ as $t \rightarrow \infty$, where $B$ is a standard Brownian motion in $\R^d$ and $\sigma\in [0, \infty)$ is a deterministic constant. By Lemma \ref{L:approx2},
   \begin{equation*}
\d_{\rm TV}( {\mathcal L}(\eta^{(t)}), {\mathcal L}(\overline{Y}_+^{(t)})) \to 0 \qquad \mbox{as} \quad t\to\infty.
\end{equation*}
This implies  $(Y_{st_L}/\sqrt{t_L})_{s\in [0, 1]}$ converges in distribution to the same limit as $(Z_{st_L}/\sqrt{t_L})_{s\in [0, 1]}$ as $t \rightarrow \infty$.  By Corollary \ref{C:approx}, we have that  $(X_{st}/\sqrt{t})_{s\in [0, 1]}$ converges in distribution to the same limit as $(Z_{st_L}/\sqrt{t_L})_{s\in [0, 1]}$ as $t \rightarrow \infty$.

In conclusion,  under $P_{t}^\gamma(X\in \cdot)$, the diffusively rescaled path
$(X_{st}/\sqrt{t})_{s\in [0, 1]}$ converges in distribution to $(\sigma B_s)_{s\in [0, 1]}$ as $t \rightarrow \infty$, where by Proposition \ref{P:posgamma}, $\sigma \in (0,\infty)$ for $\gamma >0$ sufficiently small.
\end{proof}

\subsection{Proof of Lemmas \ref{L:approx2} and \ref{L:inv} and Proposition \ref{P:posgamma}}\label{sec:twolem}

To prove Lemma \ref{L:approx2}, we first prove two preliminary results on the continuity of the path measure $P^\gamma_{X, t}((X_s)_{s\in [0, t]}\in \cdot)$ (in total variation distance) with respect to  the past $(X_s)_{s<0}$, where the path measure $P^\gamma_{X, t}$ was defined in \eqref{eq:GibbsRewrite} and \eqref{eq:vX2}.

\begin{lemma}\label{L:PXTV}
Let $(X^1_s)_{s<0}$ and $(X^2_s)_{s<0}$ be two c\`adl\`ag paths in $\Z^d\cup\{*\}$ such that $X^1_s= X^2_s$ for all $s\in (-T, 0)$. Then uniformly in $X^1$, $X^2$, $T\geq 1$, and $t>0$, we have
\begin{equation}\label{eq:PXTV}
    \d_{\rm TV}\big( P^\gamma_{X^1, t}((X_s)_{s\in [0,t]} \in \cdot), P^\gamma_{X^2, t}((X_s)_{s\in [0,t]} \in \cdot)\big) \leq \frac{C}{T^{d/2-2}}.
\end{equation}
\end{lemma}
\begin{proof}
Recall from \eqref{eq:GibbsRewrite} and \eqref{eq:vX2} that for $i=1, 2$,
$$
P^\gamma_{X^i, t}({\rm d}X) = \frac{1}{Z^\gamma_{t, X^i}} e^{-\alpha\gamma \int_0^t v(s, \widetilde X^i) \, {\rm d}s} \, \P_0({\rm d}X),
$$
where $\P_0({\rm d}X)$ denotes the law of a simple symmetric random walk $(X_s)_{s\in [0, t]}$ with $X_0=0$, and $\widetilde X^i$ is the concatenation of $X^i$ and $X$ at time $0$. As a consequence, we obtain that
\begin{align*}
    \frac{{\rm d} P^\gamma_{X^2, t}}{{\rm d} P^\gamma_{X^1, t}}(X) = \frac{Z^\gamma_{t, X^1}}{Z^\gamma_{t, X^2}} e^{-\alpha\gamma \int_0^t (v(s, \widetilde X^2) - v(s, \widetilde X^1)) \,  {\rm d}s}.
\end{align*}
Applying \eqref{eq:tvdiff}, the exponent in the previous display is uniformly bounded from above by
$$
\alpha \gamma \int_0^t \frac{C}{(T+s)^{d/2-1}} \, {\rm d}s \leq \frac{C}{T^{d/2-2}}.
$$
Therefore, $Z^\gamma_{t, X^1}/Z^\gamma_{t, X^2}$ and its reciprocal are also uniformly bounded by $1+C/T^{d/2-2}$. The total variation bound in \eqref{eq:PXTV} then follows easily.
\end{proof}

The next lemma is an analogue of Lemma \ref{L:PXTV} formulated in terms of the path increments taking values in $(\Sigma, \d)$, where the distance between two pasts $\overline{x}, \overline{y}\in \Sigma$ is measured through the metric $\d$ defined in and after \eqref{eq:pathd}. Recall the measure $\overline{P}^\gamma_{\overline{w}, t}({\rm d}\overline{z})$ defined after \eqref{tildePgamma} for some $\overline{w}\in \Sigma$, and recall the transition kernel $\Pi_h$ of the $\Sigma$-valued Markov chain introduced in Remark \ref{R:Pih}. For $t\in \N$, let $\pi_{-t, 0}: \Sigma \to S^t$ be the coordinate projection map defined by $\pi_t \overline{x} = (x_i)_{-t\leq i\leq -1}$.

\begin{lemma}\label{L:PnuTV}
Let $d \ge 5$ and
let $(\overline{w}^{(n)})_{n\in\N}$ be a sequence of elements in $(\Sigma, \d)$ such that $\overline{w}^{(n)}\to \overline{w}$. Then
\begin{equation}\label{eq:PnuTV}
    \sup_{t\in\N} \, \d_{\rm TV}\big( \overline{P}^\gamma_{\overline{w}^{(n)}, t}\circ \pi_{-t, 0}^{-1}, \, \overline{P}^\gamma_{\overline{w}, t}\circ \pi_{-t, 0}^{-1}\big) \to 0 \quad \mbox{as} \quad n\to\infty,
\end{equation}
and
\begin{equation}\label{eq:PnuTV'}
    \sup_{t\in\N} \, \d_{\rm TV}\big(\Pi_h^t(\overline{w}^{(n)}, {\rm d}\overline{y})\circ \pi_{-t, 0}^{-1}, \, \Pi_h^t(\overline{w}, {\rm d}\overline{y})\circ \pi_{-t, 0}^{-1}\big) \to 0 \quad \mbox{as} \quad n\to\infty.
\end{equation}
The same conclusions hold if we replace $\overline{P}^\gamma_{\overline{w}^{(n)}, t}$, $\overline{P}^\gamma_{\overline{w}, t}$,  $\Pi_h^t(\overline{w}^{(n)}, \cdot)$ and $\Pi_h^t(\overline{w}, \cdot)$ respectively by the mixtures $\int_\Sigma \overline{P}^\gamma_{\overline{u}, t} \, \varsigma^{(n)}({\rm d}\overline{u})$, $\int_\Sigma \overline{P}^\gamma_{\overline{u}, t} \, \varsigma({\rm d}\overline{u})$, $\int_\Sigma \Pi_h^t(\overline{u}, \cdot) \, \varsigma^{(n)}({\rm d}\overline{u})$, and $\int_\Sigma \Pi_h^t(\overline{u}, \cdot) \, \varsigma({\rm d}\overline{u})$, where $(\varsigma^{(n)})_{n\in\N}$ and $\varsigma$ are probability measures on $(\Sigma, \d)$ with $\varsigma^{(n)}\Rightarrow \varsigma$.
\end{lemma}
\begin{proof}
For $t\in\N$, by the definition in \eqref{tildePgamma}, we have
$$
(\overline{P}^\gamma_{\overline{w}, t}\circ \pi_{-t, 0}^{-1}) ({\rm d}z_1 \ldots {\rm d}z_t) = \frac{1}{Z^\gamma_{\overline{w}, t}} e^{\sum_{i=1}^t \varphi(\overline{w}z_1\ldots z_i)} \prod_{i=1}^t \mu({\rm d}z_i),
$$
where $Z^\gamma_{\overline{w}, t}$ is the normalizing constant. Therefore,
\begin{align*}
    \frac{{\rm d} \overline{P}^\gamma_{\overline{w}^{(n)}, t}\circ \pi_{-t, 0}^{-1}}{{\rm d} \overline{P}^\gamma_{\overline{w}, t}\circ \pi_{-t, 0}^{-1}}(z_1, \ldots, z_t) = \frac{Z^\gamma_{\overline{w}, t}}{Z^\gamma_{\overline{w}^{(n)}, t}} e^{\sum_{i=1}^t (\varphi(\overline{w}^{(n)}z_1\ldots z_i) -\varphi(\overline{w}z_1\ldots z_i) )}.
\end{align*}
Uniformly in $t\in\N$ and $z_1, z_2\ldots \in S$, the exponent above can be bounded by
\begin{align} \label{eq:Pwnconv}
\begin{split}
    &\Big|\sum_{i=1}^t (\varphi(\overline{w}^{(n)}z_1\ldots z_i) -\varphi(\overline{w}z_1\ldots z_i) )\Big| \\
    &\quad
    \leq \sum_{i=1}^\infty \sup_{z_1, z_2\ldots \in S} |\varphi(\overline{w}^{(n)}z_1\ldots z_i) -\varphi(\overline{w}z_1\ldots z_i)|.
\end{split}
\end{align}
Now the $i$-th summand on the right-hand side of the previous display is bounded by $\var_i(\varphi)$. Due to $d\ge 5$ by assumption, and using Proposition \ref{prop:sumVar}, we have $\sum_{i\in\N} \var_i(\varphi)<\infty$, and each summand in the right-hand side of the previous display converges to $0$ as $n\to\infty$ by the uniform continuity of $\varphi$ proved in Proposition \ref{prop:sumVar}. Therefore, the right-hand side of \eqref{eq:Pwnconv} converges to $0$ as $n\to\infty$. The uniform convergence in total variation distance in \eqref{eq:PnuTV} then follows readily.

The proof of \eqref{eq:PnuTV'} is similar. Since $\Pi_h\big(\overline{x}, {\rm d}\overline{y}\big) = 1_{\{\overline{y} = \overline{x} z\}} e^{\varphi(\overline{x}z)}\frac{h(\overline{x}z)}{\lambda h(\overline{x})} \mu({\rm d}z)$, we can write
\begin{align*}
\big(\Pi_h^t(\overline{w}, \cdot )\circ \pi_{-t, 0}^{-1}\big)({\rm d}z_1\ldots {\rm d}z_t) = e^{\sum_{i=1}^t \varphi(\overline{w}z_1\ldots z_i)} \frac{h(\overline{w}z_1\ldots z_t)}{\lambda^t h(\overline{w})} \prod_{i=1}^t \mu({\rm d}z_i),
\end{align*}
and
\begin{align*}
\frac{{\rm d} \Pi_h^t(\overline{w}^{(n)}, \cdot)\circ \pi_{-t, 0}^{-1}}{{\rm d} \Pi_h^t(\overline{w}, \cdot)\circ \pi_{-t, 0}^{-1}} (z_1, \ldots, z_t) = \frac{h(\overline{w})}{h(\overline{w}^{(n)})}\cdot \frac{h(\overline{w}^{(n)}z_1\ldots z_t)}{h(\overline{w}z_1\ldots z_t)} e^{\sum_{i=1}^t (\varphi(\overline{w}^{(n)}z_1\ldots z_i) -\varphi(\overline{w}z_1\ldots z_i) )},
\end{align*}
where the exponent converges uniformly to $0$ as $n\to\infty$ by \eqref{eq:Pwnconv}, and the quotients of $h(\cdot)$ converge uniformly to $1$ by the uniform continuity of $h$ established in Theorem \ref{thm:nuex}. The uniform convergence in \eqref{eq:PnuTV'} is then immediate.

When deterministic $\overline{w}^{(n)}\to \overline{w}$ are replaced by probability mixtures with laws $\varsigma^{(n)}$ and $\varsigma$, by Skorohod's representation theorem, we can construct coupled random variables $\overline{w}^{(n)}$ and $\overline{w}$ with laws $\varsigma^{(n)}$ and $\varsigma$, respectively, such that $\overline{w}^{(n)}\to \overline{w}$ almost surely. The desired result then follows by applying
\eqref{eq:PnuTV} and \eqref{eq:PnuTV'} to the almost sure realization of $\overline{w}^{(n)}\to \overline{w}$ and then taking expectation.
\end{proof}

We are now ready to prove Lemma \ref{L:approx2}.

\begin{proof}[Proof of Lemma \ref{L:approx2}]
The basic strategy is to reformulate the path measure $P^\gamma_t$, or rather the law of the path increments, in terms of the $\Sigma$-valued Markov chain with transition kernel $\Pi_h$ and use its ergodicity stated in Remark \ref{R:Pih}. But we first need to approximate $P^\gamma_t$, which has an empty past, by $P^\gamma_{X, t}$ with an infinite past so that the increments of the past $(X_s)_{s<0}$ belong to $\Sigma$. We proceed as follows.

For fixed $T>0$ and for $t>T$, sample $(X_s)_{s\in [0, T]}$ according to $P^\gamma_t((X_s)_{s\in [0, T]}\in \cdot)$, and let $\mu_{t, T}$ denote the law of the space-time shifted path $\widetilde X^{(T)}:=(\widetilde X^{(T)}_s)_{s\in [-T, 0]} := (X_{T+s}-X_T)_{s\in [-T, 0]}$. We note that by the Gibbs property of $P^\gamma_t$, conditioned on $(X_s)_{s\in [0, T]}$ and its space time shift $\widetilde X^{(T)}$, the law of $(X_{T+s}-X_T)_{s\in [0, t-T]}$ is given exactly by $P^\gamma_{\widetilde X^{(T)}, t-T}$. We can extend $\widetilde X^{(T)}$ backwards in time by setting $\widetilde X_s= \widetilde X^{(T)}_s$ for $s\in [-T, 0]$ and setting $\widetilde X_s=\widetilde X^{(T)}_{-T}$ for $s<-T$. We will denote the law of $\widetilde X$ also by $\mu_{t, T}$. Then Lemma \ref{L:PXTV} implies that uniformly in $T\geq 1$, $t>T$, and $\widetilde X^{(T)}$, we have
\begin{equation}\label{eq:PXTV2}
    \d_{\rm TV}\big( P^\gamma_{\widetilde X^{(T)}, t-T}((X_s)_{s\in [0,t-T]} \in \cdot), P^\gamma_{\widetilde X, t-T}((X_s)_{s\in [0,t-T]} \in \cdot)\big) \leq \frac{C}{T^{d/2-2}}.
\end{equation}
This shows that attaching an infinite past creates an error that can be made arbitrarily small in total variation distance by choosing $T$ sufficiently large. It then suffices to verify Lemma \ref{L:approx2} for each fixed $T>0$, where the past $\widetilde X$ in $P^\gamma_{\widetilde X, t-T}$ has law $\mu_{t, T}$. To apply the ergodicity of the $\Sigma$-valued Markov chain with transition kernel $\Pi_h$, we need to replace $\mu_{t, T}$ by something that does not depend on $t$ as we take the limit $t\to\infty$. This is achieved by
observing that, using the same argument as that for Lemma \ref{L:abscon}, we note that $\mu_{t, T}$ is absolutely continuous with respect to the law of the simple symmetric random walk $(X_s)_{s\in [-T, 0]}$ with $X_0=0$, such that the density is uniformly bounded and the bound depends only on $T$ and not on $t$. To verify Lemma \ref{L:approx2}, it then suffices to show that the convergence therein holds with $P^\gamma_t$ replaced by $P^\gamma_{X, t-T}$ for a typical history $X:[-T, 0]\to \Z^d$ sampled according to the law of the simple symmetric random walk (with $X_0=0$ and $X_{s}=X_{-T}$ for $s<-T$). Averaging over the law $\mu_{t, T}$ for the history $X$, which has a uniformly bounded density w.r.t.\ the law of the simple symmetric random walk, then implies Lemma \ref{L:approx2}. We will fix an arbitrary history $(X_s)_{s\leq 0}$ from now on. As in Lemma \ref{L:approx2}, we reformulate the problem in terms of the law of path increments as follows.

First we assume $t\in \N$. Otherwise we adjust $T$ slightly such that $t-T\in \N$. For $\overline{x}\in \Sigma$, recall that similar to \eqref{tildePgamma2} and \eqref{statred1},
\begin{equation}\label{eq:stared2}
\overline{P}^\gamma_{\overline{x}, t}({\rm d}\overline{y}) := \frac{(\delta_{\overline{x}} \Pi^t)({\rm d}\overline{y})}{\langle \delta_{\overline{x}}, \Pi^t1\rangle} = \frac{h(\overline{x}) \Pi_h^t(\overline{x}, {\rm d} \overline{y})\frac{1}{h(\overline{y})}}{
h(\overline{x}) (\Pi_h^t \frac{1}{h})(\overline{x})} = \frac{\Pi_h^t(\overline{x}, {\rm d} \overline{y})\frac{1}{h(\overline{y})}}{(\Pi_h^t \frac{1}{h})(\overline{x})},
\end{equation}
where $\overline{P}^\gamma_{\overline{x}, t}({\rm d}\overline{y})$ concentrates on $\overline{y}\in \Sigma$ with $(y_{-t+i})_{i\leq -1} =\overline{x}$. For $a, b\in \Z$ with $a<b\leq 0$, if $\pi_{a, b}: \Sigma \to S^{b-a}$ denotes the projection $\pi_{a, b}\overline{x}= (x_i)_{a\leq i\leq b-1}$, then $\overline{P}^\gamma_{\overline{x}, t}\circ \pi_{-t+L, -L}^{-1}$ is exactly the law of the path increments in the time interval $[L, t-L]$ under $P^\gamma_{X, t}((X_s)_{s\in [0,t]}\in \cdot)$, where the past $(X_s)_{s<0}$ is the c\`adl\`ag path constructed from $\overline{x}$ with $X_0=0$. To conclude the proof of Lemma \ref{L:approx2}, it only remains to show that
\begin{equation}\label{eq:TV2}
\d_{\rm TV}( \overline{P}^\gamma_{\overline{x}, t}\circ \pi_{-t+L, -L}^{-1}, \ \nu_h \circ \pi_{-t+L, -L}^{-1}) \to 0 \qquad \mbox{as} \quad t\to\infty.
\end{equation}

Using $\Pi_h^t = \Pi_h^L \Pi_h^{t-L}$ to perform a marginal-conditional distribution decomposition in \eqref{eq:stared2}, we can rewrite
\begin{align}
\overline{P}^\gamma_{\overline{x}, t}({\rm d}\overline{y})  & = \frac{\int_{\overline{w}\in \Sigma}\Pi_h^L(\overline{x}, {\rm d}\overline{w}) \Pi_h^{t-L}(\overline{w}, {\rm d} \overline{y})\frac{1}{h(\overline{y})}}{(\Pi_h^t \frac{1}{h})(\overline{x})} \notag \\
& =  \int_{\overline{w}\in \Sigma} \frac{\Pi_h^L(\overline{x}, {\rm d}\overline{w})(\Pi_h^{t-L} \frac{1}{h})(\overline{w})}{(\Pi_h^t \frac{1}{h})(\overline{x})} \cdot \frac{\Pi_h^{t-L}(\overline{w}, {\rm d} \overline{y})\frac{1}{h(\overline{y})}}{(\Pi_h^{t-L} \frac{1}{h})(\overline{w})} \notag \\
& = \int_{\overline{w}\in \Sigma} \nu_{t, L, \overline{x}}({\rm d}\overline{w}) \overline{P}^\gamma_{\overline{w}, t-L}({\rm d}\overline{y}), \label{stared3}
\end{align}
where we note that
\begin{align*}
\nu_{t, L, \overline{x}}({\rm d}\overline{w}) := \int_{\overline{w}\in \Sigma} \frac{\Pi_h^L(\overline{x}, {\rm d}\overline{w})(\Pi_h^{t-L} \frac{1}{h})(\overline{w})}{(\Pi_h^t \frac{1}{h})(\overline{x})}
\ \ \Rightarrow\ \ \nu_h \quad \mbox{as} \ \ t\to\infty.
\end{align*}
The weak convergence holds because $\Pi_h^L(\overline{x}, {\rm d}\overline{w}) \Rightarrow \nu_h({\rm d}\overline{w})$ as explained in Remark \ref{R:Pih}, which allows us to couple $\Sigma$-valued random variables $\overline{w}^{(L)}$ and $\overline{w}$ with law
$\Pi_h^L(\overline{x}, {\rm d}\overline{w})$ and $\nu_h({\rm d}\overline{w})$ respectively, such that $\overline{w}^{(L)}\to \overline{w}$ almost surely in $(\Sigma, \d)$. The claimed weak convergence then follows from the fact that $(\Pi_h^t \frac{1}{h})(\overline{x})\to 1$ and $(\Pi_h^{t-L} \frac{1}{h})(\overline{w})\to 1$ by \eqref{erg}, while the family of functions $(\Pi_h^m \frac{1}{h})_{m\in\N_0}$ are uniformly equicontinuous by Claim \ref{cl:uconlg}.

We can now apply Lemma \ref{L:PnuTV} to obtain
that for $t\to\infty,$
\begin{align}\label{stared4}
\d_{\rm TV}\Big(\int_{\overline{w}\in \Sigma} \nu_{t, L, \overline{x}}({\rm d}\overline{w}) \overline{P}^\gamma_{\overline{w}, t-L} \circ \pi_{-t_L, 0}^{-1}, \ \int_{\overline{w}\in \Sigma} \nu_h({\rm d}\overline{w}) \overline{P}^\gamma_{\overline{w}, t-L} \circ \pi_{-t_L, 0}^{-1}  \Big) \to 0,
\end{align}
where the first measure is exactly $\overline{P}^\gamma_{\overline{x}, t}\circ \pi_{-t+L, -L}^{-1}$ in \eqref{eq:TV2}, while the second measure can be written as
\begin{align*}
&\Big(\int_{\overline{w}\in \Sigma} \nu_h({\rm d}\overline{w}) \overline{P}^\gamma_{\overline{w}, t-L} \circ \pi_{-t_L, 0}^{-1}\Big)({\rm d}z_1\ldots {\rm d}z_{t_L})\\
&\quad =    \int_{\overline{w}\in \Sigma} \!\!\!\!\!\!\! \nu_h({\rm d}\overline{w})
    \frac{\Pi_h^{t_L}(\overline{w}, \overline{w}{\rm d}z_1\ldots {\rm d}z_{t_L})
    (\Pi_h^L\frac{1}{h})(\overline{w}z_1\ldots z_{t_L})}{(\Pi_h^{t-L} \frac{1}{h})(\overline{w})},
\end{align*}
where the denominator $(\Pi_h^{t-L} \frac{1}{h})(\overline{w})\to 1$
and hence can be removed with an asymptotically negligible difference in total variation distance, and the factor $(\Pi_h^L\frac{1}{h})(\overline{w}z_1\ldots z_{t_L})$ can be removed for exactly the same reason, while what is left is
\begin{align*}
 &\int_{\overline{w}\in \Sigma} \nu_h({\rm d}\overline{w}) \Pi_h^{t_L}(\overline{w}, \overline{w}{\rm d}z_1\ldots {\rm d}z_{t_L}) \\
 &\quad = (\nu_h \circ \pi_{-t_L, 0}^{-1})({\rm d}z_1\ldots {\rm d}z_{t_L}) = (\nu_h \circ \pi_{-t+L, -L}^{-1})({\rm d}z_1\ldots {\rm d}z_{t_L}).
\end{align*}
Substituting this measure for the second measure in \eqref{stared4} then gives \eqref{eq:TV2}, which concludes the proof of Lemma \ref{L:approx2}.
\end{proof}
\medskip

\begin{proof}[Proof of Lemma \ref{L:inv}]
For $b\in \R^d$, let us define $f_b:S\to \R$ with $f_b(z)=\langle b, z(1)\rangle$ for $z=(z(s))_{s\in [0,1]}\in S$, and define $\widehat f_b: \Sigma \to \R$ with $\widehat f_b(\overline{x}) = f_b(x_{-1})$ for $\overline{x}\in \Sigma$. By Corollary \ref{C:invariance}, we can deduce an invariance principle for $Y_n:= \sum_{i=1}^n \widehat f_b((\eta_{i+j})_{j\leq -1})$ once we show that $\bar f_b := \Pi_h \widehat f_b \in C_b(\Sigma)$ and $\var_n(\bar f_b)=O(n^{-r+1})$ for suitable choices of $r>1$. Note that
\begin{align*}
|\bar f_b(\overline{x})| &= |(\Pi_h \widehat f_b)(\overline{x})| =
\Big|\int_S \langle b, z(1)\rangle \frac{h(\overline{x}z)}{\lambda h(\overline{x})} e^{\varphi(\overline{x}z)} \, \mu({\rm d}z)\Big|\\
&\leq D |b|_\infty \sum_{i=1}^d \int_S |\langle e_i, z(1)\rangle| \,
\mu({\rm d}z) <\infty,
\end{align*}
where $e_i\in \Z^d$ denotes the unit vector in the $i$-th coordinate direction, and we used the facts that $D:=\sup_{\overline{x}\in \Sigma, z\in S}\frac{h(\overline{x}z)}{\lambda h(\overline{x})} e^{\varphi(\overline{x}z)}<\infty$ and the simple random walk has integrable displacements. The continuity of $\bar f_b$ follows from dominated convergence.

To bound $\var_n(\bar f_b)= \var_n(\Pi_h \widehat f_b)$, the same calculation as in \eqref{eqna:i1i2}
shows that for any $\overline{u}, \overline{v}\in \Sigma$ and $w_1, \ldots, w_n\in S$, we have
\begin{align*}
 &|\bar f_b(\overline{u}w_1\ldots w_n)- \bar f_b(\overline{v}w_1\ldots w_n)| \\
  &\quad \leq \int_S \big | G_1(\overline{u} w_1 \ldots w_nz) -   G_1(\overline{v} w_1 \ldots w_nz)\big | \, |f_b(z)| \,\mu({\rm d}z) \\
  &\quad  \leq  \sup \Big| \frac{G_1(\overline{u} w_1 \ldots w_nz)}{G_1(\overline{v}w_1 \ldots w_nz)} -1 \Big|
  \int_S  |f_b(z)| G_1(\overline{v}w_1 \ldots w_nz) \, \mu({\rm d}z),
\end{align*}
where $\Pi_h(\overline{u}, {\rm d}\overline{v})=1_{\{\overline{v} =\overline{u}z\}} G_1(\overline{u}z) \mu({\rm d}z)$. The integral above equals $\Pi_h |\widehat f_b|(\overline{v}w_1\ldots w_n)$ and is uniformly bounded by the same reasoning as for $\Pi_h \widehat f_b \in C_b(\Sigma)$. The supremum can be bounded by following the calculations in \eqref{eq:GQuotDist} and \eqref{eq:5varphi} to obtain
$$
\sup_{\overline{u}, \overline{v}, w_1, \ldots, w_n, z} \Big| \frac{G_1(\overline{x} w_1 \ldots w_nz)}{G_1(\overline{y}w_1 \ldots w_nz)} -1 \Big| \leq C \var_n(\varphi).
$$
Therefore
\begin{equation} \label{eq:d6}
\var_n(\bar f_b) \leq C \var_n(\varphi) \leq \frac{C}{n^r}
\end{equation}
with $r=\frac{d}{2}-1$ by Proposition \ref{prop:sumVar}. The assumption $d\geq 6$ ensures that the condition $r>3/2$ in Corollary \ref{C:invariance} is satisfied, while the condition $\int_\Sigma \widehat f_b(\overline{x}) \, \nu_h({\rm d} \overline{x})=0$ in Corollary \ref{C:invariance} follows from the symmetry of $\nu_h$ inherited from the simple symmetric random walk. This entails that $Y_n$ satisfies an invariance principle with a deterministic diffusion coefficient.

For $b\in \R^d$, note that $Y_n=\langle b, Z_n\rangle$ where $Z_n:=\sum_{i=1}^n \eta_{i-1}(1)$. The invariance principle for each component process $\langle e_i, Z_n\rangle$ implies that the $\R^d$-valued process $(Z_{st_L}/\sqrt{t_L})_{s\in [0, 1]}$
is tight as random variables in the Skorohod space $D([0,1], \R^d)$.
Furthermore, for each $r<s$, the invariance principle for $\langle b, Z_n\rangle$, $b\in \R^d$, and the Cram\'er-Wold device implies that $(Z_{st_L}-Z_{rt_L})/\sqrt{t_L}$ converges in distribution to Gaussian vector, whose covariance matrix is linear in $s-r$ and is a multiple of the identity matrix by permutation and reflection symmetry. To conclude that $(Z_{st_L}/\sqrt{t_L})_{s\in [0, 1]}$ converges in distribution to a $d$-dimensional Brownian motion, it only remains to show that the limit has independent increments. This follows from the mixing property of $\eta$. More precisely, assuming $r=0$ for simplicity, for $\nu_h$-a.e.\ $\overline{\eta}=(\eta_i)_{i\leq -1}$, the conditional law of $(\eta_i)_{L\leq i < s t_L}$ given  $\overline{\eta}$ can be written as
\begin{align*}
{\mathcal L}((\eta_i)_{L\leq i < s t_L} | \overline{\eta})
&=  \Pi_h^{st_L}(\overline{\eta}, {\rm d}\overline{y})\circ \pi_{-st_L+L, 0}^{-1} \\
&= \int_{\overline{w}\in \Sigma}\Pi_h^L(\overline{\eta}, {\rm d}\overline{w}) \Pi_h^{st_L-L}(\overline{w}, {\rm d}\overline{y})\circ \pi_{-st_L+L, 0}^{-1},
\end{align*}
where $\pi_{a,b}^{-1}: \Sigma \to S^{b-a}$ is the projection map defined before Lemma \ref{L:PnuTV}, and $\Pi_h^L(\overline{\eta}, {\rm d}\overline{w}) \Rightarrow \nu_h({\rm d}\overline{w})$ by Remark \ref{R:Pih}. We can now apply Lemma \ref{L:PnuTV} to replace $\Pi_h^L(\overline{\eta}, {\rm d}\overline{w})$ by
$\nu_h({\rm d}\overline{w})$ and use $\int_{\overline{w}\in \Sigma} \nu_h({\rm d}\overline{w}) \Pi_h^{st_L-L}(\overline{w}, {\rm d}\overline{y}) = \nu_h({\rm d}\overline{y})$ to obtain
$$
\d_{\rm TV} \big( {\mathcal L}((\eta_i)_{L\leq i < st_L} | \overline{\eta}), \, {\mathcal L}((\eta_i)_{L\leq i < s t_L}) \big) \to 0 \qquad \mbox{as} \quad t\to\infty.
$$
Therefore, the law of $(\eta_i)_{L\leq i < s t_L}$ becomes asymptotically independent of the history $\overline{\eta}=(\eta_i)_{i\leq -1}$. Together with the fact that $(\eta_i)_{0\leq i<L}$ constitutes a negligible contribution to $Z_{st_L}/\sqrt{t_L}$ by Corollary \ref{C:approx}, we deduce that the limit of $(Z_{st_L}/\sqrt{t_L})_{s\in [0, 1]}$ has independent increments, which concludes the proof of Lemma \ref{L:inv}.
\end{proof}

We conclude the article with the proof of Proposition \ref{P:posgamma}.

\begin{proof}[Proof of Proposition \ref{P:posgamma}]
It suffices to consider the process $Z_s:= \sum_{i=0}^{\lfloor s\rfloor-1} \eta_i(1)$ defined in
Lemma \ref{L:inv}, where $\eta=(\eta_i)_{i\in\Z}$ is the stationary process with $\overline{\eta}^{(n)} :=(\eta_{n+i})_{i\leq -1}$ having distribution $\nu_h$ for each $n\in\Z$. To make explicit the dependence on $\gamma,$ we write $\nu_h^\gamma$ instead.
Fix any $1\le \ell \le d$, let us define $\widehat f: \Sigma\to \R$ by $\widehat f(\overline{x}):= \langle e_\ell, x_{-1}(1)\rangle$, where $e_\ell\in \Z^d$ is the unit vector in the $\ell$-th coordinate direction. By symmetry, we have $\E_{\nu_h^\gamma}[\widehat f(\overline{\eta}^{(0)})]=0$. Since $\eta_{-1}(1)\in \Z^d$ has a non-trivial distribution, we have $\E_{\nu_h^\gamma}[\widehat f(\overline{\eta}^{(0)})^2]>0$. Furthermore, by the same argument as for Lemma \ref{L:abscon}, we have
$$
\lim_{\gamma \searrow 0}\E_{\nu_h^\gamma} [ \widehat f(\overline{\eta}^{(0)})^2] = \E_0^X[\langle e_\ell, X_1\rangle^2] \in (0,\infty),
$$
where $X=(X_s)_{s\geq 0}$ is the simple symmetric random walk on $\Z^d$ with $X_0=0$.

Let us denote $Z_n^{(\ell)}:= \langle e_\ell, Z_n\rangle = \sum_{i=0}^{n-1} \widehat f(\overline{\eta}^{(i)})$. By Corollary \ref{C:invariance}, the diffusion coefficient equals
\begin{align} \label{eq:varSum}
\sigma_\gamma^2 = \lim_{n\to\infty}\frac{1}{n} \E_{\nu_h^\gamma} \big[ (Z_n^{(\ell)})^2 \big] = \E_{\nu_h^\gamma} \big[ \widehat f(\overline{\eta}^{(0)})^2 \big] +  2 \sum_{k\in \N} \E_{\nu_h^\gamma} \big[ \widehat f(\overline{\eta}^{(0)}) \widehat f(\overline{\eta}^{(k)})\big].
\end{align}
Applying \eqref{eq:sigmabd}, where the constants $C_{\phi, \bar f}$ are uniformly bounded if $\phi$ is replaced by $\widetilde \gamma \phi$ with $\widetilde \gamma \in [0,1]$, we can bound
$$
\sum_{k\geq k_0} \big| \E_{\nu_h^\gamma} \big[ \widehat f(\overline{\eta}^{(0)}) \widehat f(\overline{\eta}^{(k)})\big] \big|
$$
uniformly in $\gamma\in [0,1]$, and the bound tends to $0$ as $k_0\to\infty$. On the other hand, by the same argument as for Lemma \ref{L:abscon}, for each $k\in\N$, we have
$$
\lim_{\gamma \searrow 0} \E_{\nu_h^\gamma} \big[ \widehat f(\overline{\eta}^{(0)}) \widehat f(\overline{\eta}^{(k)})\big]
=\E^X_0[\langle e_\ell, X_1-X_0\rangle \cdot \langle e_\ell, X_{k+1}-X_k \rangle] =0,
$$
since the simple symmetric random walk $X$ has independent and symmetric increments. By first restricting the sum in \eqref{eq:varSum} to $\sum_{1\leq k\leq k_0}$ for $k_0$ arbitrarily large and then sending $\gamma \searrow 0$, we obtain
\begin{align*}
\lim_{\gamma \searrow 0} \sigma_\gamma^2  =  \E_0^X[\langle e_\ell, X_1\rangle^2] \in (0,\infty),
\end{align*}
which implies that $\sigma_\gamma^2>0$ for $\gamma>0$ sufficiently small. This concludes the proof of Proposition \ref{P:posgamma}.
\end{proof}

\bigskip

\noindent{\textbf{Acknowledgement.}}
RS is supported by NUS Tier 1 grant A-8001448-00-00 and NSFC grant 12271475. SA is supported by the Knowledge Exchange grant at ICTS-TIFR. AD acknowledges support from DFG through the scientific network {\em Stochastic Processes on Evolving Networks.}
We thank Eric Endo, Zemer Kosloff, Omri Sarig, and Dalia Terhesiu for interesting discussions on the thermodynamic formalism.
We also thank the referees for reading the paper carefully and providing helpful comments.




\bibliographystyle{alphaurl}

\bibliography{poissonObstacles}

\begin{thebibliography}{MBOV13}

\bibitem[ADS17]{AtDrSu-16}
Siva Athreya, Alexander Drewitz, and Rongfeng Sun.
\newblock Subdiffusivity of a random walk among a {P}oisson system of moving
  traps on {$\mathbb{Z}$}.
\newblock {\em Math. Phys. Anal. Geom.}, 20(1):Art. 1, 22, 2017.
\newblock URL: \url{https://doi.org/10.1007/s11040-016-9227-8}, \href
  {http://dx.doi.org/10.1007/s11040-016-9227-8}
  {\path{doi:10.1007/s11040-016-9227-8}}.

\bibitem[ADS19]{AtDrSu-17}
Siva Athreya, Alexander Drewitz, and Rongfeng Sun.
\newblock Random walk among mobile/immobile traps: A short review.
\newblock In Vladas Sidoravicius, editor, {\em Sojourns in Probability Theory
  and Statistical Physics - III}, pages 1--22, Singapore, 2019. Springer
  Singapore.

\bibitem[BLS05]{BetzEtAl}
Volker Betz, J\'{o}zsef L\H{o}rinczi, and Herbert Spohn.
\newblock Gibbs measures on {B}rownian paths: theory and applications.
\newblock In {\em Interacting stochastic systems}, pages 75--102. Springer,
  Berlin, 2005.
\newblock URL: \url{https://doi.org/10.1007/3-540-27110-4_5}, \href
  {http://dx.doi.org/10.1007/3-540-27110-4_5}
  {\path{doi:10.1007/3-540-27110-4_5}}.

\bibitem[Bol82]{Bo-82}
E.~Bolthausen.
\newblock On the central limit theorem for stationary mixing random fields.
\newblock {\em Ann. Probab.}, 10(4):1047--1050, 1982.
\newblock URL:
  \url{http://links.jstor.org/sici?sici=0091-1798(198211)10:4<1047:OTCLTF>2.0.CO;2-E&origin=MSN}.

\bibitem[Bow08]{Bo-08}
Rufus Bowen.
\newblock {\em Equilibrium states and the ergodic theory of {A}nosov
  diffeomorphisms}, volume 470 of {\em Lecture Notes in Mathematics}.
\newblock Springer-Verlag, Berlin, revised edition, 2008.
\newblock With a preface by David Ruelle, Edited by Jean-Ren\'e Chazottes.

\bibitem[BP22]{betz2021functional}
Volker Betz and Steffen Polzer.
\newblock A functional central limit theorem for {P}olaron path measures.
\newblock {\em Comm. Pure Appl. Math.}, 75(11):2345--2392, 2022.

\bibitem[BS05]{BeSp-05}
Volker Betz and Herbert Spohn.
\newblock A central limit theorem for {G}ibbs measures relative to {B}rownian
  motion.
\newblock {\em Probab. Theory Related Fields}, 131(3):459--478, 2005.
\newblock URL: \url{https://doi.org/10.1007/s00440-004-0381-8}, \href
  {http://dx.doi.org/10.1007/s00440-004-0381-8}
  {\path{doi:10.1007/s00440-004-0381-8}}.

\bibitem[CS16]{CiSi-16}
Leandro Cioletti and Eduardo~A. Silva.
\newblock Spectral properties of the {R}uelle operator on the {W}alters class
  over compact spaces.
\newblock {\em Nonlinearity}, 29(8):2253--2278, 2016.
\newblock URL: \url{https://doi.org/10.1088/0951-7715/29/8/2253}, \href
  {http://dx.doi.org/10.1088/0951-7715/29/8/2253}
  {\path{doi:10.1088/0951-7715/29/8/2253}}.

\bibitem[CSS19]{CiSiSt-19}
L.~Cioletti, E.~Silva, and M.~Stadlbauer.
\newblock Thermodynamic formalism for topological {M}arkov chains on standard
  {B}orel spaces.
\newblock {\em Discrete Contin. Dyn. Syst.}, 39(11):6277--6298, 2019.
\newblock URL: \url{https://doi.org/10.3934/dcds.2019274}, \href
  {http://dx.doi.org/10.3934/dcds.2019274} {\path{doi:10.3934/dcds.2019274}}.

\bibitem[DFSX20]{DFSX20}
Jian Ding, Ryoki Fukushima, Rongfeng Sun, and Changji Xu.
\newblock Geometry of the random walk range conditioned on survival among
  {B}ernoulli obstacles.
\newblock {\em Probab. Theory Related Fields}, 177(1-2):91--145, 2020.
\newblock URL: \url{https://doi.org/10.1007/s00440-019-00943-z}, \href
  {http://dx.doi.org/10.1007/s00440-019-00943-z}
  {\path{doi:10.1007/s00440-019-00943-z}}.

\bibitem[DFSX21]{DFSX21}
Jian Ding, Ryoki Fukushima, Rongfeng Sun, and Changji Xu.
\newblock Distribution of the random walk conditioned on survival among
  quenched {B}ernoulli obstacles.
\newblock {\em Ann. Probab.}, 49(1):206--243, 2021.
\newblock URL: \url{https://doi.org/10.1214/20-AOP1450}, \href
  {http://dx.doi.org/10.1214/20-AOP1450} {\path{doi:10.1214/20-AOP1450}}.

\bibitem[DGRS12]{DrGaRaSu-10}
Alexander Drewitz, J\"urgen G\"artner, Alejandro~F. Ram{\'\i}rez, and Rongfeng
  Sun.
\newblock Survival probability of a random walk among a {P}oisson system of
  moving traps.
\newblock In {\em Probability in complex physical systems}, volume~11 of {\em
  Springer Proc. Math.}, pages 119--158. Springer, Heidelberg, 2012.
\newblock URL: \url{http://dx.doi.org/10.1007/978-3-642-23811-6_6}, \href
  {http://dx.doi.org/10.1007/978-3-642-23811-6_6}
  {\path{doi:10.1007/978-3-642-23811-6_6}}.

\bibitem[dHW94]{HoWe-94}
Frank den Hollander and George~H. Weiss.
\newblock {\em 4. Aspects of Trapping in Transport Processes}, chapter~4, pages
  147--203.
\newblock 1994.
\newblock URL:
  \url{https://epubs.siam.org/doi/abs/10.1137/1.9781611971552.ch4}, \href
  {http://arxiv.org/abs/https://epubs.siam.org/doi/pdf/10.1137/1.9781611971552.ch4}
  {\path{arXiv:https://epubs.siam.org/doi/pdf/10.1137/1.9781611971552.ch4}},
  \href {http://dx.doi.org/10.1137/1.9781611971552.ch4}
  {\path{doi:10.1137/1.9781611971552.ch4}}.

\bibitem[DS58]{DuSc-58}
Nelson Dunford and Jacob~T. Schwartz.
\newblock {\em Linear {O}perators. {I}. {G}eneral {T}heory}.
\newblock With the assistance of W. G. Bade and R. G. Bartle. Pure and Applied
  Mathematics, Vol. 7. Interscience Publishers, Inc., New York; Interscience
  Publishers, Ltd., London, 1958.

\bibitem[DV75]{DoVa-75}
M.~D. Donsker and S.~R.~S. Varadhan.
\newblock Asymptotics for the {W}iener sausage.
\newblock {\em Comm. Pure Appl. Math.}, 28(4):525--565, 1975.

\bibitem[DV79]{DoVa-79}
M.~D. Donsker and S.~R.~S. Varadhan.
\newblock On the number of distinct sites visited by a random walk.
\newblock {\em Comm. Pure Appl. Math.}, 32(6):721--747, 1979.
\newblock URL: \url{http://dx.doi.org/10.1002/cpa.3160320602}, \href
  {http://dx.doi.org/10.1002/cpa.3160320602}
  {\path{doi:10.1002/cpa.3160320602}}.

\bibitem[DX19]{DX19}
Jian Ding and Changji Xu.
\newblock Poly-logarithmic localization for random walks among random
  obstacles.
\newblock {\em Ann. Probab.}, 47(4):2011--2048, 2019.
\newblock URL: \url{https://doi.org/10.1214/18-AOP1300}, \href
  {http://dx.doi.org/10.1214/18-AOP1300} {\path{doi:10.1214/18-AOP1300}}.

\bibitem[DX20]{DX20}
Jian Ding and Changji Xu.
\newblock Localization for random walks among random obstacles in a single
  {E}uclidean ball.
\newblock {\em Comm. Math. Phys.}, 375(2):949--1001, 2020.
\newblock URL: \url{https://doi.org/10.1007/s00220-020-03705-4}, \href
  {http://dx.doi.org/10.1007/s00220-020-03705-4}
  {\path{doi:10.1007/s00220-020-03705-4}}.

\bibitem[GL09]{GuL09}
Massimiliano Gubinelli and J{\'o}zsef L{\"o}rinczi.
\newblock Gibbs measures on brownian currents.
\newblock {\em Communications on Pure and Applied Mathematics: A Journal Issued
  by the Courant Institute of Mathematical Sciences}, 62(1):1--56, 2009.

\bibitem[Gub06]{Gu-06}
M.~Gubinelli.
\newblock Gibbs measures for self-interacting {W}iener paths.
\newblock {\em Markov Process. Related Fields}, 12(4):747--766, 2006.

\bibitem[K{\"o}n16]{Ko-16}
Wolfgang K{\"o}nig.
\newblock {\em The parabolic {A}nderson model}.
\newblock Pathways in Mathematics. Birkh\"auser/Springer, [Cham], 2016.
\newblock Random walk in random potential.
\newblock URL: \url{https://doi.org/10.1007/978-3-319-33596-4}.

\bibitem[K{\"u}n82]{MR661133}
H.~K{\"u}nsch.
\newblock Decay of correlations under {D}obrushin's uniqueness condition and
  its applications.
\newblock {\em Comm. Math. Phys.}, 84(2):207--222, 1982.
\newblock URL: \url{http://projecteuclid.org/euclid.cmp/1103921153}.

\bibitem[KV86]{KiVa-86}
C.~Kipnis and S.~R.~S. Varadhan.
\newblock Central limit theorem for additive functionals of reversible {M}arkov
  processes and applications to simple exclusions.
\newblock {\em Comm. Math. Phys.}, 104(1):1--19, 1986.
\newblock URL: \url{http://projecteuclid.org/euclid.cmp/1104114929}.

\bibitem[Lig05]{Li-05}
Thomas~M. Liggett.
\newblock {\em Interacting particle systems}.
\newblock Classics in Mathematics. Springer-Verlag, Berlin, 2005.
\newblock Reprint of the 1985 original.

\bibitem[LL10]{LaLi-10}
Gregory~F. Lawler and Vlada Limic.
\newblock {\em Random walk: a modern introduction}, volume 123 of {\em
  Cambridge Studies in Advanced Mathematics}.
\newblock Cambridge University Press, Cambridge, 2010.

\bibitem[LMSV21]{LMSV21}
Artur~O. Lopes, Ali Messaoudi, Manuel Stadlbauer, and Victor Vargas.
\newblock Invariant probabilities for discrete time linear dynamics via
  thermodynamic formalism.
\newblock {\em Nonlinearity}, 34(12):8359--8391, 2021.
\newblock URL: \url{https://doi.org/10.1088/1361-6544/ac3382}, \href
  {http://dx.doi.org/10.1088/1361-6544/ac3382}
  {\path{doi:10.1088/1361-6544/ac3382}}.

\bibitem[MBOV13]{MBOV13}
M.~Moreau, O.~Bénichou, G.~Oshanin, and R.~Voituriez.
\newblock {The shadow principle: An optimal survival strategy for a prey chased
  by random predators}.
\newblock {\em Physica A: Statistical Mechanics and its Applications},
  392(13):2837--2846, 2013.
\newblock URL:
  \url{https://ideas.repec.org/a/eee/phsmap/v392y2013i13p2837-2846.html}, \href
  {http://dx.doi.org/10.1016/j.physa.2013.02.0}
  {\path{doi:10.1016/j.physa.2013.02.0}}.

\bibitem[MOBC03]{MOBC03}
M~Moreau, G~Oshanin, O~B{\'e}nichou, and M~Coppey.
\newblock Pascal principle for diffusion-controlled trapping reactions.
\newblock {\em Physical Review E}, 67(4):045104, 2003.

\bibitem[MOBC04]{MOBC04}
M~Moreau, G~Oshanin, O~B{\'e}nichou, and M~Coppey.
\newblock Lattice theory of trapping reactions with mobile species.
\newblock {\em Physical Review E}, 69(4):046101, 2004.

\bibitem[MPU06]{MePeUt-06}
Florence Merlev\`ede, Magda Peligrad, and Sergey Utev.
\newblock Recent advances in invariance principles for stationary sequences.
\newblock {\em Probab. Surv.}, 3:1--36, 2006.
\newblock URL: \url{http://dx.doi.org/10.1214/154957806100000202}, \href
  {http://dx.doi.org/10.1214/154957806100000202}
  {\path{doi:10.1214/154957806100000202}}.

\bibitem[Muk22]{Mu-17}
Chiranjib Mukherjee.
\newblock Central limit theorem for {G}ibbs measures on path spaces including
  long range and singular interactions and homogenization of the stochastic
  heat equation.
\newblock {\em Ann. Appl. Probab.}, 32(3):2028--2062, 2022.
\newblock URL: \url{https://doi.org/10.1214/21-aap1727}, \href
  {http://dx.doi.org/10.1214/21-aap1727} {\path{doi:10.1214/21-aap1727}}.

\bibitem[MV20]{MR4054359}
Chiranjib Mukherjee and S.~R.~S. Varadhan.
\newblock Identification of the polaron measure {I}: {F}ixed coupling regime
  and the central limit theorem for large times.
\newblock {\em Comm. Pure Appl. Math.}, 73(2):350--383, 2020.
\newblock URL: \url{https://doi.org/10.1002/cpa.21858}, \href
  {http://dx.doi.org/10.1002/cpa.21858} {\path{doi:10.1002/cpa.21858}}.

\bibitem[\"{O}19]{Oz-19}
Mehmet \"{O}z.
\newblock Subdiffusivity of {B}rownian motion among a {P}oissonian field of
  moving traps.
\newblock {\em ALEA Lat. Am. J. Probab. Math. Stat.}, 16(1):33--47, 2019.

\bibitem[Pol00]{Po-00}
Mark Pollicott.
\newblock Rates of mixing for potentials of summable variation.
\newblock {\em Trans. Amer. Math. Soc.}, 352(2):843--853, 2000.
\newblock URL: \url{https://doi.org/10.1090/S0002-9947-99-02382-X}, \href
  {http://dx.doi.org/10.1090/S0002-9947-99-02382-X}
  {\path{doi:10.1090/S0002-9947-99-02382-X}}.

\bibitem[PSSS13]{PSSS13}
Yuval Peres, Alistair Sinclair, Perla Sousi, and Alexandre Stauffer.
\newblock Mobile geometric graphs: Detection, coverage and percolation.
\newblock {\em Probability Theory and Related Fields}, 156(1-2):273--305, 2013.

\bibitem[PU05]{PU2005}
Magda Peligrad and Sergey Utev.
\newblock A new maximal inequality and invariance principle for stationary
  sequences.
\newblock {\em Ann. Probab.}, 33(2):798--815, 2005.
\newblock URL: \url{https://doi.org/10.1214/009117904000001035}, \href
  {http://dx.doi.org/10.1214/009117904000001035}
  {\path{doi:10.1214/009117904000001035}}.

\bibitem[PU06]{PU2006}
Magda Peligrad and Sergey Utev.
\newblock Central limit theorem for stationary linear processes.
\newblock {\em Ann. Probab.}, 34(4):1608--1622, 2006.
\newblock URL: \url{https://doi.org/10.1214/009117906000000179}, \href
  {http://dx.doi.org/10.1214/009117906000000179}
  {\path{doi:10.1214/009117906000000179}}.

\bibitem[Sar15]{Sa-15}
Omri~M. Sarig.
\newblock Thermodynamic formalism for countable {M}arkov shifts.
\newblock In {\em Hyperbolic dynamics, fluctuations and large deviations},
  volume~89 of {\em Proc. Sympos. Pure Math.}, pages 81--117. Amer. Math. Soc.,
  Providence, RI, 2015.
\newblock URL: \url{https://doi.org/10.1090/pspum/089/01485}, \href
  {http://dx.doi.org/10.1090/pspum/089/01485}
  {\path{doi:10.1090/pspum/089/01485}}.

\bibitem[Spo87]{SPOHN1987278}
Herbert Spohn.
\newblock Effective mass of the polaron: A functional integral approach.
\newblock {\em Annals of Physics}, 175(2):278 -- 318, 1987.
\newblock URL:
  \url{http://www.sciencedirect.com/science/article/pii/0003491687902119},
  \href {http://dx.doi.org/https://doi.org/10.1016/0003-4916(87)90211-9}
  {\path{doi:https://doi.org/10.1016/0003-4916(87)90211-9}}.

\bibitem[SS11]{Stein-11}
Elias~M. Stein and Rami Shakarchi.
\newblock {\em Functional analysis}, volume~4 of {\em Princeton Lectures in
  Analysis}.
\newblock Princeton University Press, Princeton, NJ, 2011.
\newblock Introduction to further topics in analysis.

\bibitem[Szn98]{Sz-98}
Alain-Sol Sznitman.
\newblock {\em Brownian motion, obstacles and random media}.
\newblock Springer Monographs in Mathematics. Springer-Verlag, Berlin, 1998.

\bibitem[Szn23]{sz23}
Alain-Sol Sznitman.
\newblock On the spectral gap in the {K}ac--{L}uttinger model and
  {B}ose--{E}instein condensation.
\newblock {\em Stochastic Processes and their Applications}, 166:104197, 2023.

\bibitem[Wal75]{Wa-75}
Peter Walters.
\newblock Ruelle's operator theorem and {$g$}-measures.
\newblock {\em Trans. Amer. Math. Soc.}, 214:375--387, 1975.
\newblock URL: \url{https://doi.org/10.2307/1997113}, \href
  {http://dx.doi.org/10.2307/1997113} {\path{doi:10.2307/1997113}}.

\end{thebibliography}

\end{document}